\documentclass[11pt, a4paper]{article}
\usepackage{amssymb,latexsym,amsmath,amsfonts}
\usepackage{graphicx}
\usepackage{enumitem}
\usepackage{comment}
\usepackage{todonotes}
\usepackage{amsthm}
\usepackage{mathtools}
\usepackage{hyperref}

\usepackage{overpic}
\usepackage{rotating}

\usepackage{fullpage}

 \definecolor{refkey}{gray}{.5}   

 \definecolor{labelkey}{gray}{.5} 

\def\Xint#1{\mathchoice
{\XXint\displaystyle\textstyle{#1}}%
{\XXint\textstyle\scriptstyle{#1}}%
{\XXint\scriptstyle\scriptscriptstyle{#1}}%
{\XXint\scriptscriptstyle%
\scriptscriptstyle{#1}}%
\!\int}
\def\XXint#1#2#3{{\setbox0=\hbox{$#1{#2#3}{%
\int}$ }
\vcenter{\hbox{$#2#3$ }}\kern-.6\wd0}}

\def\dashint{\Xint-}

\mathtoolsset{showonlyrefs=true}

\numberwithin{equation}{section}

\def\ds{\displaystyle}

\DeclareMathOperator*{\supp}{supp}

\DeclareMathOperator*{\Tr}{Tr}
\DeclareMathOperator*{\re}{Re}
\renewcommand{\Re}{\re}
\DeclareMathOperator*{\im}{Im}
\renewcommand{\Im}{\im}
\newcommand{\X}{\mathbf{X}}
\newcommand{\K}{\mathcal{K}}

\DeclareMathOperator{\PE}{PE}
\DeclareMathOperator{\crit}{cr}
\DeclareMathOperator{\loc}{loc}
\DeclareMathOperator{\rest}{rest}
\DeclareMathOperator{\arccot}{arccot}
\DeclareMathOperator{\sgn}{sgn}
\newcommand{\Andreief}{Andr\'{e}ief}

\newcommand{\Wcirc}{\overset{\circ}{W}}
\newcommand{\Id}{\mathbf{1}}

\newtheorem{theorem}{Theorem}[section]
\newtheorem{lemma}[theorem]{Lemma}

\newtheorem{proposition}[theorem]{Proposition}

\newtheorem{assumption}[theorem]{Assumption}
\newtheorem{Definition}[theorem]{Definition}

\newtheorem{Remark}[theorem]{Remark}
\newenvironment{remark}{\begin{Remark}\rm}{\end{Remark}}
\newtheorem{Example}[theorem]{Example}

\title{Propagation of singular behavior for Gaussian perturbations of random matrices}

\author{Tom Claeys \thanks{Institut de Recherche en Math\'ematique et Physique, Universit\'e catholique de Louvain, Chemin du Cyclotron 2, Louvain-La-Neuve, Belgium
\href{mailto:tom.claeys@uclouvain.be}{\nolinkurl{tom.claeys@uclouvain.be}}}
\and
Arno B.J. Kuijlaars \thanks{Department of Mathematics, Katholieke Universiteit Leuven, 
Celestijnenlaan 200 B, Leuven, Belgium \href{mailto:arno.kuijlaars@kuleuven.be}{\nolinkurl{arno.kuijlaars@kuleuven.be}}}
\and
Karl Liechty \thanks{Department of Mathematical Sciences, DePaul University, Chicago, IL, 60614 USA
\href{mailto:kliechty@depaul.edu}{\nolinkurl{kliechty@depaul.edu}}}
 \and
Dong Wang\thanks{Department of Mathematics, National University of Singapore, Singapore, 119076, \href{mailto:matwd@nus.edu.sg}{\nolinkurl{matwd@nus.edu.sg}}}
}

\date{\today}
\begin{document}
\maketitle

\begin{abstract}
We study the asymptotic behavior of the eigenvalues of Gaussian perturbations of large Hermitian random matrices for which the limiting eigenvalue density vanishes at a singular interior point or vanishes faster than a square root at a singular edge point. First, we show that the singular behavior propagates macroscopically for sufficiently small Gaussian perturbations, and we describe the macroscopic eigenvalue behavior for Gaussian perturbations of critical size.
Secondly, for sufficiently small Gaussian perturbations of unitary invariant random matrices, we prove that the microscopic eigenvalue correlations near the singular point are described by the same limiting kernel as in the unperturbed case. We also interpret our results in terms of nonintersecting Brownian paths with random starting positions, and we establish multi-time generalizations of the microscopic results.
\end{abstract}

\section{Introduction and statement of results}

\subsection{Introduction}

This paper deals with sums of $n \times n$ random matrices 
\begin{equation} \label{def:X}
  X = X_{\tau} := M + \sqrt{\tau} H
\end{equation} 
where $\tau \geq 0$, $M = M_n$ is a Hermitian random matrix, and $H = H_n$ is a properly scaled Gaussian unitary ensemble (GUE) matrix 
independent of $M$, defined by the probability distribution on the space of Hermitian matrices,
\begin{equation}\label{def:GUE}
  \frac{1}{Z_n^{\rm GUE}} \, e^{-\frac{n}{2} \Tr H^2} \, dH,\qquad 
  dH=\prod_{j=1}^ndH_{jj}\prod_{1\leq i<j\leq n}d\Re H_{ij} d\Im H_{ij},
\end{equation}
with $Z_n^{\rm GUE}$ a normalizing constant. $X_{\tau} = M$ when $\tau = 0$, and otherwise is an additive Gaussian perturbation of $M$.
We are interested in the global and local statistics of eigenvalues of the random matrices $X$ in the 
limit where the size $n$ of the matrices tends to infinity. 

The eigenvalue distribution of $\sqrt{\tau}H$ converges almost surely as $n \to \infty$, in the weak sense, to the semi-circle law on the interval $[-2\sqrt{\tau}, 2\sqrt{\tau}]$, which we denote $\lambda_{\tau}$, defined as
\begin{equation} \label{def:lambdatau}
  d\lambda_{\tau}(s) = \frac{1}{2\pi \tau} \sqrt{4\tau - s^2} \, ds,
		\qquad s \in [-2\sqrt{\tau}, 2\sqrt{\tau}].
	\end{equation}
We consider Hermitian random matrices $M=M_n$ in \eqref{def:X} whose eigenvalue distributions also converge almost surely, in the weak sense, to some limiting distribution $\mu_0$. We will be interested in cases where the limiting
distribution $\mu_0$ has a certain singular behavior at an interior point or at an edge point of its support. The singular interior and edge points we consider correspond to the ones that  can occur in unitary invariant random matrix ensembles, and this will be our main case of interest.   

To describe the global limiting distribution of the eigenvalues of $X$ as $n\to\infty$, we can rely on general well-known results from free probability theory, see e.g.\ \cite{Anderson-Guionnet-Zeitouni10, Hiai-Petz00, Speicher93, Dykema-Nica-Voiculescu92}, which imply that the matrix $X$ almost 
surely has a weak limiting eigenvalue distribution $\mu_\tau$. The distribution $\mu_\tau$ is the free additive convolution
\begin{equation} \label{eq:muX} 
	\mu_\tau = \mu_0 \boxplus \lambda_{\tau}
	\end{equation}
 of $\mu_0$ with the semi-circle law $\lambda_{\tau}$.
We will show that the distribution $\mu_\tau$
also has a singular point of the same type as $\mu_0$, as long as $\tau$ is smaller than a certain critical value $\tau_{\crit}>0$.

In our study of the local eigenvalue statistics, we assume additionally that $M$ is taken from a 
unitary invariant random matrix ensemble. Then the local eigenvalue statistics of $M$ are described by a 
limiting eigenvalue correlation kernel which depends on the nature of the singular point, and which is 
built out of Painlev\'e functions. We will show that the same limiting kernel describes the local eigenvalue statistics of $X=X_\tau$ for $0 \leq \tau<\tau_{\crit}$. That is, the local statistics are unchanged under small enough additive Gaussian perturbations.

Gaussian perturbations of Hermitian matrices have an alternative interpretation in terms of a model of Brownian paths 
which are conditioned not to intersect.
Consider $n$ particles in Brownian motion with diffusion parameter $n^{-1/2}$, starting at the eigenvalues 
$a_1< \cdots < a_n$ of a Hermitian matrix $M$ at time $0$, and ending at fixed points $b_1< \cdots < b_n$ at time $1$.
We condition the paths $x_1(t)<\cdots <x_n(t)$ such that they do not intersect for $t\in (0,1)$, and we 
let the ending points $b_1,\ldots, b_n$ tend to $0$. Then the joint probability density function of  $x_1(t)< \cdots < x_n(t)$ at time $t \in [0, 1)$ is the same as that of the eigenvalues $\lambda_1(t)<\cdots<\lambda_n(t)$ of the matrix
\begin{equation}\label{relation random matrix Brownian paths}
	(1 - t)M + \sqrt{t(1 - t)} H=(1-t)X_{\tau},\qquad \tau=\frac{t}{1-t}.
\end{equation}
This is true if $M$ is a deterministic matrix, but we can also take $M$ to be random, and define the model of 
nonintersecting Brownian paths such that the initial distribution of $x_1(0)<\cdots <x_n(0)$ is the same as 
that of the eigenvalues of the random matrix $M$. 
Our results in the random matrix setting will thus directly imply results for nonintersecting Brownian paths with random starting positions. In addition we will prove multi-time generalizations of those results.

\subsection{Propagation of singular behavior}

We start by describing in a precise way the conditions we impose on the limiting distribution $\mu_0$ of 
the eigenvalues of the random Hermitian matrix $M = M_n$.

\begin{assumption}\label{assumption:a}
The limiting eigenvalue distribution
  $\mu_0$ has a compact support with a continuous density $\psi_0$.
\end{assumption}
\begin{assumption} \label{assumption:b}
  There exist $x^* \in \supp(\mu_0)$, $c_0 > 0$, and an integer $k \geq 1$, 
  such that for some function $h(s)$ which is analytic at $s=x^*$ 
  with  $h(x^*)=1$, we have either
  \begin{enumerate}[label=(\alph*)]
  \item \label{enu:assumption:b_1}
    $x^*$ is an interior point of the support of $\mu_0$ and
    \begin{equation} \label{eq:singularint} 
      \psi_0(s)  = c_0^{2k+1}(s-x^*)^{2k}h(s),\qquad s\in\supp(\mu_0),
    \end{equation} 
    or
  \item \label{enu:assumption:b_2}
    there exists $\epsilon > 0$ such that  
    \begin{equation} \label{eq:singularedge} 
      \psi_0(s) =
      \begin{cases}
        c_0^{2k+3/2} (x^*-s)^{2k+1/2}h(s), & \text{ for }  s\in 
        [x^*-\epsilon, x^*], \\
        0, & \text{ for } s \in [x^*, x^*+\epsilon].
      \end{cases} 
    \end{equation}
  \end{enumerate}
\end{assumption}
If we have \eqref{eq:singularint} we say that $x^*$ is a {\em singular interior point with exponent $2k$}, and if we have \eqref{eq:singularedge} we say that $x^*$ is a {\em singular right edge point with exponent $2k+1/2$}.
There is a similar notion of {\em singular left edge point} for which completely 
analogous results hold. For ease of presentation we only consider   
singular right edge points. We note that in case \ref{enu:assumption:b_2}, $x^*$ may not be the rightmost edge of the support, since the support may be on several intervals, and $x^*$ can be the right edge of any interval. See Remark \ref{rmk:macrocritical}\ref{enu:rmk:macrocritical_4}.

These are the singular behaviors that can and do occur for eigenvalues of 
Hermitian matrices $M$  from a unitary ensemble  
\begin{equation} \label{def:UE} 
	\frac{1}{Z_n} \, e^{-n \Tr V(M)} \, dM, 
	\end{equation}
where $V : \mathbb R \to \mathbb R$ is real analytic  with sufficient growth at $\pm\infty$,
and $Z_n$ is a normalizing constant, see \cite{Deift99,Deift-Kriecherbauer-McLaughlin98}. 
For now, $M = M_n$ can be any random Hermitian matrix for which the eigenvalue distribution converges as $n \to \infty$ almost surely, weakly, to a measure $\mu_0$ satisfying Assumptions \ref{assumption:a} and \ref{assumption:b}. 


The first aim of the paper is to verify that the singular behaviors \eqref{eq:singularint}
and \eqref{eq:singularedge} persist for random matrices $X_{\tau} = M + \sqrt{\tau} H$ defined by \eqref{def:X} if
$\tau > 0$ is smaller than the critical value 
\begin{equation} \label{def:taucr} 
\tau_{\crit} = \left[ \int \frac{d\mu_0(s)}{(x^*-s)^2} \right]^{-1}.
\end{equation}
Due to the higher order vanishing  \eqref{eq:singularint}--\eqref{eq:singularedge} 
of the density of $\mu_0$ at $x^*$, the integral in \eqref{def:taucr} 
indeed converges, and therefore $\tau_{\crit} > 0$.
To prove the propagation of singular behavior, we study the behavior near $x^*$ of the density of the free convolution $\mu_\tau$ given in \eqref{eq:muX}.
We will not use the general definition of free convolution here, but rather
its special analytic form when applied to free convolution with a semi-circle law 
that is due to Biane \cite{Biane97}. Using this  approach we are able to prove that the
singular behaviors \eqref{eq:singularint} and \eqref{eq:singularedge} persist
in the following precise sense.

\begin{theorem} \label{thm:mainmacro}
  Suppose $\mu_0$ is a probability measure on $\mathbb R$ satisfying Assumptions \ref{assumption:a} and \ref{assumption:b}. Let $0 < \tau < \tau_{\crit}$
  where $\tau_{\crit}$ is given by \eqref{def:taucr} and denote
  \begin{equation}\label{def:ctau}
    c_{\tau} := \frac{\tau_{\crit}}{\tau_{\crit}-\tau}c_0.
  \end{equation}
  Then $\mu_{\tau} = \mu_0 \boxplus \lambda_{\tau}$ has a density $\psi_{\tau}$ with a singular point
  \begin{equation} \label{def:xtau}
    x_{\tau}^* = x^* + \tau \int \frac{d\mu_0(s)}{x^*-s}. 
  \end{equation}
  \begin{enumerate}[label=(\alph*)]
  \item \label{enu:thm:mainmacro_1}
    Under Assumption \ref{assumption:b}\ref{enu:assumption:b_1}, $x_{\tau}^*$
    is an interior point of the support of $\mu_{\tau}$ and 
    the density $\psi_{\tau}(s)$ satisfies
    \begin{equation} \label{eq:propsingularint} 
      \psi_\tau(s) = c^{2k+1}_{\tau} (s-x_{\tau}^*)^{2k} \left( 1 + O(s - x^*_{\tau}) \right), \quad \text{as } s \to x_{\tau}^*.
    \end{equation}
  \item  \label{enu:thm:mainmacro_2}
    Under Assumption \ref{assumption:b}\ref{enu:assumption:b_2}, $x_{\tau}^*$ is a right edge point of the support of $\mu_{\tau}$ and the density $\psi_{\tau}(s)$ satisfies
    for some $\epsilon > 0$,
    \begin{equation} \label{eq:propsingularedge} 
      \psi_{\tau}(s) =
      \begin{cases}
        c^{2k+3/2}_{\tau} (x_{\tau}^*-s)^{2k+1/2} (1 + O((x^*_{\tau}-s)^{1/2})), & \text{as } s \to (x_{\tau}^*)_-, \\
        0, & \text{for } s \in [x_{\tau}^*, x_{\tau}^* + \epsilon].
      \end{cases}
    \end{equation}
  \end{enumerate}
\end{theorem}
We call \eqref{eq:propsingularint} and \eqref{eq:propsingularedge} the propagation
of singular behavior under addition of a GUE matrix.
In the special case when the density $\psi_0$ is symmetric around $x^*$, the
propagation of a singular interior point was already 
noted by Biane  in \cite[Proposition 6]{Biane97}.
The exact order of vanishing at $x^*_{\tau}$ seems not to be in the literature. Part \ref{enu:thm:mainmacro_1} of Theorem \ref{thm:mainmacro} is proved in Section \ref{subsubsec:interior_subcrit} and part \ref{enu:thm:mainmacro_2} is proved in Section \ref{subsubsec:edge_sub}.
%

\subsection{Critical perturbations}

For $\tau=\tau_{\crit}$, the singular behavior does not persist. Several situations are possible, depending on the nature of the singular point $x^*$ and the limiting distribution $\mu_0$.

\begin{theorem}\label{thm:macrocritical}
  Suppose $\mu_0$ is a probability measure on $\mathbb R$ which satisfies Assumptions \ref{assumption:a} and \ref{assumption:b}. Let $\tau_{\crit}$ be given by \eqref{def:taucr}, and let $h$ be as in \eqref{eq:singularint} or \eqref{eq:singularedge}, depending on the case.   
  Then $\mu_{\tau_{\crit}} = \mu_0 \boxplus \lambda_{\tau_{\crit}}$ has a density $\psi_{\tau_{\crit}}$ with a singular point $x_{\tau_{\crit}}^*$ given by \eqref{def:xtau} with $\tau=\tau_{\crit}$. 
  \begin{enumerate}[label=(\alph*)]
  \item \label{enu:thm:macrocritical_I}
    Under Assumption \ref{assumption:b}\ref{enu:assumption:b_1}, if $x^*$ is a singular interior point with exponent $2k = 2$ (that is, $k = 1$), we have
\begin{equation} \label{eq:psilocalcritthmI}
  \psi_{\tau_{\crit}}(s) =
  \begin{cases} \ds
    \frac{\cos \frac{\theta}{2}}{\pi \tau_{\crit}^{3/2} c_0^{3/2}\sqrt{r}} 
    \lvert s-x^*_{\tau_{\crit}} \rvert^{1/2} (1 + O( \lvert s-x^*_{\tau_{\crit}} \rvert^{1/2})), & \text{as } s \to (x_{\tau_{\crit}}^*)_-, \\[15pt]
   \ds \frac{\sin \frac{\theta}{2}}{\pi \tau_{\crit}^{3/2} c_0^{3/2}\sqrt{r}}  
    \lvert s-x^*_{\tau_{\crit}} \rvert^{1/2} (1 + O( \lvert s-x^*_{\tau_{\crit}} \rvert^{1/2})), &\text{as } s \to (x_{\tau_{\crit}}^*)_+, 
  \end{cases}
\end{equation}
where $r>0$ and $0<\theta<\pi$ are given by ($\dashint$ stands for the Cauchy principal integral)
\begin{equation}\label{def:rtheta}
  r^2 =  \left( \dashint_{\supp(\mu_0)} \frac{h(s)}{x^*-s} ds\right)^2+\pi^2,\quad \theta = \arccot \left( \frac{1}{\pi} \dashint_{\supp(\mu_0)} \frac{h(s)}{x^*-s} ds \right) \in (0, \pi).
\end{equation}
\item \label{enu:thm:macrocritical_II}
  Under Assumption \ref{assumption:b}\ref{enu:assumption:b_1}, if $x^*$ is a singular interior point with exponent $2k$ and $k\geq 2$, we let 
  \begin{equation} \label{def:g}
    g_2 = \int \frac{d\mu_0(s)}{(s - x^*)^3} = c_0^{2k+1}\int_{\supp(\mu_0)}(s-x^*)^{2k-3}h(s)ds.
  \end{equation}
  If $g_2>0$, then 
\begin{equation} \label{eq:psilocalcritthmII}
  \psi_{\tau_{\crit}}(s) =
  \begin{cases}  
  \ds  \frac{c_0^{2k+1}}{2(\tau_{\crit}g_2)^{k+1/2}}\lvert s-x^*_{\tau_{\crit}} \rvert^{k-1/2} (1 + O( \lvert s-x^*_{\tau_{\crit}} \rvert^{1/2})), & 
\text{as } s \to (x_{\tau_{\crit}}^*)_-, \\[15pt]  
  \ds
    \frac{1}{\pi \tau_{\crit}^{3/2} g_2^{1/2}}\lvert s-x^*_{\tau_{\crit}} \rvert^{1/2} (1 + O( \lvert s-x^*_{\tau_{\crit}} \rvert^{1/2})), &\text{as } s \to (x_{\tau_{\crit}}^*)_+, 
  \end{cases}
\end{equation}
and if $g_2 < 0$, we have
\begin{equation} \label{eq:psilocalcritthmIIbis}
\psi_{\tau_{\crit}}(s) =
\begin{cases} \ds
\frac{1}{\pi \tau_{\crit}^{3/2} |g_2|^{1/2}}\lvert s-x^*_{\tau_{\crit}} \rvert^{1/2} (1 + O( \lvert s-x^*_{\tau_{\crit}} \rvert^{1/2})), & \text{as } s \to (x_{\tau_{\crit}}^*)_-, \\[15pt]
\ds  \frac{c_0^{2k+1}}{2(\tau_{\crit}|g_2|)^{k+1/2}}\lvert s-x^*_{\tau_{\crit}} \rvert^{k-1/2} (1 + O( \lvert s-x^*_{\tau_{\crit}} \rvert^{1/2})), & 
\text{as } s \to (x_{\tau_{\crit}}^*)_+.
\end{cases}
\end{equation}

\item \label{enu:thm:macrocritical_III}
  Under Assumption \ref{assumption:b}\ref{enu:assumption:b_1}, if $x^*$ is a singular interior point with exponent $2k$, $k\geq 2$, and $g_2$ given by \eqref{def:g} is equal to zero, we let
  \begin{equation} \label{def:tildeg}
    g_3 = \int \frac{d\mu_0(s)}{(s - x^*)^4} = c_0^{2k+1} \int_{\supp(\mu_0)} (s-x^*)^{2k-4}h(s)ds.
  \end{equation}
  Then $g_3 > 0$, and 
\begin{equation} \label{eq:psilocalcritthmIII}
  \psi_{\tau_{\crit}}(s) =
    \frac{\sqrt{3}}{2\pi \tau_{\crit}^{4/3} g_3^{1/3}}\lvert s-x^*_{\tau_{\crit}} \rvert^{1/3}(1+O(\lvert s-x^*_{\tau_{\crit}}\rvert^{1/3} )), \qquad 
   \text{as } s \to x_{\tau_{\crit}}^*.
\end{equation}
\item \label{enu:thm:macrocritical_IV}
  Under Assumption \ref{assumption:b}\ref{enu:assumption:b_2}, if $x^*$ is a singular right edge point with exponent $2k + 1/2$, we let
  \begin{equation} \label{eq:g_for_edge}
    g_2 = \int \frac{d\mu_0(s)}{(s - x^*)^3} = c^{2k + 3/2}_0 \int_{\supp(\mu_0)} \sgn(s - x^*) \lvert s - x^* \rvert^{2k - 5/2} h(s) ds, 
  \end{equation}
  where $\sgn(x) = 1, 0, -1$ depends on if $x$ is positive, zero, or negative. If $g_2 < 0$, then there exists $\epsilon > 0$ such that
\begin{equation} \label{eq:psilocalcritthmIV}
  \psi_{\tau_{\crit}}(s) =
  \begin{cases} \ds
    \frac{1}{\pi \tau_{\crit}^{3/2} \sqrt{|g_2|}}|s-x^*_{\tau_{\crit}}|^{1/2}  \left(1 + O(|s-x^*_{\tau_{\crit}}|^{\alpha})\right), & \text{as } s \to (x_{\tau_{\crit}}^*)_-, \\[15pt]
    \ds 
    0, & s \in [x_{\tau_{\crit}}^*, x_{\tau_{\crit}}^*+\epsilon],
  \end{cases}
\end{equation}
where $\alpha = 1/4$ if $k=1$ and $\alpha = 1/2$ otherwise. 
\item \label{enu:thm:macrocritical_V}
  Under Assumption \ref{assumption:b}\ref{enu:assumption:b_2}, if $x^*$ is a singular right edge point with exponent $2k + 1/2$, and $g_2$ defined in \eqref{eq:g_for_edge} is positive, then 
  \begin{equation}
    \psi_{\tau_{\crit}}(s) = 
  \begin{cases} 
  \ds \frac{c^{2k + 3/2}_0}{2(\tau_{\crit} g_2)^{k + 3/4}} |s-x^*_{\tau_{\crit}}|^{k - 1/4} 
  (1 + O( \lvert s - x_{\tau_{\crit}}^* \rvert^{1/4})), & \text{as } s \to (x_{\tau_{\crit}}^*)_-, \\[15pt]  
  \ds
    \frac{1}{\pi \tau_{\crit}^{3/2} \sqrt{g_2}}|s-x^*_{\tau_{\crit}}|^{1/2} (1 + O(|s-x_{\tau^*_{\crit}}|^{\alpha}), & \text{as } s \to (x_{\tau_{\crit}}^*)_+,
  \end{cases}
  \end{equation}
  where once again $\alpha = 1/4$ if $k=1$ and $\alpha = 1/2$ otherwise. 
\end{enumerate}	
\end{theorem}

The proof of parts \ref{enu:thm:macrocritical_I}, \ref{enu:thm:macrocritical_II}, and \ref{enu:thm:macrocritical_III} of Theorem \ref{thm:macrocritical} is given
in Section \ref{subsubsec:interior_crit} and the proof of parts  \ref{enu:thm:macrocritical_IV} and  \ref{enu:thm:macrocritical_V} is
in Section \ref{subsubsec:edge_crit}.

\begin{remark} \label{rmk:macrocritical}
  \begin{enumerate}[label=(\roman*)]
  \item 
    In part \ref{enu:thm:macrocritical_I} of the above theorem, if $\dashint_{\supp(\mu_0)} h(s)/(x^* - s) ds = 0$, then $\theta = \pi/2$ and the constants in front of the
    two square roots in \eqref{eq:psilocalcritthmI} are the same. This happens
    if $h(s)$ is symmetric about $s=x^*$. 
  \item 
    If $k \geq 2$ and $h(s)$ is symmetric about $s=x^*$, then $g_2 = 0$ and part  \ref{enu:thm:macrocritical_II} does not occur. 
  \item 
    If $x^*$ is the rightmost point of the support of $\mu_0$, then $g_2 < 0$
    where $g_2$ is defined by \eqref{eq:g_for_edge}, and we are in part  \ref{enu:thm:macrocritical_IV}. Furthermore, by Lemma \ref{lem:right_most_x^*} below, 
    we have $\psi_{\tau_{\crit}}(s) = 0$ for every $s > x^*_{\tau_{\crit}}$ 
    in \eqref{eq:psilocalcritthmIV}.
  \item \label{enu:rmk:macrocritical_4}
    Part \ref{enu:thm:macrocritical_V} can only happen if
    $x^*$ is a right endpoint of an interval in the support, 
    but there are other intervals to the right of $x^*$ as well, since in such 
    a situation it could happen that $g_2>0$. Then it could
    also happen that $g_2 = 0$, but this situation is more delicate and we do not investigate it here. 
  \end{enumerate}
\end{remark}

\subsection{Unitary ensembles}

For the next main result, concerning the local behavior of the eigenvalues of $X=M+\sqrt{\tau} H$, we restrict to matrices $M$ from
the unitary ensemble \eqref{def:UE} with $V$ such that $\frac{V(x)}{\log(1+|x|)} \to +\infty$ as $x \to \pm \infty$, and with $V$ real analytic on $\mathbb R$. It is a basic fact from random matrix theory, see e.g.~\cite{Deift99, Forrester10}, 
that the eigenvalues of $M$ have the joint probability density function
\begin{equation} \label{eq:UEpdf} 
  \frac{1}{\tilde{Z}_n} \Delta_n(x)^2 \prod_{j=1}^n e^{-n V(x_j)}, \qquad \Delta_n(x) = \prod_{1 \leq i < j \leq n} (x_j-x_i).
\end{equation} 
This is a determinantal point process with correlation kernel 
\begin{equation} \label{def:KnM} 
	K_n^M(x,y) = e^{-\frac{n}{2} (V(x) +  V(y))} \sum_{j=0}^{n-1} p_{j,n}(x) p_{j,n}(y),
	\end{equation}
where $p_{j,n}$, $j=0,1,2, \ldots$, are the orthonormal polynomials for the weight $e^{-nV(x)}$ on the real line, i.e., $p_{j,n}$ has degree $j$, positive leading coefficient, and 
\begin{equation} \label{eq:defn_p_jn}
  \int_{-\infty}^{\infty} p_{j,n}(x) p_{k,n}(x) e^{-n V(x)} dx = \delta_{j,k}.
\end{equation}

The limiting eigenvalue distribution $\mu_0$ exists almost surely \cite{Anderson-Guionnet-Zeitouni10} and is known as the equilibrium measure. It is a compactly supported Borel probability measure that minimizes the logarithmic energy in the external field $V$,
\begin{equation} \label{eq:logenergy} 
	\iint \log \frac{1}{|s-t|} d\mu(s) d\mu(t) + \int V(t) d\mu(t) ,
	\end{equation}
among all Borel probability measures on $\mathbb R$.

Since $V$ is real analytic, the equilibrium measure is supported on a finite
union of intervals, and the density $\psi_0$ has the form
\[ \psi_0(s) =  \frac{1}{\pi} \sqrt{q_V^-(s)} ds, \]
where $q_V : \mathbb R \to \mathbb R$ is a real analytic function, and $q_V^-$ denotes the negative part of $q_V$, see \cite{Deift-Kriecherbauer-McLaughlin98}. If $V$ is a polynomial
then $q_V$ is a polynomial as well. From this form of the density it follows that the density
is real analytic on the interior of each interval in the support. Generically, $q_V$ does not vanish in the interior of each interval, and it has a simple zero
at each of the endpoints of the intervals which leads to square root vanishing of the 
density, see \cite{Kuijlaars-McLaughlin00}.
However, for certain special potentials $V$ there can be a zero at an interior point, or a  higher order zero at one of the endpoints. Near such singular interior points or singular edge points \cite{Deift-Kriecherbauer-McLaughlin-Venakides-Zhou99, Kuijlaars-McLaughlin00}, the equilibrium density takes the form \eqref{eq:singularint} or \eqref{eq:singularedge} for some positive integer $k$.

For what follows it is convenient to express the quantities
$\tau_{\crit}$ and $x_{\tau}^*$, see \eqref{def:taucr} and \eqref{def:xtau}, in terms of $V$. This can be done because of the following lemma, which we will prove in Section \ref{sec:prooflemma}.

\begin{lemma} \label{lemma12}
  Suppose $x^*$ is a singular interior point with exponent $2k$ or a singular edge point with exponent $2k+1/2$.
Then for every $l=1,2, \ldots, 2k$, we have
\begin{equation} \label{eq:Vderivative} 
  V^{(l)}(x^*) = -2 (l-1)! \int \frac{d\mu_0(s)}{(s-x^*)^l}.
\end{equation}
\end{lemma}
In particular we find from Lemma \ref{lemma12} with $l = 1$ and $l=2$, and \eqref{def:xtau}, \eqref{def:taucr} that
\begin{align}
  x_{\tau}^* = {}& x^* + \frac{\tau}{2} V'(x^*), \label{eq:xtauV}  \\
  \tau_{\crit} = {}& - \frac{2}{V''(x^*)}. \label{eq:taucrV} 	
\end{align}

At non-singular points, the correlation kernel \eqref{def:KnM} has the usual
scaling limits from random matrix theory, namely the sine kernel at non-singular interior points and the Airy kernel at non-singular edge points, see \cite{Deift-Kriecherbauer-McLaughlin-Venakides-Zhou99}. 
The correlation kernel has non-trivial scaling limits at singular points.
We
let
\begin{equation} \label{def:gamma}
  \gamma = (\kappa + 1)^{-1},
\end{equation}
with $\kappa$ the exponent of the singular point,
\begin{equation}\label{def:kappa}
\kappa=\begin{cases}2k&\mbox{in case of a singular interior point},\\
2k+1/2&\mbox{in case of a singular edge point.}
\end{cases}
\end{equation}
Then the scaling limit takes the form
\begin{equation} \label{def:singlimit} 
  \lim_{n \to \infty} \frac{1}{c_0 n^{\gamma}}
  K_n^M\left(x^* + \frac{u}{c_0 n^{\gamma}}, x^* + \frac{v}{c_0 n^{\gamma}} \right)
  = \K_{\kappa}(u,v)
\end{equation}
with a universal limiting kernel that only depends on the exponent $\kappa$, and with $c_0$ as in \eqref{eq:singularint} and \eqref{eq:singularedge}.	
For a singular interior point \eqref{eq:singularint} with exponent $\kappa = 2$, the kernel is related to the Hastings--McLeod solution of the Painlev\'e II equation, see \cite{Bleher-Its03, Claeys-Kuijlaars06}, and in the case of a singular edge point \eqref{eq:singularedge} with $\kappa =  5/2$, the limiting kernel is related to a special solution of the second member of the Painlev\'e I hierarchy \cite{Claeys-Vanlessen07}. 

For $k\geq 2$ it seems that the nature of the limit \eqref{def:singlimit} has not been analyzed rigorously in the mathematical literature, but there is strong evidence in the physical literature that in the case of an interior critical point with $\kappa = 2k$ with $k \geq 2$ the limiting kernel is related to a member of the Painlev\'e II hierarchy \cite{Bleher-Eynard03}, and in the case of an edge critical point with $\kappa = 2k+1/2$ with $k \geq 2$ it is related to a higher order member of the Painlev\'e I hierarchy~\cite{Bowick-Brezin91, Brezin-Marinari-Parisi90}, see also \cite{Claeys-Its-Krasovsky10}.
For general $k$, existence of a kernel $\K_{\kappa}$ such that \eqref{def:singlimit} holds can be proved in a rather straightforward way using results from \cite{Deift-Kriecherbauer-McLaughlin-Venakides-Zhou99}. We will do this in Appendix \ref{sec:RHP_summary}, and we show moreover that \eqref{def:singlimit} holds not only for real values of $u$ and $v$, but uniformly for $u$ and $v$ in compact subsets of the complex plane. This fact will be important later on for the proof of our main results. Our proof of \eqref{def:singlimit} does not reveal any integrable property of $\K_{\kappa}$.
\begin{proposition} \label{prop:prop17}
  For $u$ and $v$ in compact subsets of $\mathbb C$, \eqref{def:singlimit} holds uniformly for some limiting kernel $\K_{\kappa}(u,v)$, analytic in $u$ and $v$, that only depends on the parameter $\kappa$, and with $c_0$ as in \eqref{eq:singularint} and \eqref{eq:singularedge}.	
\end{proposition}

We will prove that the limiting kernel $\K_{\kappa}$ also appears as the scaling limit for the eigenvalues of $X = M + \sqrt{\tau} H$ whenever $\tau < \tau_{\crit}$, which we will explain next. 

\subsection{Determinantal structure and double integral formula}

First of all, it is known that the eigenvalues of $X_{\tau} = M + \sqrt{\tau} H$ are determinantal, see \cite{Brezin-Hikami96,Brezin-Hikami97,Johansson01a}.  It is a much more recent result \cite{Claeys-Kuijlaars-Wang15} that there is a double integral formula for the new eigenvalue correlation kernel in terms of the one for $M$,
namely
\begin{equation} \label{eq:KnX0}
  K_n^X(x,y; \tau) = \frac{n}{2\pi i \tau} \int_{- i \infty}^{+ i \infty} dz \int_{-\infty}^{\infty} dw\, K_n^M(z, w) e^{\frac{n}{2}(V(z) - V(w))} e^{\frac{n}{2\tau} ((z-x)^2 - (w-y)^2)}.
	\end{equation}
	
The result in \cite[Theorem 2.3]{Claeys-Kuijlaars-Wang15} was formulated in the more general context of random matrices $M$ whose eigenvalues are a polynomial ensemble, which is a joint probability density on $\mathbb R^n$  of the form
\[ \frac{1}{Z_n} \Delta_n(x) \det\left[ f_{k-1}(x_j) \right]_{j,k=1}^n, \]
for certain functions $f_0, \ldots, f_{n-1}$. This is a determinantal point
process with a kernel that can be written as
\[ K_n^{\PE}(x,y) = \sum_{j=0}^{n-1} p_j(x) q_j(y), \]
where each $p_j$ is a polynomial of degree $j$ and each $q_j$ belongs to the
linear span of $f_0, \ldots, f_{n-1}$. 

The eigenvalue distribution \eqref{eq:UEpdf} of a unitary ensemble can be put
in this form, and the polynomial kernel is related to the usual kernel \eqref{def:KnM}
by
\begin{equation} \label{def:KnPE} 
    K_n^{\PE}(x,y)  = K_n^M(x,y) e^{\frac{n}{2}(V(x) - V(y))}  = e^{-n V(y)} \sum_{j=0}^{n-1} p_{j,n}(x) p_{j,n}(y),
\end{equation}
and this combination appears in the integrand of \eqref{eq:KnX0}.
Note that $K_n^{\PE}(z,w)$ is polynomial in $z$, and so there is no
problem to evaluate it for $z$ on the imaginary axis, as we do in \eqref{eq:KnX0}.

By analyticity we can move the contour for the $z$-integral in \eqref{eq:KnX0} to any other vertical line in the complex plane. It will be convenient for us to take
the vertical line passing through the singular point $x^*$, and so we will work with
\begin{equation} \label{eq:KnX}
  K_n^X(x, y; \tau) = \frac{n}{2\pi i \tau} \int_{x^*- i \infty}^{x^* + i \infty} dz \int_{-\infty}^{\infty} dw\, K_n^M(z,w) e^{\frac{n}{2}(V(z) - V(w))} e^{\frac{n}{2\tau} ((z-x)^2 - (w-y)^2)}.
\end{equation}

\subsection{Propagation of local  scaling limit}

We are now ready to formulate our main result.
\begin{theorem} \label{thm:mainmicro}
Suppose that $V$ is real analytic such that $\frac{V(x)}{\log(1+|x|)} \to +\infty$ as $x \to \pm \infty$. Suppose that the equilibrium measure $\mu_0$ satisfies Assumptions \ref{assumption:a} and \ref{assumption:b} for some $\kappa=2k$ or $\kappa=2k+1/2$.
  Let $K^X_n(x, y; \tau)$ be the kernel \eqref{eq:KnX} for the eigenvalues of $X = X_{\tau}$ defined in \eqref{def:X}.
  Let $\tau_{\crit}$, $x_{\tau}^*$, $c_{\tau}$, and $\gamma$ be as in \eqref{eq:taucrV}, \eqref{eq:xtauV}, \eqref{def:ctau}, and \eqref{def:gamma}.
  Then, for every  $\tau \in (0, \tau_{\crit})$ and $n \in \mathbb N$, there is
  a function $H_{n} : \mathbb R \to \mathbb R$ such that
  \begin{align} \label{eq:Knlimit}
    \lim_{n \to \infty} 
    \frac{e^{-H_{n}(u; \tau)+H_{n}(v; \tau)}}{c_{\tau} n^{\gamma}}
    K^X_n \left(  x^*_{\tau} + \frac{u}{c_{\tau} n^{\gamma}},  x^*_{\tau} + \frac{v}{c_{\tau} n^{\gamma}}; \tau \right)
    = \K_{\kappa}(u,v),
  \end{align}	
  uniformly for $u$ and $v$ in compact subsets of $\mathbb R$,
  with the same limiting kernel  $\K_{\kappa}$ as in \eqref{def:singlimit}.
\end{theorem}

\begin{remark}\label{remark: H} The gauge factor $e^{-H_n(u; \tau) + H_n(v; \tau)}$ in \eqref{eq:Knlimit} is irrelevant from a probabilistic
 point of view, since it does not contribute to the determinants that define
 the correlation functions of the determinantal point process.
 However, the precise formula of $H_n$ may be of interest, and it can be read off from the formula \eqref{eq:Hnhat} for $\widehat{H}_n$, which is derived
 in a multi-time context for nonintersecting Brownian paths, see below. We have
 \[ H_n(u;\tau) = \widehat{H}_n(u, t) - \frac{\tau n}{8} V'(x^*)^2, \qquad t = \frac{\tau}{1+\tau}. \] 
 Taking note that $\widehat{c}_t(1-t) = c_{\tau}$, see \eqref{eq:constants multi},
 we find from \eqref{eq:Hnhat} 
 \begin{equation} 
	 H_n(u; \tau) =  \frac{u n^{1-\gamma}}{2c_{\tau}} V'(x^*) + \frac{u^2 n^{1-2\gamma}}{4c_0 c_{\tau}} V''(x^*) + O\left(n^{1-3\gamma}\right) \label{eq:expansionHintro} \end{equation}
 as $n \to \infty$.  
 
 With some more effort we can expand the $O$-term and find an expansion
 with terms $a_j u^j n^{1-j \gamma}$ where the $a_j$ are constants that depend
 on $c_0$, $c_{\tau}$ and derivatives of $V$ at $x^*$. We only need terms with
 $1-j \gamma \geq 0$, i.e., $j \leq \gamma^{-1}$, since for $1-j \gamma < 0$, the term $n^{1-j \gamma}$ tends to $0$ as $n \to \infty$ and will not be important for
 the limit \eqref{eq:Knlimit}.   
 Since $\gamma = 1/(\kappa+1)$ and $\kappa =2k$ or $\kappa = 2k+1/2$, we 
 thus only need the terms up to $j = 2k+1$,
 which means that $H_n$ can be taken to be  a polynomial in $u$ of degree 
 $\leq 2k+1$.
 For example, for $\kappa = 2$  we could take
 \[ H_n(u; \tau) =  \frac{un^{2/3}}{2c_{\tau}} V'(x^*) + 
	 \frac{u^2 n^{1/3}}{4c_0 c_{\tau}} V''(x^*) + 	
	 \frac{u^3}{12 c_0^3} V'''(x^*). \]
\end{remark}

\subsection{Nonintersecting Brownian paths}\label{section:Brownian}

We now turn to nonintersecting Brownian paths with confluent ending points, which can be seen as a dynamical generalization of the eigenvalues of the  sum of a Hermitian matrix with a GUE matrix. 
Suppose $x_1(t), \ldots, x_n(t)$ are $n$ particles in independent Brownian motion with diffusion parameter $n^{-1/2}$, starting at $a_1< \cdots < a_n$ at time $0$, and ending at $b_1< \cdots < b_n$ at time $1$. 
We first condition the paths so that they do not intersect for $t\in[0,1]$, and then we take the confluent limit where the ending points $b_1,\ldots, b_n$ all tend to $0$. We take the starting points $x_1(0),\ldots, x_n(0)$
random, following a certain probability distribution. 
We denote the particles in this model of conditioned Brownian paths with random initial positions by $\X(t) = (x_1(t), \dotsc, x_n(t))$.

We assume that the initial distribution of the paths, $\frac{1}{n}\sum_{i=1}^n\delta_{x_i(0)}$, converges almost surely, weakly, to a probability distribution $\mu_0$ which has a singular interior point or a singular edge point $x^*$ satisfying Assumptions \ref{assumption:a} and \ref{assumption:b}. 

Since the joint probability distribution of $\X(t)$ is the same as the one of the eigenvalues of the random matrix $(1-t)X_\tau$ with $\tau=\frac{t}{1-t}$ (see Proposition \ref{prop:NIBM}), we can apply Theorem \ref{thm:mainmacro} and Theorem \ref{thm:macrocritical} to describe the limiting distribution of $\X(t)$ if $t\in[0, t_{\crit}]$, with $t_{\crit}$ given by
\begin{equation}\label{eq:defn_t_crit}
t_{\crit}=\frac{\tau_{\crit}}{1+\tau_{\crit}},
\end{equation}
with $\tau_{\crit}$ as in \eqref{def:taucr}.
This implies that the limiting distribution, as $n\to\infty$, of the particles $\X(t)$ at time $t\in(0,t_{\crit})$ still has a singular point of the same type as $\mu_0$.
More precisely, the limiting particle density of $(1-t)^{-1}\X(t)$ behaves as in  \eqref{eq:propsingularint} or \eqref{eq:propsingularedge} with $\tau=\frac{t}{1-t}$.
This is against the common intuition that the density of Brownian paths smooths out as time goes by. Moreover, the propagation of singularity does not go beyond $t_{\crit}$, and this implies a phase transition.

Now we assume that the initial particles $\X(0)$ have the joint probability density function \eqref{eq:UEpdf}, for some real analytic $V$ such that $\frac{V(x)}{\log(1+|x|)} \to +\infty$ as $x \to \pm \infty$. 
Here it is important to note that \eqref{eq:UEpdf} is the density of the particles $x_1(0),\ldots, x_n(0)$ after conditioning on the fact that the paths do not intersect.
Then the eigenvalue correlation kernel for the particles $(1-t)^{-1}\X(t)$ is given by \eqref{eq:KnX} with $\tau=\frac{t}{1-t}$, hence we can apply Theorem \ref{thm:mainmicro} for $t\in(0,t_{\crit})$, implying the propagation of the singular behavior of the particles $(1-t)^{-1}\X(t)$ on the local level.

From the nonintersecting condition, one can also derive that $\X(t)$ is a multi-time determinantal process on $\mathbb R\times [0,1]$. The $k$-point correlation function $R_k$ of this process evaluated at $k$ points $(x_i, t_i)$, $i=1,\ldots, k$ is given as the determinant
\begin{equation}\label{corr function det}
R_k\left((x_1,t_1),\ldots, (x_k, t_k)\right)
=\det\left(K_n^{\X}(x_i, x_j;t_i, t_j)\right)_{i,j=1}^k,
\end{equation} 
 where the multi-time correlation kernel $K_n^{\X}(x,y;t,t')$ is given by
  \begin{equation} \label{eq:mult_kernel}
  K^{\X}_n\left(x,y;t,t'\right) = -1_{t < t'} G_n(x,y; t, t') + \widetilde{K}^{\X}_n\left(x,y;t,t'\right),
\end{equation}
where
\begin{equation} \label{eq:Gnxy}
  G_n(x,y; t,t') = 
\frac{\sqrt{n}}{\sqrt{2\pi(t'-t)}} e^{-\frac{n(x - y)^2}{2(t'-t)} + \frac{nx^2}{2(1 - t)} - \frac{ny^2}{2(1 - t')}} \end{equation}
and
\begin{equation} \label{eq:mult_kernel_tilde}
  \widetilde{K}^{\X}_n\left(x,y;t,t'\right)= \frac{n}{2\pi i \sqrt{tt'}} \int^{x^* + i\infty}_{x^* - i\infty} dz \int^{\infty}_{-\infty} dw\, K^M_n(z, w) e^{\frac{n}{2}(V(z) - V(w))} 
  e^{\frac{n(x - (1 - t)z)^2}{2t (1 - t)}- 
  	\frac{n(y - (1 - t')w)^2}{2t'(1 - t')}}.
\end{equation}
Alternatively, using \eqref{def:KnPE} we can write $\widetilde{K}_n^{\X}$ as
\begin{equation} \label{eq:mult_kernel_tildePE}
  \widetilde{K}^{\X}_n\left(x,y;t,t'\right)= \frac{n}{2\pi i \sqrt{tt'}} \int^{x^* + i\infty}_{x^* - i\infty} dz \int^{\infty}_{-\infty} dw\, K^{\PE}_n(z, w) 
  e^{\frac{n(x - (1 - t)z)^2}{2t (1 - t)}- 
  	\frac{n(y - (1 - t')w)^2}{2t'(1 - t')}}.
\end{equation}

The expressions \eqref{eq:mult_kernel} and \eqref{eq:mult_kernel_tildePE} for the 
multi-time correlation kernel can be derived with the help of the Eynard--Mehta theorem 
\cite{Eynard-Mehta98, Borodin-Rains05}; they are proved in Appendix \ref{sec:corr_kernel_multi_time}. 
Notice that if $t=t'$, we recover the single-time correlation kernel \eqref{eq:KnX} up to the rescaling implied by \eqref{relation random matrix Brownian paths}. Indeed it is straightforward to see that
\begin{equation} \label{eq:single time equiv}
 \frac{1}{1+\tau}{K}^{\X}_n\left(\frac{x}{1+\tau},\frac{y}{1+\tau}; \frac{\tau}{1+\tau}, \frac{\tau}{1+\tau} \right) =  K_n^X\left(x, y; \tau \right),
 \end{equation}
 with $K_n^X$ given in \eqref{eq:KnX}.

In analogy to the constants $c_\tau$ and $x_\tau^*$, we define
  \begin{equation} \label{eq:constants multi}
    \widehat{c}_t = \frac{c_{\frac{t}{1-t}}}{1 - t}=\frac{t_{\crit}}{t_{\crit}-t}c_0, \qquad \widehat x_t^*= (1-t)x_{\frac{t}{1-t}}^*=(1-t)x^*+\frac{t}{2}V'(x^*),
  \end{equation}
  with $c_{\tau}$ and $x_\tau^*$ as in \eqref{def:ctau} and \eqref{eq:xtauV}, and $c_0$ as in Assumption \ref{assumption:b}.

We now present the multi-time version of Theorem \ref{thm:mainmicro}. 
\begin{theorem} \label{thm:multi}
  Suppose that $V$ is real analytic and such that $\frac{V(x)}{\log(1+|x|)} \to +\infty$ as $x \to \pm \infty$. Suppose that the equilibrium measure $\mu_0$ satisfies Assumptions \ref{assumption:a} and \ref{assumption:b}.
  Let $K^{\X}_n\left(x,y;t,t'\right)$ be the multi-time correlation kernel \eqref{eq:mult_kernel}--\eqref{eq:mult_kernel_tilde} for the particles in nonintersecting Brownian motions $\X(t)$, and let
  $\widetilde{K}_n^{\X}$ and $G_n$ be given by \eqref{eq:mult_kernel_tilde}
  and \eqref{eq:Gnxy}, respectively.
  Suppose $t, t' \in (0, t_{\crit})$. 
  Then there is a function $\widehat H_n(u; t)$ such that   the following hold.
  \begin{enumerate}[label=(\alph*)]
  \item \label{enu:thm:multi_a}
    For every $u, v \in \mathbb R$,
    \begin{equation} \label{eq:Knlimit_multi}  
      \lim_{n \to \infty}    
      \frac{e^{-\widehat H_{n}(u; t)+ \widehat H_{n}(v; t')}}{\sqrt{\widehat c_t \widehat c_{t'}} n^{\gamma}}
      \widetilde{K}^{\X}_n \left(\widehat x^*_t+ \frac{u}{\widehat c_{t} n^{\gamma}}, \widehat x^*_{t'}+ \frac{v}{\widehat c_{t'} n^{\gamma}};t, t' \right)  
      = \K_{\kappa}(u,v) 
    \end{equation}
    with the same limiting kernel  $\K_{\kappa}$ as in \eqref{def:singlimit}.
    The convergence is uniform for $u$ and $v$ in compact subsets of $\mathbb R$. 
  \item \label{enu:thm:multi_b}
    Suppose  $t < t'$ and $u \neq v$. Then  
    \begin{equation} \label{eq:Gnlimit_general} 
      \lim_{n \to \infty}    
      \frac{e^{-\widehat H_{n}(u; t)+ \widehat H_{n}(v; t')}}{\sqrt{\widehat c_t \widehat c_{t'}} n^{\gamma}}
      G_n \left(\widehat x^*_t+ \frac{u}{\widehat c_{t} n^{\gamma}}, \widehat x^*_{t'}+ \frac{v}{\widehat c_{t'} n^{\gamma}};t, t' \right)  
      = 0,
    \end{equation}
    and the convergence is uniform for $(u, v)$ in a compact subset of $\mathbb R^2$ away from the line $u = v$. Furthermore, we have that
    \begin{equation} \label{eq:Gnlimit1} 
      \lim_{n \to \infty}    
      \frac{e^{-\widehat H_{n}(u; t)+ \widehat H_{n}(v; t')}}{\sqrt{\widehat c_t \widehat c_{t'}} n^{\gamma}}
      G_n \left(\widehat x^*_t+ \frac{u}{\widehat c_{t} n^{\gamma}}, \widehat x^*_{t'}+ \frac{v}{\widehat c_{t'} n^{\gamma}};t, t' \right)  
      = \delta(u - v),
    \end{equation}
    with $\delta(u - v)$ being the Dirac $\delta$-function, see also
    Remark \ref{rem:deltafunction}.
  \end{enumerate}
\end{theorem}
The function $\widehat H_n$ is defined in \eqref{eq:Hnhat}.

By \eqref{eq:single time equiv}, the kernel \eqref{eq:KnX} can be recovered from the multi-time kernel \eqref{eq:mult_kernel} if we set $t=t'=\frac{\tau}{1+\tau}$. Using this relation, one checks directly that Theorem \ref{thm:mainmicro} is a consequence of the more general Theorem \ref{thm:multi}.
The proof of Theorem \ref{thm:multi}\ref{enu:thm:multi_a} is in Section \ref{section: proof main result (a)} and the proof of Theorem \ref{thm:multi}\ref{enu:thm:multi_b} is in Section \ref{section: proof main result (b)}.

\begin{remark} \label{rem:deltafunction}
	The convergence in \eqref{eq:Gnlimit1}
	is in the weak sense. That is, if $0 \leq t < t' < t_{\crit}$ and
	\begin{equation} \label{eq:Fnuv}
	F_n(u,v; t,t') = 
	\frac{e^{-\widehat H_{n}(u; t)+ \widehat H_{n}(v; t')}}{\sqrt{\widehat c_t \widehat c_{t'}} n^{\gamma}} G_n\left(\widehat x^*_t+ \frac{u}{\widehat c_{t} n^{\gamma}}, \widehat x^*_{t'}+ \frac{v}{\widehat c_{t'} n^{\gamma}}\right),
	\end{equation} 
	then 
	\begin{equation} \label{eq:Fnuvconverg}
          F_n(u,v; t, t') \to \delta(u-v)
        \end{equation}
	means that 
	\[ \lim_{n \to \infty} \iint F_n(u,v;t,t') \varphi(u,v) dudv =
		\int \varphi(u,u) du \]
	holds for every continuous test function $\varphi$ with compact support.
\end{remark}

\begin{remark}\label{remark: Gaussian increments}

By \eqref{eq:mult_kernel}, the two parts \ref{enu:thm:multi_a} and \ref{enu:thm:multi_b} of Theorem \ref{thm:multi} yield
\begin{equation} \label{eq:Knlimit_multi2}
\lim_{n \to \infty}    
\frac{e^{-\widehat H_{n}(u; t)+ \widehat H_{n}(v; t')}}{\sqrt{\widehat c_t \widehat c_{t'}} n^{\gamma}}
{K}^{\X}_n \left(\widehat x^*_t+ \frac{u}{\widehat c_{t} n^{\gamma}}, \widehat x^*_{t'}+ \frac{v}{\widehat c_{t'} n^{\gamma}};t, t' \right)  
= -1_{t<t'} \delta(u - v) + \K_{\kappa}(u,v).
\end{equation} 

  The limiting correlation kernel \eqref{eq:Knlimit_multi2} implies that the limiting distribution of particles around $\widehat{x}^*_t$, in the scale of $O(n^{-\gamma})$, does not change over the time $t \in [0, t_{\crit})$ after centered at $\widehat{x}^*_t$ and scaled by $\widehat{c}_t n^{\gamma}$. The infinity of the limiting correlation kernel at positions $u = v$ seems to be an anomaly, but it is natural and necessary. Consider particles initially at time $0$ distributed as a determinantal process with correlation kernel $\K_{\kappa}(u, v)$, and suppose they do not move at all over all positive time $t > 0$, then the particles are a multi-time determinantal process, and the multi-time correlation kernel is given exactly by the right-hand side of \eqref{eq:Knlimit_multi2}.
In our nonintersecting Brownian motion model, the particles are not stationary, but the fluctuation is at a scale smaller than $n^{-\gamma}$.
\end{remark}

\begin{remark} \label{remark: Gaussian scaling}
The Dirac $\delta$-function on the right-hand side of \eqref{eq:Gnlimit1} arises as a limit of Gaussian distributions which are sharply peaked on the $O(n^{-\gamma})$ scale. To detect these Gaussian fluctuations we can consider a finer scaling. Suppose now that $u$ is fixed and that $v$ depends on $n$ in such a way that the distance between $u$ and $v$ is of order $n^{-1/2+\gamma}$ as $n\to\infty$, say
\begin{equation}\label{def v x}
v=u+\frac{s}{n^{1/2-\gamma}}.
\end{equation}
We assume for simplicity that $\gamma>1/4$, which is the case if the exponent $\kappa$ of the singular point $x^*$ is smaller than $3$, see \eqref{def:gamma}.
Then we obtain a different scaling limit if $t <t'$, namely with a Gaussian limiting kernel,
\begin{equation} 
    \lim_{n \to \infty}    
         F_n \left(\widehat x^*_t+ \frac{u}{\widehat c_{t} n^{\gamma}}, \widehat x^*_{t'}+ \frac{v}{\widehat c_{t'} n^{\gamma}};t,t' \right)  
         =\frac{1}{\sqrt{2\pi} \sigma_{t,t'}}e^{-\frac{s^2}{2\sigma^2_{t,t'}}},
\label{eq:multiscalingGaussian}
  \end{equation}
  where the variance is given by
\begin{equation}
\label{def:sigmat} \sigma_{t,t'}^2=\widehat c_t\widehat c_{t'}(t'-t)=c_0^2t_{\crit}^2\frac{t'-t}{(t_{\crit}-t)(t_{\crit}-t')}.
\end{equation}
This follows from the proof of Theorem \ref{thm:multi}\ref{enu:thm:multi_b}, see \eqref{eq:peaked_gaussian}--\eqref{eq:Rnulimit} .
\end{remark}

\begin{remark}\label{remark: Gaussian increments2}
The probability distribution of having a path at position $y$ at time $t'>t$ given that there is a path at position $x$ at time $t$ is given in terms of the two-point correlation function $R_2$ by
\[\frac{R_2\left((x,t);(y,t')\right)}{K_n^{\X}\left(x,x;t,t\right)}dy
=\frac{1}{K_n^{\X}\left(x,x;t,t\right)}\det\begin{pmatrix}K_n^{\X}\left(x,x;t,t\right)&K_n^{\X}\left(x,y;t,t'\right)
\\K_n^{\X}\left(y,x;t',t\right)&K_n^{\X}\left(y,y;t',t'\right)\end{pmatrix}dy.\]
If we scale $x$ and $y$ by setting, as in Remark \ref{remark: Gaussian scaling},
\[x=\widehat x_t^*+\frac{u}{\widehat c_tn^\gamma},\qquad y=\widehat x_{t'}^*+\frac{u}{\widehat c_{t'}n^\gamma}+\frac{s}{\widehat c_{t'}n^{1/2}},\]
for some fixed $s\in\mathbb R$,
we can use  \eqref{eq:multiscalingGaussian} and \eqref{eq:Knlimit_multi}--\eqref{eq:Gnlimit_general} to conclude that, as $n\to\infty$, this distribution converges to
\[\frac{\widehat c_{t'} \sqrt{n}}{\sqrt{2\pi} \sigma_{t,t'}}e^{-\frac{s^2}{2\sigma^2_{t,t'}}}dy=\frac{1}{\sqrt{2\pi} \sigma_{t,t'}}e^{-\frac{s^2}{2\sigma^2_{t,t'}}}ds.\]

The heuristic interpretation of this fact is as follows: if there is a path at position $x = \widehat{x}^*_t + u/(\widehat{c}_t n^{\gamma})$ at time $t$, then it will be at position
\[y=\widehat x_{t'}^*+\frac{u}{\widehat c_{t'}n^\gamma}+ \frac{\sigma_{t,t'}}{\widehat c_{t'}n^{1/2}}\chi+o(n^{-1/2})=\widehat x_{t'}^*+\frac{u}{\widehat c_{t'}n^\gamma}+ \sqrt{\frac{(t'-t)(t_{\crit}-t')}{t_{\crit}-t}}\frac{1}{n^{1/2}}\chi+o(n^{-1/2})\]
at time $t'$ as $n\to\infty$ if $t<t'<t_{\crit}$,
where $\chi$ is a standard Gaussian random variable. This suggests that the path is approximated by a Brownian bridge \cite[Section 5.6 B]{Karatzas-Shreve91} between times $t$ and $t_{\crit}$, with the start $x$ and the end $\widehat{x}^*_{t_{\crit}}$, and diffusion parameter $n^{-1/2}$ that is the same as the diffusion parameter of an individual Brownian path in the model. The rigorous justification of the Brownian bridge approximation is out of the scope of the current paper.
\end{remark}

\section{Proofs of Theorem \ref{thm:mainmacro} and Theorem \ref{thm:macrocritical}}

\subsection{Free convolution with a semi-circle law}

In \cite{Biane97}, Biane describes how to calculate $\mu_{\tau} = \mu_0 \boxplus \lambda_{\tau}$ and its Cauchy transform,
\begin{equation} \label{def:Gtau} 
  G_{\mu_{\tau}}(z) = \int \frac{d\mu_{\tau}(s)}{z-s}, \qquad \Im z > 0, 
\end{equation}
from the knowledge of $G_{\mu_0}$. By \cite[Corollary 1]{Biane97}, we can extend the domain of $G_{\mu_{\tau}}(z)$ continuously to $\mathbb C^+ \cup \mathbb R$, and let
\begin{equation}
  G_{\mu_{\tau}}(x) = \lim_{y \to 0_+} G_{\mu_{\tau}}(x + iy).
\end{equation}
By the regularity of $\mu_{\tau}$ shown in \cite{Biane97}, particularly \cite[Corollaries 4 and 5]{Biane97}, the density of the measure $\mu_{\tau}$ can be recovered from its Cauchy transform by the formula
\begin{equation}
  \psi_{\tau}(x) = -\frac{1}{\pi} \Im G_{\mu_{\tau}}(x).
\end{equation}

For all $\tau > 0$, let
\begin{equation} \label{def:ytau} 
  y_{\tau}(x) = \inf \left\{ y > 0 \, \bigg| \, \int \frac{d\mu_0(s)}{(x-s)^2 + y^2} \leq \frac{1}{\tau} \right\}, 
  \qquad x \in \mathbb R.
\end{equation}
Then $x \mapsto y_{\tau}(x)$ is continuous \cite[Lemma 2]{Biane97}, and 
\begin{equation} \label{def:Omega} 
  \Omega_{\tau} = \{ x + iy \in \mathbb C^+ \mid y > y_{\tau}(x) \} 
\end{equation}
is the domain above the graph of $y_{\tau}$.
Biane \cite[Lemma 4]{Biane97} shows that $w \mapsto w + \tau G_{\mu_0}(w)$
is a conformal map from $\Omega_{\tau}$ onto $\mathbb C^+$ that extends continuously
to the closure. The boundary of $\Omega_{\tau}$ 
(which is the graph of $y_{\tau}$) is mapped bijectively to the
real axis. 
Furthermore, let $F_{\tau} : \overline{\mathbb C^+} \to \overline{\Omega}_{\tau} $ be the inverse mapping. Then by \cite[Proposition 2 and Corollary 2]{Biane97} \footnote{In \cite[Corollary 2]{Biane97}, a negative sign is missing for $p_t(x)$.}
\begin{equation} \label{eq:Gtau} 
  G_{\mu_{\tau}}(z) = G_{\mu_0}(F_{\tau}(z)) \quad \text{and} \quad  \psi_{\tau}(x) = - \frac{1}{\pi} \Im G_{\mu_0}(F_{\tau}(x)), \quad \text{for $x \in \mathbb R$}. 
\end{equation}

In the situation of a singular point specified in Assumption \ref{assumption:b}\ref{enu:assumption:b_1} or Assumption \ref{assumption:b}\ref{enu:assumption:b_2}, and if $0 < \tau \leq \tau_{\crit}$, it follows from \eqref{def:taucr} and \eqref{def:ytau}
that
\begin{equation}
  \int \frac{d\mu_0(s)}{(x^* - s)^2} \leq \frac{1}{\tau}.
\end{equation}
Thus $y_{\tau}(x^*) = 0$ and $x^*$ belongs to the boundary of $\Omega_{\tau}$. It is mapped by $w \mapsto w + \tau G_{\mu_0}(w)$ to
\begin{equation}
  x^* + \tau G_{\mu_0}(x^*) = x_{\tau}^*,
\end{equation}
see \eqref{def:xtau}. The inverse  $F_{\tau}$ maps $x_{\tau}^*$ back to $x^*$
and then it is clear by \eqref{eq:Gtau} that
\begin{equation} \label{eq:reality_of_G_mu_tau(x*_tau)}
  G_{\mu_{\tau}}(x^*_{\tau}) = G_{\mu_0}(x^*) = \int \frac{d\mu_0(s)}{x^*-s} \quad \text{is real.} 
\end{equation}
Thus by \eqref{eq:Gtau} the density of $\mu_{\tau}$ vanishes 
at $x_{\tau}^*$.

Under Assumption \ref{assumption:b}\ref{enu:assumption:b_2}, we have the following result.
\begin{lemma} \label{lem:right_most_x^*}
  If $x^*$ is the rightmost edge of the support of $\mu_0$, and $0 < \tau \leq \tau_{\crit}$, then $\psi_{\tau}(x) = 0$ for all $x > x^*_{\tau}$.
\end{lemma}
\begin{proof}
  For $\tau \in (0, \tau_{\crit}]$,
  \begin{equation}
    \int \frac{d\mu_0(s)}{(x - s)^2} < \int \frac{d\mu_0(s)}{(x^* - s)^2} \leq \frac{1}{\tau}, \quad \text{for all $x > x^*$},
  \end{equation}
  which implies that $y_{\tau}(x) = 0$ for all $x > x^*$, and then the interval $[x^*, \infty)$ is part of the boundary of $\Omega_{\tau}$. By the bijectivity of $F_{\tau}$, we have that $F_{\tau}$ maps the interval $[x^*_{\tau}, \infty)$ bijectively and continuously to $[x^*, \infty)$. Then it is clear that by \eqref{eq:Gtau} and analogous to \eqref{eq:reality_of_G_mu_tau(x*_tau)},
  \begin{equation}
    G_{\mu_{\tau}}(x) = G_{\mu_0}(F_{\tau}(x)) = \int \frac{d\mu_0(s)}{F_{\tau}(x) - s} \quad \text{is real, for all $x > x^*_{\tau}$}.
  \end{equation}
  We conclude that $\psi_\tau(x)=0$ for all $x\geq x^*_{\tau}$.
\end{proof}

Below we establish the precise order of 
vanishing as given in \eqref{eq:propsingularint} and \eqref{eq:propsingularedge} for $0<\tau<\tau_{\crit}$, and their critical counterparts as $\tau=\tau_{\crit}$.

\subsection{Singular interior point} \label{subsec:interior_global}

\subsubsection{Preliminaries} \label{subsubsec:preliminaries_interior}

We assume that $x^*$ is a singular interior point as in Assumption \ref{assumption:b}\ref{enu:assumption:b_1} and use the fact that the density $\psi_0(s)$
is analytic near $s=x^*$. This implies that $G_{\mu_0}$ as given by 
\eqref{def:Gtau} with $\tau = 0$ has an analytic continuation into the
lower half-plane in a neighborhood of $x^*$, say for $|w-x^*| < \delta$ for some
$\delta > 0$. We denote the analytic continuation also
by $G_{\mu_0}$ and it is given by
\begin{equation} \label{eq:G_mu_0(z)_continuation}
  G_{\mu_0}(w) = \int \frac{d\mu_0(s)}{w-s} - 2\pi i \psi_0(w) \quad \text{for }
  \Im w < 0, \, |w - x^*| < \delta.
\end{equation}
Let 
\begin{equation} \label{eq:Gmu0expansion1}
G_{\mu_0}(w) = - \sum_{j=0}^{\infty} g_j (w-x^*)^j 
\end{equation}
be the power series expansion of $G_{\mu_0}$ around $x^*$. Since $\psi_0$ vanishes
to order $2k$ at $x^*$, the coefficients $g_0, \ldots, g_{2k-1}$ are real and
they are given by
\begin{equation} \label{defn_g_j_interior}
  g_j = \int \frac{d\mu_0(s)}{(s-x^*)^{j+1}} =  c_0^{2k+1}\int_{\supp(\mu_0)} (s-x^*)^{2k-j-1} h(s) ds, \quad j = 0, 1, \dotsc, 2k - 1.
\end{equation}
In particular 
\begin{equation} \label{def:g1} 
  g_0 = -G_{\mu_0}(x^*) = -\frac{x^*_{\tau} - x^*}{\tau}, 
  \qquad g_1 = \frac{1}{\tau_{\crit}},
\end{equation}
see \eqref{def:taucr} and \eqref{def:xtau}. The coefficient $g_{2k}$ has a non-zero
imaginary part and is given by 
\begin{equation} \label{eq:Img2k} 
  g_{2k} = \lim_{y \to 0_+} \int \frac{d\mu_0(s)}{(s-(x^* + iy))^{2k +1}} = c_0^{2k+1}\dashint_{\supp(\mu_0)} \frac{h(s)}{s-x^*} ds + c_0^{2k+1}\pi i.
\end{equation}
Using \eqref{def:xtau} and \eqref{def:g1}, and \eqref{eq:Gmu0expansion1} we have the 
power series  in a neighborhood of $x^*$
\begin{equation} \label{eq:wplustauGmu0} 
  w + \tau G_{\mu_0}(w) 
  =  x^*_{\tau} + \left(1 - \frac{\tau}{\tau_{\crit}}\right) (w-x^*) - \tau \sum^{\infty}_{j=2} g_j (w-x^*)^j. 
\end{equation}

\subsubsection{Proof of Theorem \ref{thm:mainmacro}\ref{enu:thm:mainmacro_1}} \label{subsubsec:interior_subcrit}

We compute the expansion of $F_{\tau}(z)$, which is the inverse function 
of $w + \tau G_{\mu_0}(w)$ as in \eqref{eq:wplustauGmu0}, 
under the condition $\tau < \tau_{\crit}$. Then $w + \tau G_{\mu_0}(w)$ is 
analytic around $x^*$ with a non-zero derivative at $x^*$, see \eqref{eq:wplustauGmu0},
and thus has an analytic inverse locally near $x^*_{\tau} = x^* + \tau G_{\mu_0}(x^*)$.
The domain of the local inverse function overlaps with that of $F_{\tau}$, 
and in the common domain they coincide due to their uniqueness. 

It follows that $F_{\tau}$ also has an analytic continuation from the upper half plane to a full neighborhood of $x_\tau^*$, and we consider $F_{\tau}$ on this extended domain.  
We write the power series for $\lvert z - x^*_{\tau} \rvert < \epsilon$ 
where $\epsilon > 0$ is small enough,
\begin{equation} \label{eq:Ftauexpansion} 
  F_{\tau}(z) = x^* + \sum^{\infty}_{l = 1} f_l (z-x^*_{\tau})^l
\end{equation}
with certain coefficients $f_1, f_2, \ldots$. Plugging the expansion \eqref{eq:Ftauexpansion}  of $w= F_{\tau}(z)$ into $z = w + \tau G_{\mu_0}(w)$, by \eqref{eq:wplustauGmu0}, we have  
\begin{equation} \label{eq:asy_exp_F_tau}
  z - x^*_{\tau} = \left(1 - \frac{\tau}{\tau_{\crit}}\right) \sum_{l = 1}^{2k} f_l (z-x^*_{\tau})^l - \tau \sum^{2k}_{j = 2}  g_j \left[ \sum^{2k}_{l = 1} f_l (z - x^*_{\tau})^l \right]^j + O((z-x^*_{\tau})^{2k+1}),
\end{equation}
as $z\to x^*_{\tau}$,
and by comparing coefficients, this identity recursively yields identities for $f_j$, $j=0,\ldots, f_{2k-1}$ in terms of $g_0, \dotsc, g_j$. For instance, we find
\begin{equation} \label{def:f1} 
  f_1 = \left(1 - \frac{\tau}{\tau_{\crit}}\right)^{-1} =  \frac{\tau_{\crit}}{\tau_{\crit}-\tau}  > 0.
\end{equation}

The next coefficients $f_2, \ldots, f_{2k-1}$  are all real, since 
$g_0, g_1, \ldots, g_{2k - 1}$ are real, but their
precise values are not of interest to us. The coefficient $g_{2k}$ is not real by \eqref{eq:Img2k}, and as a result $f_{2k}$ has a non-zero imaginary part.
Equating the coefficients of $(z-x^*_{\tau})^{2k}$ on both sides of \eqref{eq:asy_exp_F_tau} and keeping only the imaginary part, we obtain the identity 
\begin{equation} \label{eq:intermediate_step}
  0 = \left(1 - \frac{\tau}{\tau_{\crit}} \right) \left(\Im f_{2k}\right)- \tau \left(\Im g_{2k} \right) f_1^{2k},
\end{equation}
which by \eqref{eq:Img2k}, \eqref{def:f1}, and \eqref{def:ctau} leads to
\begin{equation}  \label{eq:Imf2k}
  \Im f_{2k} = \tau \pi \left( \frac{\tau_{\crit}}{\tau_{\crit}-\tau}\right)^{2k+1} c_0^{2k+1}=\tau \pi  c_\tau^{2k+1}.
\end{equation}

We next insert the expansion \eqref{eq:Ftauexpansion} for $w = F_{\tau}(z)$ 
back into \eqref{eq:Gmu0expansion1} around $z = x^*_{\tau}$.
Then $G_{\mu_0}(F_{\tau}(z))$ has a convergent power series in 
a neighborhood of $x_\tau^*$, say for $\lvert z - x^*_{\tau} \rvert < \epsilon$ 
with some $\epsilon > 0$.
By \eqref{def:Gtau} we obtain the expansion up to order $(z - x^*_{\tau})^{2k}$,
\begin{equation}\label{eq:expansion G}
  G_{\mu_{\tau}}(z) = G_{\mu_0}(F_{\tau}(z)) = -\sum_{j = 0}^{2k}  g_j \left[ \sum_{l = 1}^{2k} f_l (z-x^*_{\tau})^l \right]^j + O((z-x_{\tau}^*)^{2k+1}).
\end{equation}
We now let $s \in \mathbb R$ be in the interval 
$(x^*_{\tau} - \epsilon, x^*_{\tau} + \epsilon)$, and substitute $z = s$ in \eqref{eq:expansion G}. Since the coefficients $g_1, \ldots, g_{2k-1}$,
$f_1, \ldots, f_{2k-1}$ are real, there are only two contributions on the right-hand side of \eqref{eq:expansion G}
to the  imaginary part up to order $(s - x^*_{\tau})^{2k}$, and we have
\begin{equation}
  \Im G_{\mu_{\tau}}(s) = (- g_1 \Im f_{2k} - f^{2k}_1 \Im g_{2k})  (s - x^*_{\tau})^{2k} + O((s - x^*_{\tau})^{2k+1}),\qquad s\to x^*_{\tau},
\end{equation}
which by \eqref{def:g1}, \eqref{eq:Img2k}, \eqref{def:f1}, and  \eqref{eq:Imf2k} leads to
\begin{equation} \label{eq:ImGtau} 
  \Im G_{\mu_\tau}(s) = - \pi c_0^{2k+1} (s-x_{\tau}^*)^{2k} + O ( (s-x_{\tau}^*)^{2k+1} ),\qquad s\to x_\tau^*.
\end{equation}
Then by \eqref{eq:Gtau} we obtain \eqref{eq:propsingularint} and prove Theorem \ref{thm:mainmacro}\ref{enu:thm:mainmacro_1}.

\subsubsection{Proof of Theorem \ref{thm:macrocritical}\ref{enu:thm:macrocritical_I}, \ref{enu:thm:macrocritical_II}, \ref{enu:thm:macrocritical_III}}\label{subsubsec:interior_crit}

We now assume $\tau = \tau_{\crit}$. Then, the series \eqref{eq:wplustauGmu0} becomes 
\begin{equation}\label{eq:wcrit}
  w + \tau_{\crit} G_{\mu_0}(w) = x^*_{\tau_{\crit}} - \tau_{\crit} \sum^{\infty}_{j=2} g_j (w-x^*)^j
\end{equation}
such that its first derivative at $x^*$ vanishes. We would like to extend $F_{\tau_{\crit}}(z)$, the inverse function of $w + \tau_{\crit} G_{\mu_0}(w)$ from $\mathbb C^+$ to $\Omega_{\tau_{\crit}}$, to a neighborhood of $x^*_{\tau_{\crit}}$, as we did in Section \ref{subsubsec:interior_subcrit}. 
But due to the vanishing of the first derivative, the inverse function of $w + \tau_{\crit} G_{\mu_0}(w)$ is no longer well defined locally in a neighborhood 
of $x^*_{\crit}$, although it still can be defined with a branch cut. The behavior of the local inverse function (with a branch cut) depends on the first non-zero term in the series in the right-hand side of \eqref{eq:wcrit}.

If $k=1$, we have by \eqref{eq:Img2k}, 
\begin{equation} \label{eq:g2CaseI}
  g_2 = c_0^{3}\,\dashint_{\supp(\mu_0)} \frac{h(s)}{s-x^*} ds + \pi i c_0^3.
\end{equation}
Thus $-\tau_{\crit} g_2$ is the first nonzero coefficient, 
and it has a negative imaginary part. This is our \textbf{Case \ref{enu:thm:macrocritical_I}} and it corresponds
to part \ref{enu:thm:macrocritical_I} of Theorem \ref{thm:macrocritical}.

If $k \geq 2$, we have
\begin{equation} \label{eq:g2CaseII}
  g_2 = c_0^{2k+1}\int_{\supp(\mu_0)} (s-x^*)^{2k-3}h(s)ds. 
\end{equation}
Generically $g_2 \neq 0$, and $-\tau_{\crit} g_2$ is the first nonzero coefficient. 
This is \textbf{Case \ref{enu:thm:macrocritical_II}} and it corresponds to part \ref{enu:thm:macrocritical_II} of the theorem. 
The difference with Case \ref{enu:thm:macrocritical_I} is that now $g_2$ is real.

But $g_2$ could vanish, and it does so for instance if $h$ is symmetric about $x^*$. 
So besides Cases \ref{enu:thm:macrocritical_I} and \ref{enu:thm:macrocritical_II}, we have \textbf{Case \ref{enu:thm:macrocritical_III}}, when $g_2 = 0$, and 
$-\tau_{\crit} g_3$ is the first nonzero coefficient. 
Since $h(s) > 0$ almost everywhere on $\supp(\mu_0)$, we always have
\begin{equation} \label{eq:positive_g_3}
  g_3 = c_0^{2k+1} \int (s-x^*)^{2k-4}h(s)ds>0.
\end{equation}

For each of the three cases, we need the following result that generalizes the 
analytic inverse function theorem:
\begin{lemma} \label{lem:first_Puiseux}
  Let $f(w)$ be an analytic function in a neighborhood of $a$, such that $f(a) = b$, the derivatives $f'(a) = f''(a) = \dotsb = f^{(m - 1)}(a) = 0$, and the $m$-th derivative $f^{(m)}(a) = m! c^m$ for some $c \neq 0$. Let $\epsilon > 0$ be small enough. Then on $\{ z \in \mathbb C \mid \lvert z - b \rvert < \epsilon \, \pi < \arg(z - b) < \pi \}$, there exist $m$ analytic functions $\varphi_0(z), \varphi_1(z), \dotsc, \varphi_{m - 1}(z)$ defined by the Puiseux series
  \begin{equation}
    \varphi_j(z) = a + c^{-1} e^{2j\pi i/m}  (z - b)^{1/m} + \sum^{\infty}_{l = 2} c_{l, j} (z - b)^{l/m}, \quad j = 0, 1, \dotsc, m - 1,
  \end{equation}
  where we take the principal value for power functions, such that $f(\varphi_j(z)) = z$, and they are the only functions that satisfy the condition.
\end{lemma}

Lemma \ref{lem:first_Puiseux} is proved by an application of the usual analytic 
inverse function theorem, see e.g.\ \cite[Pages 105--106]{Noguchi98}. 

In what follows, when we consider the value of the function $(z-x^*_{\tau_{\crit}})^{\gamma}$ for $z \in (-\infty, x^*_{\tau_{\crit}})$, 
we take it as the continuation from $\mathbb C^+$. 

\paragraph{Case \ref{enu:thm:macrocritical_I}: Proof of Theorem \ref{thm:macrocritical}\ref{enu:thm:macrocritical_I}}

Note that $\Im g_2 > 0$ by \eqref{eq:g2CaseI}, and in this case we take the value of $\sqrt{-g_2}$ in the way that $3\pi/2 < \arg\sqrt{-g_2} < 2\pi$.  
Using Lemma \ref{lem:first_Puiseux} with $f(w) = w + \tau_{\crit} G_{\mu_0}(w)$, 
$a = x^*$, $b = x^*_{\tau_{\crit}}$ and $c = \sqrt{\tau_{\crit}} \sqrt{-g_2}$, 
we have that for $\lvert z - x^*_{\tau_{\crit}} \rvert < \epsilon$ and $-\pi < \arg(z - x^*_{\tau_{\crit}}) < \pi \}$ where $\epsilon > 0$ is small enough, there exist 
analytic functions
\begin{equation}\label{eq:Fcrit}
  \varphi_j(z) = x^* + \frac{(-1)^j}{\sqrt{\tau_{\crit}} \sqrt{-g_2}}(z - x^*_{\tau_{\crit}})^{1/2} + \sum^{\infty}_{l = 2} c_{l, j} (z - x^*_{\tau_{\crit}})^{l/2}, \quad j = 0, 1,
\end{equation}
such that $f(\varphi_j(z)) = z$. Since $F_{\tau_{\crit}}(z)$ is the inverse function of $f(w)$ from $\mathbb C^+$ to $\Omega_{\tau_{\crit}}$, it also satisfies 
$f(F_{\tau_{\crit}}(z)) = z$. It is not hard to check that in the common domain 
where $\lvert z - x^*_{\tau_{\crit}} \rvert < \epsilon$ and $\Im z > 0$,  $F_{\tau_{\crit}}(z)$ agrees with $\varphi_0(z)$ but not with $\varphi_1(z)$. 
Thus $F_{\tau_{\crit}}(z)$ can be extended to be defined for 
$\lvert z - x^*_{\tau_{\crit}} \rvert < \epsilon$ and 
$ -\pi < \arg(z - x^*_{\tau_{\crit}}) < \pi$ by the formula of $\varphi_0(z)$ in \eqref{eq:Fcrit}. 
We assume this extension, and use the Puiseux expansion \eqref{eq:Fcrit} 
with $j = 0$ for $F_{\tau_{\crit}}(z)$ around $x^*_{\tau_{\crit}}$. 

By \eqref{eq:Gtau}, \eqref{eq:Gmu0expansion1}, and \eqref{eq:Fcrit}, we can then compute the local behavior of the density $\psi_{\tau_{\crit}}(s)$ as $s \to x^*_{\tau_{\crit}}$:
\begin{equation}
  \psi_{\tau_{\crit}}(s) = \frac{1}{\pi} \Im \left[ g_0 +  \frac{g_1}{\sqrt{\tau_{\crit}} \sqrt{-g_2}}(s - x^*_{\tau_{\crit}})^{1/2} + O(s - x^*_{\tau_{\crit}})  \right],
\end{equation}
where $(s - x^*_{\tau_{\crit}})^{1/2}$ is positive if $s > x^*_{\tau_{\crit}}$ 
and is equal to $i \lvert s - x^*_{\tau_{\crit}} \rvert^{1/2}$ if $s < x^*_{\tau_{\crit}}$. Since $g_0$ and $g_1 = \tau^{-1}_{\crit}$ are real, we obtain
\begin{equation}\label{eq:psilocalcrit}
  \psi_{\tau_{\crit}}(s) = \frac{1}{\pi\tau_{\crit}^{3/2}}\Im \left[ \frac{1}{\sqrt{-g_2}} (s-x^*_{\tau_{\crit}})^{1/2} \right] + O(s-x_{\tau_{\crit}}^*),\qquad s \to x^*_{\tau_{\crit}}.
\end{equation}
Writing $g_2=-c_0^{3}re^{-i\theta}$, where $r$ and $\theta$ have values as in \eqref{def:rtheta}, we finally obtain
\eqref{eq:psilocalcritthmI}. This proves Theorem \ref{thm:macrocritical}\ref{enu:thm:macrocritical_I}.

\paragraph{Case \ref{enu:thm:macrocritical_II}: Proof of Theorem \ref{thm:macrocritical}\ref{enu:thm:macrocritical_II}}

Here $g_2$ is real by \eqref{eq:g2CaseII}, and nonzero. We let $\sqrt{-g_2}$ be real and positive if $g_2 < 0$, and be equal to $-i g^{1/2}_2$ if $g_2 > 0$. As in Case \ref{enu:thm:macrocritical_I}, we use Lemma \ref{lem:first_Puiseux} with $f(w) = w + \tau_{\crit} G_{\mu_0}(w)$, $a = x^*$, $b = x^*_{\tau_{\crit}}$, and $c = \sqrt{\tau_{\crit}} \sqrt{-g_2}$, and have that 
in a region where $\lvert z - x^*_{\tau_{\crit}} \rvert < \epsilon$ and 
$ -\pi < \arg(z - x^*_{\tau_{\crit}}) < \pi $ where $\epsilon > 0$ is a small enough, $\varphi_0(z)$ and $\varphi_1(z)$ are well defined in \eqref{eq:Fcrit} such that $f(\varphi_j(z)) = z$. 
By the same argument as in Case \ref{enu:thm:macrocritical_I}, we can check that $F_{\tau_{\crit}}(z)$ agrees with $\varphi_0(z)$ but not with $\varphi_1(z)$. 
So we use the Puiseux expansion \eqref{eq:Fcrit} with $j = 0$ 
for $F_{\tau_{\crit}}(z)$ around $x^*_{\tau_{\crit}}$.

The proof now proceeds in a manner completely analogous to Case \ref{enu:thm:macrocritical_I}. In equation \eqref{eq:wcrit} we replace $w$ with the    Puiseux expansion \eqref{eq:Fcrit} with $j = 0$, i.e. $w = \varphi_0(z) = F_{\tau_{\crit}}(z)$. Keeping terms in the Puiseux expansion up to $O((z - x^*_{\tau_{\crit}})^k)$, we have that as $z\to x^*_{\tau_{\crit}}$ with 
$-\pi < \arg(z - x^*_{\tau_{\crit}}) < \pi$,
\begin{equation} \label{eq:Puiseux_case_II_pos}
  z = x^*_{\tau_{\crit}} -\tau_{\crit} \sum^{2k}_{j = 2}  g_j \left[ \frac{1}{\sqrt{\tau_{\crit}} \sqrt{-g_2}} (z - x^*_{\tau_{\crit}})^{1/2} + \sum^{2k - 1}_{l = 2} c_{l, 0} (z - x^*_{\tau})^{l/2} \right]^j + O ( (z-x^*_{\tau_{\crit}})^{(2k + 1)/2} ).
\end{equation}
Similar to the coefficients in \eqref{eq:Ftauexpansion}, the coefficient
$c_{l, 0}$ depends on $g_2, g_3, \dotsc, g_{l+1}$, and can be computed recursively. 

First we consider the case that $g_2 < 0$. Then $(\sqrt{\tau_{\crit}} \sqrt{-g_2})^{-1}$ is real and positive. 
Since $g_2, \ldots, g_{2k - 1}$ are all real, we have that $c_{2, 0}, \ldots, 
c_{2k - 2, 0}$ are all real, and we do not need their explicit expressions. 
Then analogous to \eqref{eq:intermediate_step}, by comparing the imaginary part of the coefficients of the $(z - x^*_{\tau_{\crit}})^{k}$ term, we have
\begin{equation} \label{eq:im_part_Case_II}
  0 = -\tau_{\crit} \left( 2 g_2 \frac{1}{\sqrt{\tau_{\crit}} \sqrt{-g_2}} \Im c_{2k - 1, 0} + \left( \frac{1}{\sqrt{\tau_{\crit}} \sqrt{-g_2}} \right)^{2k} \Im g_{2k} \right),
\end{equation}
and obtain by \eqref{eq:Img2k} ,
\begin{equation} \label{eq:c_2k-1_0_Case_II}
  \Im c_{2k - 1, 0} = \frac{\pi c_0^{2k+1}}{2\sqrt{-\tau^{2k - 1}_{\crit} g^{2k + 1}_2}} > 0.
\end{equation}
Then using the Puiseux expansion \eqref{eq:Fcrit} (with $j = 0$) of $F_{\tau_{\crit}}(z)$, we have that as $s \to \left(x^*_{\tau_{\crit}}\right)_-$, by \eqref{eq:Gtau}, \eqref{eq:Gmu0expansion1}, and \eqref{def:g1}, 
\begin{equation} \label{eq:Case_II_g_2<0_left}
  \begin{split}
    \psi_{\tau_{\crit}}(s) = {}& -\frac{1}{\pi} \Im G_{\mu_0}(F_{\tau_{\crit}}(s)) \\
    = {}& -\frac{1}{\pi} \Im \left[- g_1 \frac{1}{\sqrt{\tau_{\crit}} \sqrt{-g_2}} (s - x^*_{\tau_{\crit}})^{1/2}  + O(s - x^*_{\tau_{\crit}}) \right] \\
    = {}& \frac{1}{\pi \tau^{3/2}_{\crit} \sqrt{-g_2}} \lvert s - x^*_{\tau_{\crit}} \rvert^{1/2} + O(s - x^*_{\tau_{\crit}}),
  \end{split}
\end{equation}
which is similar to \eqref{eq:psilocalcrit}.
On the other hand, for $s > x^*_{\tau_{\crit}}$, since $g_1, \ldots, g_{2k - 1}$, $c_{2, 0}, \ldots, c_{2k - 2, 0}$ are all real, we have by \eqref{eq:Gtau}, \eqref{eq:Gmu0expansion1}, and \eqref{eq:c_2k-1_0_Case_II}
\begin{equation}
  \begin{split}
    \psi_{\tau_{\crit}}(s) = {}& - \frac{1}{\pi} \Im G_{\mu_0}(F_{\tau_{\crit}}(s)) \\
    = {}& -\frac{1}{\pi} \Im \left[- g_1 c_{2k - 1, 0} (s - x^*_{\tau_{\crit}})^{k-1/2} + O \left( (s - x^*_{\tau_{\crit}})^k\right)  \right] \\
    = {}& \frac{c_0^{2k+1}}{2(-\tau_{\crit} g_2)^{k+1/2}} \lvert s - x^*_{\tau_{\crit}} \rvert^{k-1/2} + O ( (s - x^*_{\tau_{\crit}})^k ),\qquad s \to \left(x^*_{\tau_{\crit}}\right)_+.
  \end{split}
\end{equation}
Hence we proved the $g_2 < 0$ part of Theorem \ref{thm:macrocritical}\ref{enu:thm:macrocritical_II}.

The proof for $g_2 > 0$ is nearly identical.

\paragraph{Case \ref{enu:thm:macrocritical_III}: Proof of Theorem \ref{thm:macrocritical}\ref{enu:thm:macrocritical_III}}

In Case \ref{enu:thm:macrocritical_III} we have by \eqref{eq:wplustauGmu0}, as $w\to x^*$,
\begin{equation}\label{eq:wcritIII}
  w + \tau_{\crit} G_{\mu_0}(w) = x^*_{\tau_{\crit}} - \tau_{\crit} g_3 (w-x^*)^3 + O((w-x^*)^{4}),
\end{equation}
where the coefficient $-\tau_{\crit} g_3$ of $(w - x^*)^3$ is negative by \eqref{eq:positive_g_3}. Using Lemma \ref{lem:first_Puiseux} with $f(w) = w + \tau_{\crit} G_{\mu_0}(w)$, $a = x^*$, $b = x^*_{\tau_{\crit}}$ and $c = -(\tau_{\crit}g_2)^{1/3}$, we have that for $z$ with $\lvert z - x^*_{\tau_{\crit}} \rvert < \epsilon$ and 
$ -\pi < \arg(z - x^*_{\tau_{\crit}}) < \pi$ where $\epsilon > 0$ is small enough, there exist analytic functions
\begin{equation}\label{eq:Fcrit_Case_III}
  \varphi_j(z) = x^* - \frac{e^{2j \pi i/3}}{(\tau_{\crit} g_2)^{1/3}}(z - x^*_{\tau_{\crit}})^{1/3} + \sum^{\infty}_{l = 2} c_{l, j} (z - x^*_{\tau_{\crit}})^{l/3}, \quad j = 0, 1, 2,
\end{equation}
such that $f(\varphi_j(z)) = z$. Since $F_{\tau_{\crit}}(z)$ is the inverse function of $f(z)$ from $\mathbb C^+$ to $\Omega_{\tau_{\crit}}$, it also satisfies $f(F_{\tau_{\crit}}(z)) = z$. It is not hard to check that in the common domain 
where $\lvert z - x^*_{\tau_{\crit}} \rvert < \epsilon$ and $\Im z > 0$, $F_{\tau_{\crit}}(z)$ agrees with $\varphi_2(z)$ but not with $\varphi_0(z)$ or $\varphi_1(z)$. Thus $F_{\tau_{\crit}}(z)$ can be extended to be defined for
$z$ in $\lvert z - x^*_{\tau_{\crit}} \rvert < \epsilon$ and $-\pi < \arg(z - x^*_{\tau_{\crit}}) < \pi$ by the formula of $\varphi_0(z)$ in 
\eqref{eq:Fcrit_Case_III}. 

Then by \eqref{eq:Gtau}, \eqref{eq:Gmu0expansion1}, and the Puiseux expansion of $F_{\tau_{\crit}}$, we obtain that as $s \to x^*_{\tau_{\crit}}$,
\begin{equation} \label{eq:psilocalcritfinalIII}
  \begin{split}
    \psi_{\tau_{\crit}}(s) = {}& \frac{1}{\pi} \Im \left[ g_0 + g_1 \frac{e^{i\pi/3}}{(\tau_{\crit} g_3)^{1/3}} (s - x^*_{\tau_{\crit}})^{1/3}  \right] +O ( (s - x^*_{\tau_{\crit}})^{2/3}) \\
    = {}& \frac{\sqrt{3}}{2\pi\tau_{\crit}^{4/3}g_3^{1/3}} \lvert s-x^*_{\tau_{\crit}} \rvert^{1/3} + O ( (s - x^*_{\tau_{\crit}})^{2/3}).
  \end{split}
\end{equation}
Here we used the fact that $g_0$ is real, $g_1 = \tau^{-1}_{\crit}$ and $g_3>0$. Thus we have proven Theorem \ref{thm:macrocritical}\ref{enu:thm:macrocritical_III}.

\subsection{Singular edge point}
We now turn to the case of a singular right endpoint. We assume that $x^*$ is the point in Assumption \ref{assumption:b}\ref{enu:assumption:b_2}, and make use of the analyticity of $h(x)$ at $x^*$. 

\subsubsection{Preliminaries} \label{subsubsec:preliminaries_edge}

Similar to \eqref{defn_g_j_interior} and \eqref{eq:Img2k} in Section \ref{subsec:interior_global}, we define
\begin{equation} \label{eq:defn_g_j_edge}
  g_j = \int \frac{d\mu_0(s)}{(s-x^*)^{j + 1}} =c_0^{2k+3/2} \int_{\supp(\mu_0)} \frac{|s-x^*|^{2k + 1/2}}{(s-x^*)^{j+1}} h(s) ds, \quad j = 0, 1, \ldots, 2k.
\end{equation}
We have
\begin{equation}
  \frac{1}{w-s} = \frac{1}{x^*-s} - \sum_{j=1}^{2k} \frac{(w-x^*)^j}{(s-x^*)^{j+1}} + \frac{(w-x^*)^{2k + 1}}{(w-s)(s-x^*)^{2k + 1}},
\end{equation}
and then similar to \eqref{eq:Gmu0expansion1}, we find 
\begin{equation} \label{eq:Gmu0expansion2} 
  G_{\mu_0}(w) = -\sum_{j=0}^{2k} g_j  ( w-x^*)^j + F(w) (w-x^*)^{2k + 1}, 
\end{equation}
with real coefficients $g_1, \ldots, g_{2k}$ and
\begin{equation}
  F(w) = \int \frac{d\mu_0(s)}{(w - s) (s-x^*)^{2k+1}} = c^{2k + 3/2}_0 \int_{\supp(\mu_0)} \frac{\sgn(s - x^*) h(s)}{(w - s) \lvert s - x^* \rvert^{1/2}} ds.
\end{equation}
The expression \eqref{defn_g_j_interior} for $g_j$ ($j = 0, \dotsc, 2k - 1$) continues to hold, and it generalizes to $j = 2k$. Also \eqref {def:g1} holds for $g_0$ and $g_1$.
 
Since $x^*$ is a right endpoint of an interval of the support of $\mu_0$ 
and $h$ is real analytic near $x^*$ with $h(x^*)=1$, we find after a straightforward calculation that 
\begin{equation}  \label{eq:psi0integral}
  F(w) = -\frac{c_0^{2k+3/2}\pi}{(w-x^*)^{1/2}} + O(1),\quad \text{as $w \to x^*$,}
\end{equation}
where $(w-x^*)^{1/2}$ has its branch cut on $(-\infty, x^*)$ and we take the branch $(w - x^*)^{1/2} > 0$ for $s \in (x^*, \infty)$. 
Although $F(w)$ is defined on $\mathbb C^+$, it can be extended to
the region  $\lvert w - x^* \rvert < \delta$ and 
$-\pi < \arg(w - x^*) < \pi$,  since for a small enough $\delta > 0$, $(x^*, x^* + \delta)$ is not in the support of $\mu_0$. On the branch cut $x \in (x^* - \delta, x^*)$, we have jump condition,
\begin{equation}
  \lim_{y \to 0_+} F(x + iy) - \lim_{y \to 0_-} F(x + iy) = -2\pi i c_0^{2k+3/2}\frac{h(x)}{\sqrt{x^* - x}}.
\end{equation}
It follows that that $F(w) + \pi c_0^{2k+3/2}h(w)/(w - x^*)^{1/2}$ can be extended to a holomorphic function in a neighborhood of $x^*$. Then by \eqref{eq:Gmu0expansion2}, $G_{\mu_0}(w)$ can be extended to 
$ \lvert x - x^* \rvert < \delta$, $-\pi < \arg(w - x^*) < \pi$, with the Puiseux expansion 
\begin{equation} \label{eq:expansion_G_mu_0_edge}
  G_{\mu_0}(w) = -\sum^{2k}_{j = 0} g_j( w-x^*)^j - \sum^{\infty}_{j = 2k + 1} \left( g_{j - 1/2} (w - x^*)^{j - 1/2} + g_j (w - x^*)^j \right),
\end{equation}
where all the coefficients $g_j$ and $g_{j - 1/2}$ are real, and by the assumption that $h(x^*) = 1$,
\begin{equation} \label{eq:value_g_2k+1/2_edge}
  g_{2k + 1/2} = \pi c^{2k + 3/2}_0 h(x^*) = \pi c^{2k + 3/2}_0.
\end{equation}
                
Using \eqref{eq:expansion_G_mu_0_edge}, \eqref{def:xtau}, and \eqref{def:g1}, we obtain
\begin{multline} \label{eq:wplustauGmu02}
  w + \tau G_{\mu_0}(w) = x^*_{\tau} + \left(1 - \frac{\tau}{\tau_{\crit}}\right) (w-x^*) \\
  - \tau \sum^{2k}_{j = 2} g_j( w-x^*)^j - \tau \sum^{\infty}_{j = 2k + 1} \left( g_{j - 1/2} (w - x^*)^{j - 1/2} + g_j (w - x^*)^j \right).
\end{multline}

For the analysis in Sections \ref{subsubsec:edge_sub} and \ref{subsubsec:edge_crit}, we define for $\lvert u \rvert < \delta^{1/2}$, where $\delta$ is the convergence radius in \eqref{eq:expansion_G_mu_0_edge},
\begin{equation} \label{eq:tilde_G_tau_expansion}
  \tilde{G}_{\tau}(u) = x^*_{\tau} + \left(1 - \frac{\tau}{\tau_{\crit}}\right) u^2 - \tau \sum^{2k}_{j = 2} g_j u^{2j} - \tau \sum^{\infty}_{j = 2k + 1} \left( g_{j - 1/2} u^{2j - 1} + g_j u^{2j} \right).
\end{equation}
The convergence of \eqref{eq:expansion_G_mu_0_edge} implies the convergence of \eqref{eq:tilde_G_tau_expansion}. Then $\tilde{G}_{\tau}(u)$ is an analytic function in a neighborhood around $0$, with
\begin{equation}
  w + \tau G_{\mu_0}(w) = \tilde{G}_{\tau}(\sqrt{w - x^*})
\end{equation}
for $\lvert w - x^* \rvert < \delta$ and $-\pi < \arg(w - x^*) < \pi$. Since $F_{\tau}(z): \mathbb C_+ \to \Omega_{\tau}$ is the inverse mapping of $w + \tau G_{\mu_0}(w)$, we define the function
\begin{equation} \label{eq:defn_tilde_F_edge}
  \tilde{F}_{\tau}(z) = (F_{\tau}(z) - x^*)^{1/2},
\end{equation}
and have that $\tilde{F}_{\tau}(z)$, from $\mathbb C^+$ to $\{ u \in \mathbb C^+ \mid u^2 + x^2 \in \Omega_{\tau} \}$, is the inverse mapping of $\tilde{G}_{\tau}$.


\subsubsection{Proof of Theorem \ref{thm:mainmacro}\ref{enu:thm:mainmacro_2}} \label{subsubsec:edge_sub}

Using Lemma \ref{lem:first_Puiseux} with $f(w) = \tilde{G}_{\tau}(w)$, $a = 0$, $b = x^*_{\tau}$, and $c = (1 - \tau/\tau_{\crit})^{1/2}$, we have that in the region  $\lvert z - x^*_{\tau} \rvert < \epsilon$, $\pi < \arg(z - x^*{\tau}) < \pi$ where $\epsilon > 0$ is small enough, there exist analytic functions
\begin{equation} \label{eq:tilde_F_sub}
  \varphi_j(z) = \frac{(-1)^j}{\sqrt{1 - \tau/\tau_{\crit}}} (z - x^*_{\tau})^{1/2} + \sum^{\infty}_{l = 2} c_{l, j} (z - x^*_{\tau})^{l/2}, \quad j = 0, 1,
\end{equation}
such that $\tilde{G}_{\tau}(\varphi_j(z)) = z$. Since $\tilde{F}_{\tau}(z)$, from $\mathbb C^+$ to $\{ u \in \mathbb C^+ \mid u^2 + x^2 \in \Omega_{\tau} \}$, is the inverse mapping of $\tilde{G}_{\tau}$, it is easily checked that in the common domain 
where $\lvert z - x^*_{\tau} \rvert < \epsilon$ and $\Im z > 0$,  $\tilde{F}_{\tau}(z)$ agrees with $\varphi_0(z)$ but not $\varphi_1(z)$. Thus $\tilde{F}_{\tau}(z)$ can be extended to be defined in $\lvert z - x^*_{\tau} \rvert < \epsilon$ with $ \pi < \arg(z - x^*_{\tau}) < \pi$ by the formula of $\varphi_0(z)$ in \eqref{eq:tilde_F_sub}, and we assume this extension below. 

From the coefficients $g_j$ (and $g_{j - 1/2}$) of $\tilde{G}_{\tau}$ in \eqref{eq:tilde_G_tau_expansion} we can derive the coefficients $c_{l, 0}$ of $\varphi_0(z)$ in \eqref{eq:tilde_F_sub}. Since all coefficients in the expansion of $\tilde{G}_{\tau}$ are real, we have that all $c_{l, 0}$ are real. Since the coefficients of the $u^{2j - 1}$ terms vanish for $j = 1, \dotsc, k$, we have that $c_{2l, 0} = 0$ for $l = 1, \dotsc, 2k - 1$. Then from the coefficient $=\tau g_{2k + 1/2} = -\pi \tau c^{2k + 3/2}_0$ of $u^{4k + 1}$ we have
\begin{equation}
  0 = \left( 1 - \frac{\tau}{\tau_{\crit}} \right) \frac{2}{\sqrt{1 - \tau/\tau_{\crit}}} c_{4k + 1, 0} - \tau g_{2k + 1/2} \left( \frac{1}{\sqrt{1 - \tau/\tau_{\crit}}} \right)^{4k + 1},
\end{equation}
and
\begin{equation} \label{eq:c_4k_0_sub}
  c_{4k, 0} = \frac{\pi \tau c^{2k + 3/2}_0}{2(1 - \tau/\tau_{\crit})^{2k + 1}}.
\end{equation}
Taking the square of the expansion \eqref{eq:tilde_F_sub} of $\varphi_0(z)$ and the relation \eqref{eq:defn_tilde_F_edge}, we can extend $F_{\tau}(z)$ to 
$\lvert z - x^*_{\tau} \rvert < \epsilon$, $\pi < \arg(z - x^*_{\tau}) < \pi$, 
and have its Puiseux expansion
\begin{equation} \label{eq:Ftauexpansion2} 
  F_{\tau}(z) = x^* + \sum^{2k}_{l = 1} f_l (z - x^*_{\tau})^l + \sum^{\infty}_{l = 2k + 1} \left( f_{l - 1/2} (z - x^*_{\tau})^{l - 1/2} + f_l (z - x^*_{\tau})^l \right).
\end{equation}
Then we have that all the coefficients of $F_{\tau}(z)$ in \eqref{eq:Ftauexpansion2} are real,
\begin{equation} \label{eq:f10}
	f_1  = \frac{\tau_{\crit}}{\tau_{\crit} - \tau}  > 0,
	\end{equation}
the same as \eqref{def:f1}, and then by \eqref{eq:c_4k_0_sub},
\begin{equation} \label{eq:fkappa}	
  f_{2k + 1/2} = \frac{2}{\sqrt{1 - \tau/\tau_{\crit}}} c_{4k, 0} = \tau \pi c_0^{2k+3/2} \left( \frac{\tau_{\crit}}{\tau_{\crit}-\tau}\right)^{2k + 3/2}.
\end{equation}

We now find, as $s \to \left(x_{\tau}^*\right)_-$,
\begin{equation} \label{eq:G_tau_expansion_edge_sub}
  \begin{split}
    \psi_{\tau}(s) = {}& -\frac{1}{\pi} \Im G_{\tau}(s) = -\frac{1}{\pi} \Im G_{\mu_0}(F_{\tau}(s)) \\
    = {}& \frac{1}{\pi} \sum_{j=1}^{2k} g_j
    \Im \left[ \sum^{2k}_{l = 1} f_l (s - x^*_{\tau})^l + \sum^{\infty}_{l = 2k + 1} (f_{l - 1/2} (s - x^*_{\tau})^{l - 1/2} + f_l (s - x^*_{\tau})^l) \right]^j \\
    & + \frac{1}{\pi} g_{2k + 1/2} \Im \left[f_1 (s - x_{\tau}^*) + O \left( (s - x^*)^2 \right) \right]^{2k+1/2} +
    O( \lvert s - x_{\tau}^* \rvert^{2k + 1}) \\
    = {}& \frac{1}{\pi} (g_1 f_{2k + 1/2} + g_{2k + 1/2} f^{2k + 1/2}_1) \lvert s - x^*_{\tau} \rvert^{2k + 1/2} + O( \lvert s - x_{\tau}^* \rvert^{2k + 1}) \\
    = {}& c^{2k + 3/2}_0 \left( \frac{\tau_{\crit}}{\tau_{\crit} - \tau} \right)^{2k + 3/2} \lvert s - x^*_{\tau} \rvert^{2k + 1/2} + O( \lvert s - x_{\tau}^* \rvert^{2k + 1}).
  \end{split}
\end{equation}
Here we use that the Puiseux expansion of $G_0(F_{\tau}(z))$ at $x^*_{\tau}$ has the form $a_0 + a_1(z - x^*_{\tau}) + a_2 (z - x^*_{\tau})^2 + \dotsb + a_{2k} (z - x^*_{\tau})^{2k} + a_{2k + 1} (z - x^*_{\tau})^{2k + 1} + O((z - x^*_{\tau})^{2k + 1})$, and all the coefficients $a_0, a_1, \dotsc, a_{2k}, a_{2k + 1/2}$ are  real. So we only need to compute the value of $a_{2k + 1/2}$, as in \eqref{eq:G_tau_expansion_edge_sub}. This gives us \eqref{eq:propsingularedge} by \eqref{def:ctau}.

On the other hand, since all the coefficients of the Puiseux expansion \eqref{eq:expansion_G_mu_0_edge} for $G_{\mu_0}(w)$ are real and all the coefficients of the Puiseux expansion \eqref{eq:Ftauexpansion2} for $F_{\tau}(z)$ are real, we have that with small enough $\epsilon$, $G_{\mu_{\tau}}(s) = G_{\mu_0}(F_{\tau}(s))$ is real valued for $s \in (x_{\tau}^*, x_{\tau}^* + \epsilon)$, and then $\psi_{\tau}(s) = 0$, for those $s$. Thus we complete the proof of Theorem \ref{thm:mainmacro}\ref{enu:thm:mainmacro_2}.

\subsubsection{Proof of Theorem \ref{thm:macrocritical}\ref{enu:thm:macrocritical_IV}, \ref{enu:thm:macrocritical_V}} \label{subsubsec:edge_crit}

If $\tau=\tau_{\crit}$, the coefficient of $(w - x^*)$ in \eqref{eq:wplustauGmu02} vanishes. We still define $\tilde{G}_{\tau}(u)$ by \eqref{eq:tilde_G_tau_expansion} with $\tau = \tau_{\crit}$, such that $g_2, \dotsc, g_{2k}, g_{2k + 1/2}, \dotsc$ are real coefficients, while the coefficient for $u^2$ vanishes.

We let $(-g_2)^{1/4}$ and $\sqrt{-g_2}$ be positive if $g_2 < 0$, and let $(-g_2)^{1/4} = e^{-\pi i/4} g^{-1/4}_2$ and $\sqrt{-g_2} = -i \sqrt{g_2}$ if $g_2 > 0$. Using Lemma \ref{lem:first_Puiseux} with $f(w) = \tilde{G}_{\tau}(w)$, $a = 0$, $b = x^*_{\tau}$, and $c = \tau^{1/4}_{\crit} (-g_2)^{1/4}$, we have analogous to \eqref{eq:tilde_F_sub}, in a region 
$\lvert z - x^*_{\tau_{\crit}} \rvert < \epsilon$, $\pi < \arg(z - x^*_{\tau_{\crit}}) < \pi$ where $\epsilon > 0$ is small enough, there exist analytic functions
\begin{equation} \label{eq:tilde_F_crit}
  \varphi_j(z) = \frac{e^{j\pi i/2}}{\tau^{1/4}_{\crit} (-g_2)^{1/4}} (z - x^*_{\tau_{\crit}})^{1/4} + \sum^{\infty}_{l = 2} c_{l, j} (z - x^*_{\tau_{\crit}})^{l/4}, \quad j = 0, 1, 2, 3,
\end{equation}
such that $f(\varphi_j(z)) = z$. 

Since function $\tilde{F}_{\tau_{\crit}}(z)$ defined in \eqref{eq:defn_tilde_F_edge} with $\tau = \tau_{\crit}$ is the inverse mapping of $\tilde{G}_{\mu_{\tau}}$ from $\mathbb C^+$ to $\{ u \in \mathbb C^+ \mid u^2 + x^2 \in \Omega_{\tau_{\crit}} \}$, it is not hard to check that in the common domain $ \lvert z - x^*_{\tau_{\crit}} \rvert < \epsilon$, $\Im z > 0$, $\tilde{F}_{\tau_{\crit}}(z)$ agrees with $\varphi_0(z)$ but not $\varphi_1(z), \varphi_2(z), \varphi_3(z)$. Thus $\tilde{F}_{\tau_{\crit}}(z)$ can be extended to be defined in $\lvert z - x^*_{\tau_{\crit}} \rvert < \epsilon$, 
$\pi < \arg(z - x^*{\tau_{\crit}}) < \pi$ where $\epsilon > 0$ is small enough by the formula of $\varphi_0(z)$ in \eqref{eq:tilde_F_crit}, and we assume the extension below. We consider \textbf{Case \ref{enu:thm:macrocritical_IV}} that $g_2 < 0$ and \textbf{Case \ref{enu:thm:macrocritical_V}} that $g_2 > 0$ separately, and omit the discussion of the more delicate $g_2 = 0$ case.

\paragraph{Case \ref{enu:thm:macrocritical_IV}: Proof of Theorem \ref{thm:macrocritical}\ref{enu:thm:macrocritical_IV}}
We assume $g_2 < 0$.
Similar to the case when $\tau < \tau_{\crit}$, we can compute the coefficients $c_{l, 0}$ for $\tilde{F}_{\tau_{\crit}}(z) = \varphi_0(z)$ in \eqref{eq:tilde_F_crit} by the coefficients in \eqref{eq:tilde_G_tau_expansion}. More precisely, if $k > 1$, then for $j = 1, \dotsc, 2k - 2$, $c_{2j, 0}$ and $c_{2j + 1, 0}$ depend on $g_2, g_3, \dotsc, g_{j + 2}$, otherwise $c_{l, 0}$ depends on all $g_j$ with $j \leq (l + 3)/2$. Since all the coefficients $g_j$ and $g_{j - 1/2}$ in \eqref{eq:tilde_G_tau_expansion} are real, we have that all $c_{l, 0}$ are real, and $c_{2j, 0} = 0$ for $j = 1, 2, \dotsc, 2k - 2$ if $k > 1$. Thus similar to \eqref{eq:Ftauexpansion2}, we can extend $F_{\tau_{\crit}}(z)$ to $ \lvert z - x^*_{\tau_{\crit}} \rvert < \epsilon$, $\pi < \arg(z - x^*{\tau_{\crit}}) < \pi$ with a Puiseux expansion
\begin{multline}\label{eq:Fcrit2}
  F_{\tau_{\crit}}(z) = x^* + \frac{1}{\sqrt{\tau_{\crit}} \sqrt{-g_2}} (z - x^*_{\tau_{\crit}})^{1/2} + \sum^{2k - 1}_{l = 2} f_l (z - x^*_{\tau_{\crit}})^{l/2} \\
  + \sum^{\infty}_{l = 2k} \left( f_{l - 1/2}(z - x^*_{\tau_{\crit}})^{l/2 - 1/4} + f_l(z - x^*_{\tau_{\crit}})^{l/2} \right),
\end{multline}
and all the coefficients $f_l$ and $f_{l - 1/2}$ are real. We now find, 
as $s\to\left(x_{\tau_{\crit}}^*\right)_-$,
\begin{equation}\label{eq:psilocalcrit2}
  \begin{aligned}
    \psi_{\tau_{\crit}}(s) = {}& -\frac{1}{\pi} \Im G_{\tau_{\crit}}(s) = \frac{1}{\pi} \frac{g_1}{\sqrt{\tau_{\crit}} \sqrt{-g_2}} \lvert s - x^*_{\tau_{\crit}} \rvert^{1/2} + O \left( \lvert s - x^*_{\tau_{\crit}} \rvert^{\alpha} \right) & & \\
    = {}& \frac{1}{\pi \tau^{3/2}_{\crit} \sqrt{-g_2}} \lvert s - x^*_{\tau_{\crit}} \rvert^{1/2} + O \left( \lvert s - x^*_{\tau_{\crit}} \rvert^{\alpha} \right), \quad \text{where} \quad \alpha =
    \begin{cases}
      3/4 & \text{if $k = 1$}, \\
      1 & \text{otherwise}.
    \end{cases}
  \end{aligned}
\end{equation}
On the other hand, since all the coefficients of the Puiseux expansion \eqref{eq:expansion_G_mu_0_edge} for $G_{\mu_0}(w)$ are real and all the coefficients of the Puiseux expansion \eqref{eq:Fcrit2} for $F_{\tau_{\crit}}(z)$ are real, we have that with small enough $\epsilon$, $G_{\mu_{\tau_{\crit}}}(s) = G_{\mu_0}(F_{\tau_{\crit}}(s))$ is real valued for $s \in (x_{\tau_{\crit}}^*, x_{\tau_{\crit}}^* + \epsilon)$, and then $\psi_{\tau_{\crit}}(s) = 0$ for 
those $s$. Thus we complete the proof of Theorem \ref{thm:macrocritical}\ref{enu:thm:macrocritical_IV}.

\paragraph{Case \ref{enu:thm:macrocritical_V}: Proof of Theorem \ref{thm:macrocritical}\ref{enu:thm:macrocritical_V}}

We assume $g_2 > 0$.
The computation of coefficients $c_{l, 0}$ for $\tilde{F}_{\tau_{\crit}}(z) = \varphi_0(z)$ in \eqref{eq:tilde_F_crit} is the same as in the $g_2 < 0$ case. We have that if $k > 0$, then for $j = 1, 2, \dotsc, 2k - 2$, $c_{2j, 0} = 0$, $c_{2j + 1, 0}$ is a real number times $e^{(2j + 1)\pi i/4}$.

Then coefficient $c_{4k - 2, 0}$ does not vanish. Using the Puiseux expansions for $\tilde{F}_{\tau_{\crit}}(z) = \varphi_0(z)$ and $\tilde{G}_{\tau_{\crit}}(u)$ to compute $\tilde{G}_{\tau_{\crit}}(\tilde{F}_{\tau_{\crit}}(z)) = z$, and comparing the coefficient of $(z - x^*_{\tau_{\crit}})^{k + 1/4}$, we have
\begin{equation}
  \begin{split}
    0 = {}& -\tau_{\crit} g_2 \left[ 4 \left( \frac{1}{\tau^{1/4}_{\crit} (-g_2)^{1/4}} \right)^3 c_{4k - 2, 0} \right] - \tau_{\crit} g_{2k + 1/2} \left( \frac{1}{\tau^{1/4}_{\crit} (-g_2)^{1/4}} \right)^{4k + 1} \\
    = {}& -\tau_{\crit} g_2 \frac{4 e^{3\pi i/4}}{\tau^{3/4}_{\crit} g^{3/4}_2} c_{4k - 2, 0} - \tau_{\crit} g_{2k + 1/2} \frac{(-1)^k e^{\pi i/4}}{\tau^{k + 1/4}_{\crit} g^{k + 1/4}_2}.
  \end{split}
\end{equation}
Then by \eqref{eq:value_g_2k+1/2_edge}, we have
\begin{equation}
  c_{4k - 2, 0} = \frac{\pi i (-1)^k c^{2k + 3/2}_0}{4 \tau^{k - 1/2}_{\crit} g^{k + 1/2}_2}.
\end{equation}
Then by \eqref{eq:defn_tilde_F_edge}, we have that $F_{\tau_{\crit}}(z)$ can be extended to $\lvert z - x^*_{\tau_{\crit}} \rvert < \epsilon$, $\pi < \arg(z - x^*_{\tau_{\crit}}) < \pi$, and have its Puiseux expansion
\begin{equation} \label{eq:Ftauexpansion2_crit} 
  F_{\tau}(z) = x^* + \sum^{2k - 1}_{l = 1} f_l (z - x^*_{\tau})^{l/2} + \sum^{\infty}_{l = 2k} \left( f_{l - 1/2} (z - x^*_{\tau})^{l/2 - 1/4} + f_l (z - x^*_{\tau})^{l/2} \right).
\end{equation}
We have that for $l = 1, \dotsc, 2k - 1$, $f_l$ is real if $l$ is even, and $f_l$ is purely imaginary if $l$ is odd. Also we have
\begin{equation}
  f_1 = \frac{i}{\tau^{1/2}_{\crit} g^{1/2}_2}, \quad f_{2k - 1/2} = \frac{2}{\tau^{1/4}_{\crit} (-g)^{1/4}} c_{4k - 2, 0} = \frac{\pi e^{3\pi i/4} (-1)^k c^{2k + 3/2}_0}{2 \tau^{k - 1/4}_{\crit} g^{k + 3/4}_2}.
\end{equation}
As $s \to\left(x_{\tau_{\crit}}^*\right)_+$, we have by \eqref{eq:Gtau}
\begin{equation} \label{eq:edge_last}
  \begin{aligned}
    \psi_{\tau_{\crit}}(s) = {}& -\frac{1}{\pi} \Im G_{\tau_{\crit}}(s) = \frac{1}{\pi} \frac{g_1}{\sqrt{\tau_{\crit}} \sqrt{g_2}} \lvert s - x^*_{\tau_{\crit}} \rvert^{1/2} + O \left( \lvert s - x^*_{\tau_{\crit}} \rvert^{\alpha} \right) && \\
    = {}& \frac{1}{\pi \tau^{3/2}_{\crit} \sqrt{g_2}} \lvert s - x^*_{\tau_{\crit}} \rvert^{1/2} + O \left( \lvert s - x^*_{\tau_{\crit}} \rvert^{\alpha} \right), \quad \text{where} \quad \alpha =
    \begin{cases}
      3/4 & \text{if $k = 1$}, \\
      1 & \text{otherwise}.
    \end{cases}
  \end{aligned}
\end{equation}
As $s \to\left(x_{\tau_{\crit}}^*\right)_-$, we have similarly ($\alpha$ is defined in \eqref{eq:edge_last})
\begin{equation}
  \begin{split}
    \psi_{\tau_{\crit}}(s) = {}& -\frac{1}{\pi} \Im G_{\tau_{\crit}}(s) \\
    = {}& \frac{1}{\pi} \sum_{j=1}^{2k} g_j
    \Im \left[ \sum^{2k - 1}_{l = 1} f_l (s - x^*_{\tau})^{l/2} + \sum^{\infty}_{l = 2k} (f_{l - 1/2} (s - x^*_{\tau})^{l/2 - 1/4} + f_l (s - x^*_{\tau})^{l/2}) \right]^j \\
    & + \frac{1}{\pi} g_{2k + 1/2} \Im \left[f_1 (s - x_{\tau}^*)^{1/2} + O \left( (s - x^*)^{\alpha} \right) \right]^{2k + 1/2} +
    O( \lvert s - x_{\tau}^* \rvert^{2k + 1/2}) \\
    = {}& \frac{1}{\pi} g_1 \Im(f_{2k - 1/2} (s - x^*_{\tau})^{k - 1/4}) + O( \lvert s - x_{\tau}^* \rvert^k) \\
    = {}& \frac{c^{2k + 3/2}_0}{2(\tau_{\crit} g_2)^{k + 3/4}} \lvert s - x^*_{\tau} \rvert^{k - 1/4} + O( \lvert s - x_{\tau}^* \rvert^k).
  \end{split}
\end{equation}
Thus we proved Theorem \ref{thm:macrocritical}\ref{enu:thm:macrocritical_V}.

\section{Proof of Lemma \ref{lemma12}}\label{sec:prooflemma}

\begin{proof}
The equilibrium measure $\mu_0$ satisfies for some constant $\ell$,
\begin{equation} \label{eq:ELequation}
	2 \int \log |x-s| d\mu_0(s) = V(x) + \ell \qquad \text{on the support of $\mu_0$} 
	\end{equation}
which is one of the Euler-Lagrange variational conditions for the minimization problem
for \eqref{eq:logenergy}, see \cite{Deift99, Saff-Totik97}.
If we define the $g$-function as
\begin{equation} \label{def:gs} 
g(z) = \int \log(z-s) d\mu_0(s) ,
\end{equation}
this means that $g_+ + g_- = V + \ell$ on the support of $\mu_0$.
It follows that $V-g+ \ell$ is the analytic continuation of $g$ across any interval
in the support of $\mu_0$ (recall that $V$ is real analytic). 

For every $s$ that is interior to $\supp(\mu_0)$ we have
convergent power series
\[ g(z) = \sum_{l=0}^{\infty} \frac{g_\pm^{(l)}(s)}{l!} (z-s)^l, \qquad \pm \Im z > 0 \]
for $z$ near $s$, with
\begin{equation} \label{eq:g+g-sum} 
	g_+^{(l)}(s) + g_-^{(l)}(s) = V^{(l)}(s) \qquad \text{for } l \geq 1. 
	\end{equation}
Taking derivatives of \eqref{def:gs} we get 
\begin{equation} \label{eq:glz} 
	g^{(l)}(z) = - (l-1)! \int \frac{d\mu_0(s)}{(s-z)^l}, \qquad l \geq 1. 
	\end{equation}
Since the density of $\mu_0$ vanishes with an exponent $2k$ or $2k + 1/2$ at $x^*$,
 we see by taking $z \to x^*$ in \eqref{eq:glz}
\begin{equation} \label{eq:g+isg-} g^{(l)}_+(x^*) = g^{(l)}_-(x^*) = - (l-1)! \int \frac{d\mu_0(s)}{(s-x^*)^l},
	\qquad 1 \leq l \leq 2k. \end{equation}
We take the limit $s \to x^*$ in  \eqref{eq:g+g-sum} with $l \geq 2k$ and
combining this with \eqref{eq:g+isg-}, we 
obtain \eqref{eq:Vderivative}. 
\end{proof}

For future reference we note that the above proof shows that
$g_+^{(l)}(x^*) = \frac{1}{2} V^{(l)}(x^*)$ for $1 \leq l \leq 2k$ and 
\begin{equation} \label{eq:gexpansion} 
	g(z) = g_\pm(x^*) + \frac{1}{2} \sum_{l=1}^{2k} \frac{V^{(l)}(x^*)}{l!} (z-x^*)^l + O\left((z-x^*)^{2k+1}\right) 
	\end{equation}
	as $z \to x^*$ from $\mathbb C^{\pm}$, uniformly for $z$ in a full neighborhood of $x^*$.

\section{Proof of Theorem \ref{thm:multi}\ref{enu:thm:multi_a}} \label{section: proof main result (a)}

This section is devoted to the proof of part \ref{enu:thm:multi_a} of Theorem \ref{thm:multi}.

\subsection{Splitting of the double integral} \label{section:splitting}
We start from the double integral representation 
\eqref{eq:mult_kernel_tilde} of $\widetilde{K}^{\X}_n$.
Taking a fixed but large enough $R > 0$, we decompose the double integral  \eqref{eq:mult_kernel_tilde} as follows. Let
\begin{equation}
  \begin{split}
   \Sigma_{\loc} & = {} \text{line segment from $x^* - iR n^{-\gamma}$ to $x^* + iR n^{-\gamma}$}, \\
   \Sigma_{\rest} & = {}  (x^* + i \mathbb R) \setminus \Sigma_{\loc},
  \end{split}
\end{equation}
and denote the integrand in \eqref{eq:mult_kernel_tilde} for $\widetilde{K}^{\X}_n(x, y; t, t')$ by $f(z, w; x, y; t, t')$. Define
\begin{align}
 \widetilde{K}_{n,\loc}^{\X}(x, y; t, t')  & = {} \frac{n}{2\pi i \sqrt{tt'}} \int_{\Sigma_{\loc}} dz   \int_{x^*-R n^{-\gamma} }^{x^*+Rn^{-\gamma}} dw f(z, w; x, y; t, t'), \label{def:KnXloc} \\
 \widetilde{K}_{n, 1}^{\X}(x, y; t, t')  & = {} \frac{n}{2\pi i \sqrt{tt'}} \int_{\Sigma_{\rest}} dz   \int_{\mathbb R \setminus (x^*-R n^{-\gamma}, x^*+Rn^{-\gamma})} dw f(z, w; x, y; t, t'), 
 	\label{eq:KnX1} \\
 \widetilde{K}_{n, 2}^{\X}(x, y; t, t')  & = {} \frac{n}{2\pi i \sqrt{tt'}} \int_{\Sigma_{\loc}} dz   \int_{\mathbb R \setminus (x^*-R n^{-\gamma}, x^*+Rn^{-\gamma})} dw f(z, w; x, y; t, t'), 
 	\label{eq:KnX2} \\
 \widetilde{K}_{n, 3}^{\X}(x, y; t, t')  & = {} \frac{n}{2\pi i \sqrt{tt'}} \int_{\Sigma_{\rest}} dz   \int_{x^*-R n^{-\gamma}}^{x^*+Rn^{-\gamma}} dw f(z, w; x, y; t, t').
 	\label{eq:KnX3}
\end{align} 
Then we have the decomposition
\begin{equation} \label{eq:KnXsplit} 
  \widetilde{K}^{\X}_n \left(x, y; t, t' \right) = \widetilde{K}_{n, \loc}^{\X} \left(x, y; t, t' \right) + \widetilde{K}_{n, \rest}^{\X} \left( x, y; t, t' \right),
\end{equation}
where
\begin{equation} \label{eq:KnXrest}
  \widetilde{K}_{n, \rest}^{\X} \left( x, y; t, t' \right) = \sum^3_{j = 1} \widetilde{K}_{n, j}^{\X}(x, y; t, t').
\end{equation}

We prove the following result:
\begin{proposition} \label{prop:main}
  Fix a compact set $K \subset \mathbb R$ and let $u,v\in K$. Let $\widehat c_t$ and $\widehat x_t^*$ be as in \eqref{eq:constants multi}
  and let $R > 1$ be sufficiently large. Then there exists $\widehat{H}_n(u;t)$ such that the following hold.
  \begin{enumerate}[label=(\alph*)]
  \item \label{enu:prop:main_a}
    We have
  \begin{equation} \label{eq:KnXloclimit} 
    \lim_{n \to \infty} \frac{e^{-\widehat H_n(u; t) + \widehat H_n(v; t')}}{\sqrt{\widehat c_{t} \widehat c_{t'}} \,n^{\gamma}}  \widetilde{K}_{n,\loc}^{\X} \left( \widehat x_t^*+\frac{u}{\widehat c_t n^\gamma}, \widehat x_{t'}^*+\frac{v}{\widehat c_{t'} n^\gamma}; t, t' \right) = \K_{\kappa}(u, v),
  \end{equation}
  uniformly for $u, v \in K$, and
\item \label{enu:prop:main_b}
  there exists a constant $c=c_R>0$ such that
  \begin{equation} \label{eq:KnXrestestimate}
    e^{- \widehat H_n(u; t) + \widehat H_n(v; t')} \left\lvert \widetilde{K}_{n, \rest}^{\X} \left( \widehat x_t^*+\frac{u}{\widehat c_t n^\gamma}, \widehat x_{t'}^*+\frac{v}{\widehat c_{t'} n^\gamma}; t, t' \right)   \right\rvert \leq e^{-c n^{1-2\gamma}}
  \end{equation}
  for $n$ sufficiently large.
  \end{enumerate}
\end{proposition}
Recall that $\gamma \leq 1/3$, so that $1-2\gamma > 0$. 
Then it is clear that \eqref{eq:Knlimit_multi} and thus Theorem \ref{thm:multi}\ref{enu:thm:multi_a} immediately follow from Proposition \ref{prop:main}.

\subsection{Proof of Proposition \ref{prop:main}\ref{enu:prop:main_a}} \label{subsec:kernel_asy_loc} 

We write
\[x=\widehat x_{t}^*+\frac{u}{\widehat c_{t} n^\gamma},\qquad y=\widehat x_{t'}^*+\frac{v}{\widehat c_{t'} n^\gamma}. \]
After a change of variables $\xi = c_0 n^{\gamma} (z - x^*)$, $\eta = c_0 n^{\gamma} (w - x^*)$, we obtain from  \eqref{def:KnXloc}, \eqref{eq:mult_kernel_tilde}
and \eqref{eq:xtauV} that
\begin{multline} \label{eq:KnXscaled1}
  \widetilde{K}_{n, \loc}^{\X} \left( x,y; t, t' \right) \\= 
  \frac{n^{1-\gamma}}{2\pi i \sqrt{t t'} c_0 } \int_{-i c_0R}^{i c_0R}  d\xi  \int_{-c_0R}^{c_0R} d\eta 
  \frac{1}{c_0n^{\gamma}} 
  K^M_n \left(x^* + \frac{\xi}{c_0n^{\gamma}}, x^* + \frac{\eta}{c_0n^{\gamma}} \right) 
  \frac{e^{\Phi_n(\xi,u; t)}}{e^{\Phi_n(\eta,v; t')}},
\end{multline}
where 
\begin{equation} \label{def:Fnsx} 
  \Phi_n(\xi,u; t) = \frac{n}{2} V\left(x^* + \frac{\xi}{c_0n^{\gamma}}\right) + \frac{n }{2t(1-t)} \left( t\frac{V'(x^*)}{2} + \frac{u}{\widehat c_t n^{\gamma}}-\frac{(1-t)\xi}{c_0n^\gamma} \right)^2. 
\end{equation}
The proof is based on a saddle point analysis of \eqref{eq:KnXscaled1} and \eqref{def:Fnsx}.  

Applying a Taylor expansion of $V$ at $x^*$ up to the second order in \eqref{def:Fnsx} and using \eqref{eq:constants multi}, we obtain
\begin{multline} \label{eq:Fnsx}
    \Phi_n(\xi,u; t) =  \frac{n}{2} V(x^*) + \frac{t n}{8(1-t)} V'(x^*)^2 + \frac{u n^{1-\gamma}}{2\widehat c_{t}(1-t)} V'(x^*) + \frac{u^2n^{1-2\gamma}}{4c_0 \widehat c_{t}(1-t)} V''(x^*) \\
    + \frac{n^{1-2\gamma}}{2 t c_0 \widehat c_{t}} (\xi - u)^2 +  n^{1 - 3\gamma} R_n(\xi),
\end{multline}
where the remainder $R_n$ is 
\begin{equation}  \label{eq:Rnxi}
	R_n(\xi) =  \frac{n^{3\gamma}}{2}\left(V\left(x^*+\frac{\xi}{c_0n^\gamma}\right)-V(x^*)
	-V'(x^*)\frac{\xi}{c_0n^\gamma}-V''(x^*)\frac{\xi^2}{2c_0^2n^{2\gamma}}\right).
\end{equation}
Then  $R_n$ is an analytic function in a neighorbood of $\xi=0$ that grows with $n$, with a Taylor expansion
\begin{equation} \label{eq:TaylorRn}
	R_n(\xi) = \frac{1}{2} \sum^{\infty}_{k = 3} \frac{V^{(k)}(x^*)}{k!} \frac{\xi^k}{c^k_0 n^{(k - 3)\gamma}} .
\end{equation}
In particular, 
\begin{equation} \label{eq:Rnlimit}
	R_n(\xi) = \frac{V'''(x^*)}{12} \frac{\xi^3}{c^3_0} + O(n^{-\gamma} \xi^4)
	\qquad \text{ as } n \to \infty,
\end{equation}
uniformly $\xi$ in compact subsets of $\mathbb C$.
The first four terms on the right of \eqref{eq:Fnsx} do not depend on $\xi$ and they will contribute to the gauge factor $\widehat{H}_n$, see \eqref{eq:Hnhat} below. 

Because of \eqref{eq:Fnsx} the saddle point equation $\frac{\partial \Phi_n}{\partial \xi} = 0$ has a solution at a value close to $u$. We denote it by $s_n(u,t)$ and it is such that 
\begin{equation} \label{eq:snut} 
	\frac{n^{1-2\gamma}}{tc_0 \widehat{c}_t} (s_n(u,t) -u) + n^{1-3\gamma} R_n'(s_n(u,t)) = 0. 
	\end{equation}
From \eqref{eq:TaylorRn}, \eqref{eq:Rnlimit} and \eqref{eq:snut}, we find that
\begin{equation} \label{eq:snutlimit} 
	s_n(u,t) = u  - \frac{t \widehat{c}_t  V'''(x^*)}{4c_0 n^{\gamma}} u^2 + O(n^{-2\gamma} u^3) \qquad \text{ as } n \to \infty. 
	\end{equation}
Then we have
\begin{equation} \label{eq:Rnrewrite} 
	\frac{n^{1-2\gamma}}{2tc_0 \widehat{c}_t} (\xi-u)^2 + n^{1-3\gamma} R_n(\xi) =
	\frac{n^{1-2\gamma}}{2tc_0 \widehat{c}_t} (\xi-s_n(u,t))^2 + n^{1-3\gamma} \widehat{R}_n(\xi,u,t) \end{equation}
where
\begin{equation} \label{eq:Rnhat}
	\widehat{R}_n(\xi, u,t)
	= R_n(\xi) + \frac{n^{\gamma}}{tc_0 \widehat{c}_t} (s_n(u,t) - u)(\xi-u) 
	- \frac{n^{\gamma}}{2tc_0 \widehat{c}_t} (s_n(u,t) - u)^2.  \end{equation}

We define
\begin{equation} \label{eq:Hnhat}
	\widehat{H}_n(u,t) = 
	 \frac{t n}{8(1-t)} V'(x^*)^2 + \frac{u n^{1-\gamma}}{2\widehat c_{t}(1-t)} V'(x^*) + \frac{u^2n^{1-2\gamma}}{4c_0 \widehat c_{t}(1-t)} V''(x^*)	
	 + n^{1-3\gamma} \widehat{R}_n(s_n(u,t), u,t).
	 \end{equation}
Then by \eqref{eq:KnXscaled1}, \eqref{eq:Fnsx}, \eqref{eq:Rnrewrite}, \eqref{eq:Rnhat} and \eqref{eq:Hnhat} we obtain
\begin{multline} \label{eq:KnXscaled2}
  \frac{e^{-\widehat H_{n}(u; t)+ \widehat H_{n}(v; t')}}{\sqrt{\widehat c_{t} \widehat c_{t'}} n^{\gamma}} \tilde{K}_{n, \loc}^{\X} \left(x,y; t, t' \right) \\
  = \frac{n^{1-2\gamma}}{2\pi i \sqrt{t t'} c_0 \sqrt{\widehat{c}_t \widehat{c}_{t'}}}  \int_{-i c_0R}^{i c_0R}  d\xi 
  \int_{-c_0R}^{c_0R} d\eta  
  \frac{1}{c_0n^{\gamma}} K_n^M \left(x^* + \frac{\xi}{c_0 n^{\gamma}}, x^* + \frac{\eta}{c_0 n^{\gamma}} \right) \\
  \times  
   \exp \left( \frac{n^{1-2\gamma}}{2t c_0 \widehat c_{t}} \left( \xi - s_n(u,t) \right)^2  -  \frac{n^{1-2\gamma}}{2t' c_0 \widehat c_{t'}} \left( \eta - s_n(v,t') \right)^2 \right) \frac{e^{n^{1 - 3\gamma} \Psi_n(\xi, u; t)}}{e^{n^{1 - 3\gamma} \Psi_n(\eta, v; t')}},
\end{multline}
where
\begin{equation} \label{def:Psin} \Psi_n(\xi,u;t) = \widehat{R}_n(\xi,u,t) - \widehat{R}_n(s_n(u,t), u,t). 
\end{equation}
It is obvious that 
\begin{equation} \label{eq:Psinzero}
	\Psi_n(s_n(u,t), u;t) = 
	\left. \frac{\partial \Psi_n}{\partial\xi} (\xi, u; t) \right\rvert_{\xi = s_n(u,t)}  = 0
	\end{equation}
because of the choice of $s_n(u,t)$ as the saddle point of $\Phi_n$, see \eqref{eq:snut} and \eqref{eq:Rnhat}. In addition by \eqref{def:Psin}, \eqref{eq:Rnhat} and \eqref{eq:TaylorRn}, we have
\begin{equation} \label{eq:Psin2bound} 
	\frac{\partial^2 \Psi_n}{\partial \xi^2} = R_n''(\xi) = 
	\frac{V'''(x^*)}{4c_0^3} \xi + O(n^{-\gamma} \xi^2)\end{equation}
as $n \to \infty$, uniformly for $\xi$ in compacts. 

We assume $R$ is large enough, so that
\[ c_0 R > 3 \max(|u|, |v|). \]
Then, if $n$ is large enough, the interval $[-c_0R, c_0R]$ contains the saddle point
$s_n(v,t')$ for the $\eta$-integral in \eqref{eq:KnXscaled2},
see \eqref{eq:snutlimit}. In addition, by analyticity, the $\xi$-integral in \eqref{eq:KnXscaled2} can be deformed to a path of descent passing through the saddle $s_n(u,t)$.

It follows from \eqref{eq:Psinzero} and \eqref{eq:Psin2bound} that
there exists $C > 0$ such that
\[ |\Psi_n(\xi,u;t)| \leq C (\xi-s_n(u,t))^2,
	\quad |\Psi_n(\eta,v;t')| \leq C (\eta-s_n(v,t'))^2, \]
for all $\eta \in [-c_0R, c_0R]$ and all $\xi$ on the (new) path of integration.  
Furthermore, by Proposition \ref{prop:prop17} the rescaled  kernel
in \eqref{eq:KnXscaled2} tends to $\K_{\kappa}(\xi,\eta)$ as $n \to \infty$,
uniformly for $\xi$ and $\eta$ in compacts.
Then by a standard saddle point approximation,  
we conclude that \eqref{eq:KnXscaled2}  has the same limit
as $\K_{\kappa}(s_n(u,t), s_n(v,t'))$ as $n \to \infty$, 
which is $\K_{\kappa}(u,v)$ because of  \eqref{eq:snutlimit}.
This completes the proof of Proposition \ref{prop:main}\ref{enu:prop:main_a}.

\subsection{Proof of Proposition \ref{prop:main}\ref{enu:prop:main_b}} \label{subsec:kernel_rest} 

\subsubsection{Preliminaries: deformation of the $w$-contour}

We first take a closer look at the correlation kernel $K_n^M$
for the eigenvalues of the unitary invariant random matrix $M$, see \eqref{def:KnM} . Because of the Christoffel-Darboux
formula for orthogonal polynomials we can write
\[ K_n^M(x,y) = e^{-\frac{n}{2}(V(x) + V(y))} \frac{\kappa_{n-1,n}}{\kappa_{n,n}} 
	\frac{p_{n,n}(x) p_{n-1,n}(y) - p_{n-1,n}(x) p_{n,n}(y)}{x-y} \]
where $\kappa_{j,n}$ is the leading coefficient of $p_{j,n}$ for $j=n-1,n$. 
The kernel can be further expressed as
\begin{equation} \label{eq:KnMinYn} 
	K_n^M(x,y) = \frac{e^{-\frac{n}{2}(V(x)+V(y))}}{2\pi i(x-y)} 
	\left[ Y_n^{-1}(y) Y_n(x) \right]_{2,1}, 
	\end{equation}
	where
$Y_n : \mathbb C \setminus \mathbb R \to \mathbb C^{2 \times 2}$ is given by
\begin{equation} \label{def:Yn} 
Y_n(z) = \begin{pmatrix} \kappa_{n,n}^{-1} p_{n,n}(z) & 
\ds \frac{\kappa_{n,n}^{-1}}{2\pi i} \int_{-\infty}^{\infty} \frac{p_{n,n}(s) e^{-nV(s)}}{s-z} ds \\
- 2 \pi i \kappa_{n-1,n} p_{n-1,n}(z) & \ds -\kappa_{n-1,n} \int_{-\infty}^{\infty} \frac{p_{n-1,n}(s) e^{-nV(s)}}{s-z} ds
\end{pmatrix}. \end{equation} 
The right-hand side of \eqref{eq:KnMinYn} only involves entries in the first column of $Y_n$, which are polynomials and can thus be evaluated at real points $x$ and $y$ without ambiguity.
The matrix-valued function $Y_n$ is the solution of the Riemann--Hilbert problem for orthogonal polynomials
\cite{Deift99, Fokas-Its-Kitaev92}. 

An alternative expression of $K^M_n(x, y)$, and equivalently of $K^{\PE}_n(x, y)$, is derived in \cite[Formula (3.20)]{Claeys-Kuijlaars-Wang15} in the following way. As a consequence of the jump condition satisfied by $Y_n$ for $x\in\mathbb R$,
$Y_{n,+}(x) = Y_{n,-}(x) \begin{pmatrix} 1 & e^{-nV(x)} \\ 0 & 1 \end{pmatrix}$, where $Y_{n,\pm}(x)$ denote the boundary values from above ($+$) and from below ($-$), we have that
\[
Y_{n,-}(w) \begin{pmatrix} 1 \\ 0 \end{pmatrix}
= e^{nV(w)} \left( Y_{n,+}(w) - Y_{n,-}(w) \right)
\begin{pmatrix} 0 \\ 1 \end{pmatrix}
\]
for $w\in\mathbb R$.
Using this in \eqref{def:KnPE} and \eqref{eq:KnMinYn}, we obtain the formula
  \begin{align} \nonumber
    K_n^{\text{PE}}(z, w) = {}& K_n^M(z, w) e^{\frac{n}{2}(V(z) - V(w))} \\
    = {}& \frac{-1}{2\pi i(z - w)}
    \begin{pmatrix} 0 & 1 \\ \end{pmatrix}
    Y_n^{-1}(z) \left( Y_{n,+}(w) - Y_{n,-}(w) \right)
    \begin{pmatrix} 0 \\ 1 \end{pmatrix} \nonumber \\
    & = \frac{-1}{2\pi i(z - w)} \left(
        \left[ Y_n^{-1}(z) Y_{n,+}(w) \right]_{2,2}
         - \left[ Y_n^{-1}(z) Y_{n,-}(w) \right]_{2,2} \right),
    \label{eq:KnPEinY} 
  \end{align} 
which holds for $z\in\mathbb C$ and $w\in\mathbb R$.
	
Our strategy  is  to deform the contour for the $w$-integral in
\eqref{eq:KnX1} and \eqref{eq:KnX2} into the complex plane, and then utilize the asymptotic results 
from the Riemann--Hilbert analysis for \eqref{def:Yn} from \cite{Deift-Kriecherbauer-McLaughlin-Venakides-Zhou99}. 
The deformation of contours is somewhat similar to the ``opening of the lens'' step 
in the Deift--Zhou nonlinear steepest descent method for Riemann-Hilbert problems. It enables  
us to move certain integrals from the real line, where the integrands are oscillatory, 
to contours in the complex plane where the integrands are small as $n \to \infty$.

Let $\Gamma_j$, $j=1,2,3,4$, denote the four half rays from $x^* \pm R n^{-\gamma}$ that
make angles $\pm \pi/6$ with the positive or negative real axis, namely
\begin{equation} \label{def:Gammaj}
	\Gamma_j = \{ w \in \mathbb C \mid \arg [w - (x^* \pm Rn^{-\gamma})] = \theta_j \}, \qquad j=1,2,3,4,
	\end{equation}
with $\theta_1 = \pi/6$, $\theta_2 = 5 \pi/6$, $\theta_3 = -5 \pi/6$, $\theta_4 = -\pi/6$, and the $\pm$ sign is positive for $j = 1, 4$ and negative for $j = 2, 3$.
All contours $\Gamma_j$ are oriented from left to right. See Figure \ref{fig:Gamma_j} for an illustration.
\begin{figure}[htb]
  \centering
  \includegraphics{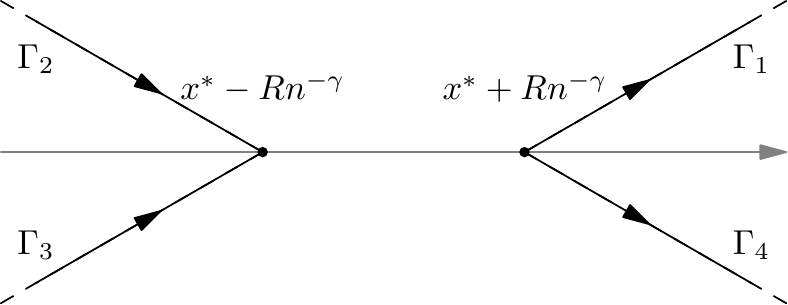}
  \caption{The shape and orientation of $\Gamma_j$ ($j = 1, 2, 3, 4$).}
  \label{fig:Gamma_j}
\end{figure}

We also put 
\begin{equation} \label{def:Gamma}
  \Gamma = (-\Gamma_1) \cup (-\Gamma_2) \cup \Gamma_3 \cup \Gamma_4,
\end{equation}
where the negative sign means the orientation of the contour is reversed.

Using \eqref{eq:KnPEinY}, for $z \in \mathbb C \setminus [x^* + Rn^{-\gamma}, +\infty)$, we have
\begin{multline} \label{eq:first_deformation_Gamma} 
  \int_{x^* + Rn^{-\gamma}}^{\infty} K_n^{\text{PE}}(z, w) e^{-\frac{n(y - (1-t')w)^2}{2t'(1-t')}} dw  \\
  = -\frac{1}{2\pi i} \int_{x^* + Rn^{-\gamma}}^{\infty}
  \frac{1}{z - w} \left[Y_n^{-1}(z) Y_{n,+}(w) \right]_{2,2}  e^{-\frac{n(y - (1-t')w)^2}{2t'(1-t')}} dw \\
  +\frac{1}{2\pi i} \int_{x^* + Rn^{-\gamma}}^{\infty}
  \frac{1}{z - w} \left[ Y_n^{-1}(z) Y_{n,-}(w) \right]_{2,2} 
  e^{-\frac{n(y - (1-t')w)^2}{2t'(1-t')}} dw,
\end{multline}
and by analyticity we are allowed to deform contours. Since the two entries of the second column of $Y_n(w)$ are $O(w^{-n})$ as $w \to \infty$ (this follows from the orthogonality relations), we can move the first integral 
to $\Gamma_1$ and the second integral to $\Gamma_4$. 
We similarly can consider the above integral with $(-\infty, x^* - Rn^{-\gamma})$ as the interval of
integration, and deform the integral to $\Gamma_2 \cup \Gamma_3$. Then integrating over $z \in \Sigma_{\rest}$ or $z \in \Sigma_{\loc}$, we have for $j=1, 2$, see \eqref{eq:KnX1}, \eqref{eq:KnX2}, and \eqref{eq:mult_kernel_tildePE}, 
\begin{equation} \label{def:KnXrestI} 
  \widetilde{K}^{\X}_{n, j}(x, y; t, t') = \frac{n}{(2\pi i)^2 \sqrt{tt'}} \int_{\Sigma_{*}} dz \int_{\Gamma} dw \frac{1}{z - w} \left[ Y_n^{-1}(z) Y_n(w) \right]_{2,2} e^{\frac{n(x - (1-t)z)^2}{2t(1-t)} -
  \frac{n (y - (1-t')w)^2}{2t'(1-t')}} 
\end{equation}
where $\Sigma_*=\Sigma_{\rest}$ for $j=1$ and $\Sigma_*=\Sigma_{\loc}$ for $j=2$.
Next, by \eqref{eq:KnX3} and \eqref{eq:KnPEinY}, we can write 
\begin{equation}
  \widetilde{K}^{\X}_{n, 3}(x, y; t, t') = 
  \widetilde{K}^{\X}_{n, 3+}(x, y; t, t') - \widetilde{K}^{\X}_{n, 3-}(x, y; t, t'),
\end{equation}
where
\begin{multline} \label{def:KnXrestI_3+-} 
  \widetilde{K}^{\X}_{n, 3\pm}(x, y; t, t') \\= \frac{\pm n}{(2\pi i)^2 \sqrt{tt'}} \int_{\Sigma_{\rest}} dz \int^{x^* + Rn^{-\gamma}}_{x^* - Rn^{-\gamma}} dw \frac{1}{z - w} \left[ Y_n^{-1}(z) Y_{n, \pm}(w) \right]_{2,2}  e^{\frac{n(x - (1-t)z)^2}{2t(1-t)} -
    \frac{n (y - (1-t')w)^2}{2t'(1-t')}}.
\end{multline}

\subsubsection{Proof of Proposition \ref{prop:main}\ref{enu:prop:main_b}}  \label{subsec:kernel_asy_rest}
In the paper \cite{Deift-Kriecherbauer-McLaughlin-Venakides-Zhou99}, the Riemann--Hilbert steepest descent method of Deift and Zhou was
 used to obtain large $n$ asymptotics for $Y_n$. This asymptotic analysis was completed also in 
 situations where the equilibrium measure $\mu_0$ has singular points, but without explicit 
 construction of local parametrices near the singular points; only their existence was proved. 
 Nevertheless, we can use the results from \cite{Deift-Kriecherbauer-McLaughlin-Venakides-Zhou99} to obtain an upper bound for 
 $\left[ Y_n^{-1}(z) Y_n(w) \right]_{2,2}$, which we will need
to estimate \eqref{def:KnXrestI} and \eqref{def:KnXrestI_3+-}.
Note that \eqref{eq:KnMinYn} contains the $(2,1)$-entry of $Y_n^{-1}(z) Y_n(w)$
while we are now interested in the $(2,2)$-entry, which appears in \eqref{def:KnXrestI}
and \eqref{def:KnXrestI_3+-}. 

\begin{lemma} \label{lem:Y}
  Let $V$ be real analytic and such that $\frac{V(x)}{\log(1+|x|)}\to +\infty$ as $x\to \pm\infty$. 
  Let $x^*$ be a singular interior or edge point with exponent $\kappa$ and $\gamma = (\kappa+1)^{-1}$.
  Then there exists a constant $C > 0$ such that for large enough $n$, we have the inequality
  \begin{equation} \label{eq:Yassump}  
  \left\lvert \left[ Y_n^{-1}(z) Y_n(w) \right]_{2,2} \right\rvert  \leq C n^{\gamma}
  	\left\lvert e^{n(g(z)-g(w))}\right\rvert,
  \end{equation}
  for all $z \in x^* + i\mathbb R$ and  $w \in \Gamma \cup (x^* - R n^{-\gamma}, x^* + R n^{-\gamma})$
  with $\Gamma$ given by \eqref{def:Gamma}, and $g$ is defined in \eqref{def:gs}. For $w \in (x^* - R n^{-\gamma}, x^* + R n^{-\gamma})$, $Y_n(w)$ in \eqref{eq:Yassump} is understood as either $Y_{n, +}(w)$ or $Y_{n, -}(w)$, with $g(w)$ meaning $g_{\pm}(w)$ accordingly.
\end{lemma}
Because it is rather technical, the proof of Lemma \ref{lem:Y} is postponed to Appendix \ref{sec:RHP_summary}.

The factor $n^\gamma$ on the right-hand side of \eqref{eq:Yassump} is relevant only for the case of
an edge point, and it could be dropped for the case of a singular interior point, see \eqref{eq:ineq_1st_column}
and \eqref{eq:ineq_2nd_column}.

From \eqref{eq:Yassump} we conclude that
\[  \left\lvert \left[ Y_n^{-1}(z) Y_n(w) \right]_{2,2} \right\rvert  \leq  
	C n^{\gamma}\left\lvert e^{n(g(z)-g_+(x^*))} \right\rvert \cdot \left\lvert e^{-n(g(w)-g_+(x^*))} \right\rvert.
    \]

Using Lemma \ref{lem:Y} and the observation that $\lvert z - w \rvert \geq R n^{-\gamma}$ on the relevant integration contours in \eqref{def:KnXrestI} and \eqref{def:KnXrestI_3+-}, we obtain the estimates
\begin{align}
  \left\lvert \widetilde{K}^{\X}_{n, 1}(x, y; t, t') \right\rvert \leq {}& \frac{C n^{1+ 2\gamma}}{2 \pi^2 \sqrt{tt'} R}  I_{n, \rest}(x) J_{n, \rest}(y), \label{eq:PartIestimate} \\
  \left\lvert \widetilde{K}^{\X}_{n, 2}(x, y; t, t') \right\rvert \leq {}& \frac{C n^{1+ 2\gamma}}{2 \pi^2 \sqrt{tt'} R}  I_{n, \loc}(x) J_{n, \rest}(y), \label{eq:PartIestimate_2} \\
  \left\lvert \widetilde{K}^{\X}_{n, 3\pm}(x, y; t, t') \right\rvert \leq {}& \frac{C n^{1+ 2\gamma}}{2 \pi^2 \sqrt{tt'} R}  I_{n, \rest}(x) J_{n, \loc}(y), \label{eq:PartIestimate_3}
\end{align}
where
\begin{align}
  I_{n, \Diamond}(x) = {}& \int_{\Sigma_{\Diamond}} \left\lvert \exp \left( n \left[ g(z)-g_+(x^*) + \frac{ (x - (1-t)z)^2}{2t(1-t)} \right] \right) \right\rvert \lvert dz \rvert, \quad \Diamond = \loc \text{ or } \rest, \label{def:I1} \\
  J_{n, \rest}(y) = {}& \int_{\Gamma}  \left\lvert \exp \left( -n \left[ g(w) - g_+(x^*) + \frac{(y - (1-t')w)^2}{2t'(1-t')}  \right] \right) \right\rvert  \lvert dw \rvert, \label{def:J0} \\
  J_{n, \loc}(y) = {}& \int^{x^* + Rn^{-\gamma}}_{x^* - Rn^{-\gamma}} \exp \left( -n \left[ \Re g_{\pm}(w) - \Re g_{\pm}(x^*) + \frac{(y - (1-t')w)^2}{2t'(1-t')}  \right] \right) dw. \label{def:J0_3} 
\end{align} 
In \eqref{def:J0_3} we note that $\Re g_+(x) = \Re g_-(x)$ if $x \in \mathbb R$.
Combining \eqref{def:KnXrestI} and \eqref{def:KnXrestI_3+-}, we have that
\begin{multline} \label{eq:KnXrestestimate2}
  \left\lvert \widetilde{K}^{\X}_{n, \rest} \left(x, y; t, t' \right) \right\rvert \leq \\
  \frac{C n^{1+ 2\gamma}}{2 \pi^2 \sqrt{tt'} R} \big( \left\lvert I_{n, \rest}(x) J_{n, \rest}(y) \right\rvert + \left\lvert I_{n, \loc}(x) J_{n, \rest}(y) \right\rvert + 2 \left\lvert I_{n, \rest}(x) J_{n, \loc}(y) \right\rvert \big).
\end{multline}
The proof of \eqref{eq:KnXrestestimate} now relies on good estimates for 
$I_{n, \loc}(x), I_{n, \rest}(x), J_{n, \loc}(y), J_{n, \rest}(y)$. It suffices to have 
the following lemma, whose proof is given in Section \ref{subsec:asy_Lemma_proof}.

\begin{lemma} \label{lem:IJestimates}
  Let $u, v$ be in a compact subset of $\mathbb R$, and let $x, y$ depend on $u, v$ by 
  \begin{equation}\label{def xy}
  x=\widehat x_t^*+\frac{u}{\widehat c_t n^\gamma},\qquad y=\widehat x^*_{t'}+\frac{v}{\widehat c_{t'} n^\gamma}.
  \end{equation} 
  Then there exist constants $c_1, c_2 > 0$ that do not depend on $n$ or $R$ such that the following inequalities hold for $n$ and $R$ large enough,
\begin{align}
  \log I_{n, \loc}(x) \leq {}& \frac{n t V'(x^*)^2}{8(1-t)}  +  \frac{n^{1-\gamma} V'(x^*)}{2(1-t)\widehat c_{t}} u	+ n^{1-2\gamma} c_1 u^2, \label{eq:I0bound} \\
  \log I_{n, \rest}(x) \leq {}& \frac{n t V'(x^*)^2}{8(1-t)} + \frac{n^{1-\gamma} V'(x^*)}{2(1-t)\widehat c_{t}} u + n^{1-2\gamma} (c_1 u^2 - c_2 R^2), \label{eq:Ijbound} \\
  \log J_{n, \loc}(y) \leq {}& -\frac{n t' V'(x^*)^2}{8(1-t')} - \frac{n^{1-\gamma} V'(x^*)}{2(1-t')\widehat c_{t'}} v + n^{1-2\gamma} c_1 v^2, \label{eq:J0bound} \\
  \log J_{n, \rest}(y) \leq {}& -\frac{n t' V'(x^*)^2}{8(1-t')} -\frac{n^{1-\gamma} V'(x^*)}{2(1-t')\widehat c_{t'}} v + n^{1-2\gamma} (c_1 v^2 - c_2 R^2).  \label{eq:Jjbound}
\end{align}
\end{lemma}	

Assuming Lemma \ref{lem:IJestimates}, we can estimate from \eqref{eq:KnXrestestimate2},
\begin{multline} \label{eq:KnXrestestimate3}
  \left\lvert \widetilde{K}^{\X}_{n, \rest}
  \left(\widehat x_t^*+\frac{u}{\widehat c_t n^\gamma}, \widehat x^*_{t'}+\frac{v}{\widehat c_{t'} n^\gamma};t,t'\right) \right\rvert
  \leq  \frac{4Cn^{1+2\gamma}}{\pi^2 \sqrt{tt'} R} e^{\frac{-n(t'-t)}{8(1-t)(1-t')}V'(x^*)^2} \\
  \times \exp \left[ n^{1-\gamma} V'(x^*) \left( \frac{u}{2(1-t)\widehat c_{t}} - \frac{v}{2(1-t')\widehat c_{t'}} \right)
    + n^{1-2\gamma}(c_1 (u^2 + v^2)  - c_2 R^2) \right].
\end{multline}
Furthermore, from \eqref{eq:Hnhat}, we obtain that
\begin{align}
 \widehat H_n(u; t) = {}& \frac{nt}{8(1-t)}V'(x^*)^2+\frac{n^{1-\gamma} V'(x^*)}{2(1-t)\widehat c_{t}} u + \frac{n^{1-2\gamma} V''(x^*)}{4c_0 (1-t)\widehat c_{t}} u^2 + O(n^{1-3\gamma}), \\
  \widehat H_n(v; t') = {}& \frac{nt'}{8(1-t')}V'(x^*)^2+\frac{n^{1-\gamma} V'(x^*)}{2(1-t')\widehat c_{t'}} v + \frac{n^{1-2\gamma} V''(x^*)}{4c_0 (1-t)\widehat c_{t'}} v^2 + O(n^{1-3\gamma}),
\end{align}
as $n \to \infty$, uniformly for $u, v$ in a compact set, where we use the fact that $R_n(u; t)$ and $R_n(v; t')$ are bounded as $n\to \infty$, see \eqref{eq:Rnxi}.

This implies that for some constant $c_3 > 0$, and for $n$ large enough,
\begin{multline} 
  -\widehat H_n(u; t) +\widehat H_n(v; t') \leq \frac{n(t'-t)}{8(1-t)(1-t')}V'(x^*)^2\\- n^{1-\gamma} V'(x^*) \left( \frac{u}{2(1-t)\widehat c_{\tau}} - \frac{v}{2(1-t')\widehat c_{t'}} \right) 
  + c_3 n^{1-2\gamma} (u^2 + v^2).
\end{multline}
Combining this with \eqref{eq:KnXrestestimate3} we get
\begin{multline} \label{eq:KnXrestestimate4} 
  e^{-\widehat H_n(u; t)+ \widehat H_n(v; t')} \left\lvert \widetilde{K}^{\X}_{n, \rest}  
    \left(\widehat x_t^*+\frac{u}{\widehat c_t n^\gamma}, \widehat x^*_{t'}+\frac{v}{\widehat c_{t'} n^\gamma}; t,t'\right)
  \right\rvert \\
  \leq  \frac{4 C n^{1+2\gamma}}{\pi^2 \sqrt{tt'} R} 
  e^{n^{1-2\gamma} ((c_1 + c_3) (u^2 + v^2)  - c_2 R^2)}. 
\end{multline}
Taking $R$ large enough (namely such that $c_2 R^2 > (c_1+c_3)(u^2 + v^2)$, which we
can do since $c_1, c_2, c_3$ do not depend on $R$, while $u$ and $v$ are restricted to a compact),  we obtain \eqref{eq:KnXrestestimate} 
for some $c > 0$.

This completes the proof of Proposition \ref{prop:main}\ref{enu:prop:main_b}, pending
the proofs of Lemma \ref{lem:Y} and \ref{lem:IJestimates}.
As already said, the proof of Lemma \ref{lem:Y} is in Appendix \ref{sec:RHP_summary}.
We end this section with the proof of Lemma \ref{lem:IJestimates}. 

\subsubsection{Proof of Lemma \ref{lem:IJestimates}} \label{subsec:asy_Lemma_proof}

In this section, we assume that $u,v$ are real numbers in a compact subset of $\mathbb R$, and $x, y$ depend on $u,v$ by \eqref{def xy}.

\paragraph{Estimates on $I_{n, \loc}$ and $I_{n, \rest}$.}
We take $z = x^* + i \zeta$ with real $\zeta > 0$. Then by the definition \eqref{def:gs} of the $g$-function, we have
\begin{align*} \Re g(x^* + i\zeta) & = \int \log |x^* + i \zeta - u| d\mu_0(u)
	\end{align*}
with derivative
\[ \frac{d}{d\zeta} \Re g(x^* + i\zeta) = \int \frac{\zeta d\mu_0(u)}{\zeta^2 + (x^*-u)^2} \leq 
\zeta \int \frac{d\mu_V(u)}{(x^*-u)^2} = \frac{\zeta}{\tau_{\crit}}, \]
see \eqref{def:taucr}.
Thus after integration
\begin{equation} \label{eq:Regestimate} 
  \Re \left[ g(x^* + i\zeta) - g_+(x^*) \right] \leq  \frac{\zeta^2}{2 \tau_{\crit}},  \qquad \zeta > 0. 
\end{equation}

First we consider $I_{n, \rest}$ defined in \eqref{def:I1}, and prove \eqref{eq:Ijbound}. We note that $g(\bar{z}) = \overline{g(z)}$, so we only need to consider the integral in \eqref{def:I1} on $\Sigma_{\rest} \cap \mathbb C_+$, and have
\begin{equation} \label{eq:integral_est_I_rest}
  \begin{split}
    \log I_{n, \rest}(x) & = \log \left( 2 \int^{+\infty}_{Rn^{-\gamma}} e^{n \Re \left[ g(x^* + i \zeta) -g_+ (x^*) + \frac{1}{2t(1-t)} (x-(1-t)x^*-i(1-t)\zeta)^2 \right]} d\zeta \right) \\
    & \leq  \frac{n}{2t(1-t)} \left(\frac{t}{2} V'(x^*) + \frac{u}{\widehat c_{t} n^{\gamma}} \right)^2
    + \log \int_{Rn^{-\gamma}}^{\infty} e^{-\frac{n}{2}(\frac{1-t}{t}-\frac{1-t_{\crit}}{t_{\crit}}) \zeta^2} d\zeta + \log 2 \\
    & = \frac{n t V'(x^*)^2}{8(1-t)}  + \frac{n^{1-\gamma} V'(x^*)}{2(1-t)\widehat c_{t}} u 
    + \frac{n^{1-2\gamma}}{2 t(1-t) \widehat c_{t}^2} u^2 + 
    \log \int_{Rn^{-\gamma}}^{\infty} e^{-\frac{n}{2}(\frac{1-t}{t}-\frac{1-t_{\crit}}{t_{\crit}}) \zeta^2} d\zeta + \log 2.
  \end{split}
\end{equation}
Note that $t < t_{\crit}$ so that $\frac{1-t}{t}-\frac{1-t_{\crit}}{t_{\crit}} > 0$.
Then it is straightforward to obtain \eqref{eq:Ijbound} 
if we take $c_1 \geq \frac{1}{2t(1-t)\widehat c_{t}^2}$ and $c_2 < \frac{1}{2}(\frac{1-t}{t}-\frac{1-t_{\crit}}{t_{\crit}})$.

Next we consider $I_{n, \loc}$, also defined in \eqref{def:I1}, and prove \eqref{eq:I0bound}. By the argument above, with the integral domain in \eqref{eq:integral_est_I_rest} changed from $(R n^{-\gamma}, +\infty)$ to $(0, R n^{-\gamma})$, we find
\begin{equation}
  \log I_{n, \loc}(x) \leq \frac{nt V'(x^*)^2}{8(1-t)}
  + \frac{n^{1-\gamma} V'(x^*)}{2(1-t)\widehat c_t} u + 
  \frac{n^{1-2\gamma}}{2 t(1-t)\widehat c_{t}^2} u^2 + \log
  \int^{R n^{-\gamma}}_0 e^{-\frac{n}{2}(\frac{1-t}{t}-\frac{1-t_{\crit}}{t_{\crit}}) \zeta^2} d\zeta + \log 2. 
\end{equation}
This gives \eqref{eq:I0bound}, provided again that $c_1 \geq \frac{1}{2t(1-t) \widehat c_{t}^2}$
and $n$ is large enough.

\paragraph{Estimates on $J_{n, \loc}$ and $J_{n, \rest}$.}

As in the estimates on $I_{n, \loc}$ and $I_{n, \rest}$, we only need to consider $w \in \mathbb C_+$.
\begin{lemma} \label{lem:est_J}
  Let either $w = x^* + R n^{-\gamma} + r e^{\pi i/6}$ or $w = x^* - R n^{-\gamma} + r e^{5\pi i/6}$. For $n$ large enough, there exists $C > 0$ such that for all $r > 0$,
  \begin{equation} \label{eq:Realpos} 
    \Re f(w) > -C n^{-3\gamma}, \quad \text{where} \quad f(w) = g(w) - g_+(x^*) - \frac{V'(x^*)}{2} (w - x^*) + \frac{(w - x^*)^2}{2 \tau_{\crit}}.
  \end{equation}
\end{lemma}
\begin{proof}
  If $x^*$ is a singular interior point, then the Taylor expansion of $G(w) = g'(w)$ is given in  \eqref{eq:Gmu0expansion1}, and if $x^*$ is a singular right edge point, then the Puiseux expansion of $G(w) = g'(w)$ is given in \eqref{eq:expansion_G_mu_0_edge}. In both cases, we have that for $w \in \mathbb C_+$ and $w\to x^*$,
  \begin{equation}\label{expansion g}
    g(w) = g_+(x^*)- g_0(w - x^*) - \frac{g_1}{2} (w - x^*)^2 - \frac{g_2}{3} (w - x^*)^3 + O((w - x^*)^4).
  \end{equation}
  Furthermore, by the formulas \eqref{def:g1} and \eqref{eq:xtauV} for the values of $g_0$ and $g_1$, we have
  \begin{equation}
    f(w) = -\frac{g_2}{3} (w - x^*)^3 + O((w - x^*)^4),\qquad w\to x^*,
  \end{equation}
  and for sufficiently small $\epsilon > 0$, Lemma \ref{lem:est_J} is proved for $r \in (0, \epsilon)$, by direct calculation.

  For $r > \epsilon$, 
we need to prove \eqref{eq:Realpos} with $w$ defined by $w = x^* + r e^{\pi i/6}+\delta$ or $w = x^* + r e^{5\pi i/6}-\delta$, and with $\delta=Rn^{-\gamma}$.
Given $R>0$, if $n$ is sufficiently large, then $\delta$ can be taken arbitrarily small. Since $f(x^* + r e^{\pi i/6}+\delta)$ and $f(x^* + r e^{5\pi i/6}-\delta)$ depend continuously on $\delta$, it is sufficient to prove \eqref{eq:Realpos} for $\delta=0$.

By \eqref{def:gs}, \eqref{eq:Vderivative}, and \eqref{eq:taucrV},
  \begin{equation}
    f(w) = \int \left[ \log \left(1 - \frac{w - x^*}{u-x^*}\right) +  
      \frac{w - x^*}{u-x^*} + \frac{(w - x^*)^2}{2(u-x^*)^2} \right] d\mu_0(u).
  \end{equation}
  Note that $\arg\frac{w - x^*}{u-x^*} \in \{ \pm \frac{\pi}{6}, \pm \frac{5\pi}{6} \}$ for $u \in \mathbb R$.
  Since $\mu_0$ is a positive measure, it is therefore enough to prove that
  \begin{equation} \label{lem52:toprove} 
    \arg v \in \left\{\pm \frac{\pi}{6}, \pm \frac{5\pi}{6} \right\} \implies
    \Re \left[ \log \left(1 - v \right) + v + \frac{v^2}{2} \right] > 0. 
  \end{equation}
  in order to obtain \eqref{eq:Realpos}.
  
  Putting $v = r e^{i \theta}$, where $\theta \in \{ \pm \frac{\pi}{6}, \pm \frac{5\pi}{6} \}$, we have
  \begin{equation} \label{eq:Repart} 
    \Re \left[ \log \left(1 - v \right) + v + \frac{v^2}{2} \right] 
    =  \frac{1}{2} \log\left(1+r^2 - 2r \cos \theta \right)	
    + r \cos(\theta) +  \frac{r^2}{2} \cos(2 \theta) ,
  \end{equation}
  which is zero for $r=0$ and which has an $r$-derivative
  \begin{equation} \label{eq:rderiv} 
    \frac{r^2}{1+ r^2 + 2r \cos \theta} 
    \left( \cos \theta (1-2 \cos(2\theta)) + r \cos(2\theta) \right).  
  \end{equation}
  Since $\theta \in \{ \pm \frac{\pi}{6}, \pm \frac{5\pi}{6} \}$, we have $\cos(2\theta) = 1/2$,
  and \eqref{eq:rderiv} reduces to $\frac{r^3}{2(1+ r^2 + 2 r \cos \theta)}$
which is $> 0$. Thus
\eqref{eq:Repart} increases if $\theta = \arg w \in \{ \pm \pi/6, \pm 5 \pi/6 \}$ and $r = \lvert w \rvert$  increases. 
Then \eqref{lem52:toprove} is proved, and this completes the proof of the lemma.
\end{proof}
Because of Lemma \ref{lem:est_J}, we find for $w \in \Gamma_1 \cup \Gamma_2$,
\begin{multline} \label{eq:Realineq1}  
    \Re \left[ g(w) - g_+(x^*) + \frac{1}{2t'(1-t')} (y-(1-t')w)^2 \right]\\ 
    \geq \Re\left[ \frac{V'(x^*)}{2} (w - x^*) - \frac{(w - x^*)^2}{2 \tau_{\crit}}  + \frac{1}{2t'(1-t')} (y-(1-t')w)^2 \right] - C n^{-3\gamma}.
\end{multline}
By \eqref{eq:xtauV} and the expression \eqref{def xy} for $y$, it is straightforward to check that the right-hand side of this expression is equal to 
\begin{multline}\label{estimateGamma1}
  \frac{1}{2} \left(\frac{t_{\crit}-t'}{t_{\crit}t'}\right) \Re
  \left[ \left( w - x^* - \frac{v}{c_0n^{\gamma}} \right)^2 \right] + \frac{t' V'(x^*)^2}{8(1-t')}\\  + \frac{n^{-\gamma} V'(x^*)}{2(1-t')\widehat c_{t'}} v
  - \frac{n^{-2\gamma}}{2\widehat c_{t'}^2 (t_{\crit}-t')} v^2 - Cn^{-3\gamma}.
\end{multline}
  
For $w \in \Gamma_1$, we have $\Re (w - (x^* + R n^{-\gamma}))^2 = \frac{1}{2} \lvert w - (x^* + R n^{-\gamma}) \rvert^2$
and  one may check that
\begin{align}
&\Re
  \left[ \left( w - x^* - \frac{v}{c_0n^{\gamma}} \right)^2 \right]= \Re
  \left[ \left( w - x^* - Rn^{-\gamma} \right)^2 +\frac{(c_0R-v)^2}{c_0^2n^{2\gamma}}+2\frac{c_0R-v}{c_0n^\gamma}(w-x^*-\frac{R}{n^\gamma})\right]\nonumber\\
  &=\frac{1}{2} \lvert w - (x^* + R n^{-\gamma}) \rvert^2+\frac{(c_0R-v)^2}{c_0^2n^{2\gamma}}+\sqrt{3}\frac{c_0R-v}{c_0n^\gamma}|w-x^*-Rn^{-\gamma}|\nonumber\\&=
 \frac{1}{2}\left( \lvert w - (x^* + R n^{-\gamma}) \rvert +\sqrt{3}\frac{c_0R-v}{c_0n^\gamma}\right)^2-\frac{(c_0R-v)^2}{2c_0^2n^{2\gamma}},
\end{align}
for $w\in\Gamma_1$.

Substituting this in \eqref{estimateGamma1}, we obtain that for $w \in \Gamma_1$,
\begin{multline} \label{eq:Realineq2} 
  \Re \left[ g(w) - g_+(x^*)  + \frac{1}{2t'(1-t')} (y-(1-t')w)^2  \right] 
  \geq 
  c_4 \left( \lvert w - (x^* + R n^{-\gamma}) \rvert+\sqrt{3}\frac{c_0R-v}{c_0n^\gamma}\right)^2  \\
  + \frac{t' V'(x^*)^2}{8(1-t')} + \frac{V'(x^*)}{2 (1-t')\widehat c_{t'} n^{\gamma}} v
  -c_4\frac{(c_0R-v)^2}{c_0^2n^{2\gamma}} - \frac{n^{-2\gamma}}{2\widehat c_{t'}^2 (t_{\crit}-t')} v^2- C n^{-3\gamma}, 
\end{multline}
with $c_4  = \frac{t_{\crit}-t'}{4t_{\crit}t'}$.

We obtain from \eqref{eq:Realineq2}, for $j=1,2,3,4$,
\begin{multline*}
  \log \int_{\Gamma_{1}} \left\lvert e^{-n(g(w)-g_+(x^*) + \frac{1}{2t'(1-t')}(y-(1-t')w)^2} \right\rvert \lvert dw \rvert 
  \leq -\frac{nt' V'(x^*)^2}{8(1-t')} -  \frac{n^{1-\gamma}V'(x^*)}{2(1-t')\widehat c_{t'}}v
\\+c_4\frac{(c_0R-v)^2}{c_0^2}n^{1-2\gamma}  +c_5v^2n^{1-2\gamma}   
 + \log 
  \int_{\Gamma_{1}}  e^{-n c_4 \left( \lvert w - (x^* + R n^{-\gamma}) \rvert+\sqrt{3}\frac{c_0R-v}{c_0n^\gamma}\right)^2} \lvert dw \rvert + Cn^{1-3\gamma},
\end{multline*}
with $c_5=\frac{1}{2\widehat c_{t'}^2(t_{\crit}-t')}$.

By direct computation, the logarithm of the integral is bounded by
$-3c_4n^{1-2\gamma}\left(\frac{c_0R-v}{c_0}\right)^2$ for $R$ large enough, and we obtain
 for $n$ and $R$ large enough,
\begin{multline} \label{eq:final_est_Gamma_1}
   \log \int_{\Gamma_1} \left\lvert e^{-n(g(w)-g_+(x^*) + \frac{1}{2t'(1-t')}(y-(1-t)w)^2} \right\rvert \lvert dw \rvert - \left(  -\frac{nt' V'(x^*)^2}{8(1-t')} -  \frac{n^{1-\gamma}V'(x^*)}{2(1-t')\widehat c_{t'}} v \right) \\
    \leq {}    
-2c_4n^{1-2\gamma}\left(\frac{c_0R-v}{c_0}\right)^2   +c_5v^2n^{1-2\gamma}    
    + C n^{1 - 3\gamma} 
    < {} n^{1 - 2\gamma}(c_1 v^2 - c_2 R^2),
\end{multline}
for suitable constants $c_1, c_2$ depending on $t'$ but not on $R$ or $n$.

A similar argument gives the estimate \eqref{eq:final_est_Gamma_1} with $\Gamma_1$ replaced by $\Gamma_2$. Then by the complex conjugate invariance of the integrand on the right-hand side of \eqref{def:J0}, we prove \eqref{eq:Jjbound}.

\medskip
Finally, we prove \eqref{eq:J0bound}.
By \eqref{expansion g}, we have
for $w\in\mathbb R$ and $w\to x^*$,
\begin{multline}
\Re g_+(w) - \Re g_+(x^*) + \frac{1}{2t'(1-t')} (y - (1-t')w)^2\\ =-g_0(w-x^*)-g_1(w-x^*)^2+ \frac{1}{2t'(1-t')} (y - (1-t')w)^2+O((w-x^*)^3).
\end{multline}
By \eqref{def xy},
the right-hand side is equal to 
\[\frac{1}{2} \left(\frac{t_{\crit}-t'}{t_{\crit}t'}\right) \Re
  \left[ \left( w - x^* - \frac{v}{c_0n^{\gamma}} \right)^2 \right] + \frac{t' V'(x^*)^2}{8(1-t')}\\  + \frac{n^{-\gamma} V'(x^*)}{2(1-t')\widehat c_{t'}} v
  - \frac{n^{-2\gamma}}{2\widehat c_{t'}^2 (t_{\crit}-t')} v^2 +O((w-x^*)^3)\]
  as $w\to x^*$.
By \eqref{def:J0_3}, this implies
\[\log J_{n,\loc}(y)\leq -\frac{nt' V'(x^*)^2}{8(1-t')}- \frac{n^{1-\gamma} V'(x^*)}{2(1-t')\widehat c_{t'}} v+\frac{n^{1-2\gamma}}{2\widehat c_{t'}^2 (t_{\crit}-t')} v^2+O(n^{1-3\gamma}),\] as $n\to\infty$,
and \eqref{eq:J0bound} follows.

\section{Proof of Theorem \ref{thm:multi}\ref{enu:thm:multi_b}} \label{section: proof main result (b)}

Let us recall the precise definition of the function $\widehat{H}_n(u; t)$ appearing on the left-hand side of \eqref{eq:Gnlimit_general} and \eqref{eq:Gnlimit1}.  It was introduced in \eqref{eq:Hnhat} as
\begin{equation} \label{eq:Hnhat2} 
\widehat{H}_n(u,t)  =
\frac{t n}{8(1-t)} V'(x^*)^2 + \frac{u n^{1-\gamma}}{2\widehat c_{t}(1-t)} V'(x^*) + \frac{u^2n^{1-2\gamma}}{4c_0 \widehat{c}_{t}(1-t)} V''(x^*)	
+ n^{1-3\gamma} \widehat{R}_n(s_n(u,t), u,t)
\end{equation}
where $s_n(u,t)$ satisfies, see \eqref{eq:snut} and \eqref{eq:snutlimit},
\begin{equation} \label{eq:snut2} 
	s_n(u,t) = u - t c_0 \widehat{c}_t n^{-\gamma} R_n'(s_n(u,t)) =  
	u - \frac{t \widehat{c}_t V'''(x^*)}{4c_0} n^{-\gamma} u^2 + O(n^{-2\gamma} u^4)  
\end{equation}
as $n\to\infty$,
with  
\begin{equation} \label{def:Rns} R_n(s) =  \frac{n^{3\gamma}}{2}\left(V\left(x^*+\frac{s}{c_0n^\gamma}\right)-V(x^*)
-V'(x^*)\frac{s}{c_0n^\gamma}-V''(x^*)\frac{s^2}{2c_0^2n^{2\gamma}}\right)
\end{equation}
as in \eqref{eq:Rnxi}, and 
\begin{equation} \label{eq:Rnsnut}
\widehat{R}_n(s_n(u,t), u,t)
= R_n(s_n(u,t)) + \frac{n^{\gamma}}{2tc_0 \widehat{c}_t} (s_n(u,t) - u)^2.  \end{equation}
see \eqref{eq:Rnhat}.

We first look at the expression in the exponential factor of $G_n$, see \eqref{eq:Gnxy}, which can be written as
\begin{equation} \label{eq:Gaussrewrite}
  -\frac{(x - y)^2}{2(t' - t)} + \frac{x^2}{2(1 - t)} - \frac{y^2}{2(1 - t')} = \\
 -\frac{1}{2(t'-t)}\left(x\sqrt{\frac{1-t'}{1-t}}-y\sqrt{\frac{1-t}{1-t'}}\right)^2.
\end{equation}
If we set, as in Theorem \ref{thm:multi},
\[
x=\widehat x_t^*+\frac{u}{\widehat c_t n^\gamma},\qquad
y=\widehat x_{t'}^*+\frac{v}{\widehat c_{t'} n^\gamma}, 
\]
and use \eqref{eq:constants multi} and \eqref{eq:Gaussrewrite}, we obtain after some calculations
\begin{multline} \label{eq:Gaussrewrite2}
  -\frac{(x - y)^2}{2(t' - t)} + \frac{x^2}{2(1 - t)} - \frac{y^2}{2(1 - t')} \\
  \begin{aligned}
    = {}& -\frac{(1-t)(1-t')}{2(t'-t)}\left(\frac{u}{(1-t)\widehat c_t n^{\gamma}}-\frac{v}{(1-t')\widehat c_{t'}n^{\gamma}}-\frac{(t'-t)V'(x^*)}{2(1-t)(1-t')}\right)^2 \\
    = {}&
    -\frac{t'-t}{8(1-t)(1-t')}V'(x^*)^2  
    +\frac{V'(x^*)}{2n^\gamma}\left(\frac{u}{(1-t)\widehat c_{t}}-\frac{v}{(1-t')\widehat c_{t'}}\right)\\
    & -\frac{1}{n^{2\gamma}}\frac{(1-t)(1-t')}{2(t'-t)}\left(\frac{u}{(1-t)\widehat c_{t}}-\frac{v}{(1-t')\widehat c_{t'}}\right)^2.
  \end{aligned}
\end{multline}
We can further rewrite this, using the  formula \eqref{eq:constants multi} for $\widehat c_t, \widehat c_{t'}$ and the fact that $\frac{V''(x^*)}{2}=-\frac{1-t_{\crit}}{t_{\crit}}$ 
(which follows from \eqref{eq:taucrV} and \eqref{eq:defn_t_crit}).
Then \eqref{eq:Gaussrewrite2} and \eqref{eq:Hnhat2} give us that 
\begin{multline} \label{eq:almost_done_simpler_multi}
-\frac{(x - y)^2}{2(t' - t)} + \frac{x^2}{2(1 - t)} - \frac{y^2}{2(1 - t')}  =
	\frac{\widehat H_n(u; t)}{n} - \frac{\widehat H_n(v; t')}{n} \\
 -\frac{n^{-2\gamma}}{2(t' - t)} \frac{(u-v)^2}{\widehat c_t \widehat c_{t'}} 
+ n^{-3\gamma} (\widehat{R}_n(s_n(v,t'), v,t') - \widehat{R}_n(s_n(u,t), u, t)).
\end{multline}
Hence by \eqref{eq:Gnxy}, \eqref{eq:Fnuv} and \eqref{eq:almost_done_simpler_multi},
\begin{equation} \label{eq:peaked_gaussian}
F_n(u,v;t,t') = \frac{\sqrt{n^{1-2\gamma}}}{\sqrt{2\pi} \sigma_{t,t'}} e^{- \frac{n^{1-2\gamma}}{2\sigma^2_{t,t'}}  (u-v)^2}
 e^{n^{1-3\gamma} (\widehat{R}_n(s_n(v,t'), v, t') - \widehat{R}_n(s_n(u,t), u, t))}, 
\end{equation}
with 
\begin{equation} \label{eq:sigmatt}
	\sigma_{t,t'}^2 = (t'-t) \widehat{c}_t \widehat{c}_{t'}.
	\end{equation}
From \eqref{eq:peaked_gaussian} and the fact that $\gamma=1/(\kappa+1) \leq 1/3$, 
it is clear that $F_n(u,v;t,t') \to 0$ as $n \to \infty$
whenever $t'> t$ and $u \neq v$. This proves \eqref{eq:Gnlimit_general}. 

\bigskip

Furthermore, it is a standard fact that
\begin{equation} \label{eq:model_deta_convergence}
  \lim_{n\to\infty} \frac{\sqrt{n^{1-2\gamma}}}{\sqrt{2\pi} \sigma_{t,t'}} e^{- \frac{n^{1-2\gamma}}{2\sigma^2_{t,t'}}  (u-v)^2} = \delta(u-v),
\end{equation}
in the weak distributional sense.
The Gaussian factor is concentrated in the region where 
$|u-v| \leq n^{-1/2 + \gamma+ \varepsilon}$ for any $\varepsilon > 0$. 
By \eqref{eq:Rnlimit},
\eqref{eq:snutlimit} and \eqref{eq:Rnsnut}, the leading order behavior of $\widehat{R}_n(s_n(u,t),u,t)$ is equal to
\begin{equation} \label{eq:Rnulimit} 
	\widehat{R}_n(s_n(u,t),u,t) = \frac{V'''(x^*)}{12c_0^3} u^3 + O(n^{-\gamma}u^4)
	\qquad \text{ as } n \to \infty. \end{equation}
If $\gamma = 1/3$ or $2/7$ (so that $\kappa=2$ or $5/2$),  \eqref{eq:model_deta_convergence} and \eqref{eq:Rnulimit}
are enough to directly prove 
\begin{equation} \label{eq:weak_conv_F_n}
  \lim_{n\to\infty} F_n(u,v; t,t') = \delta(u-v),
\end{equation}
which gives us \eqref{eq:Gnlimit1}. However for smaller $\gamma$ (recall by
\eqref{def:gamma} that $\gamma = 1/(\kappa + 1) \leq 1/3$), the limit \eqref{eq:Rnlimit} is 
not enough to conclude \eqref{eq:weak_conv_F_n} and a more refined analysis is needed.

To analyze $F_n(u, v; t, t')$ for general $\gamma$, we first note that
by \eqref{eq:Rnsnut} and \eqref{eq:snut2}
\begin{equation} \label{eq:Rnhatprime}
 \frac{d}{du} \widehat{R}_n(s_n(u,t),u,t) = 
	 - \frac{n^{\gamma}}{t c_0 \widehat{c}_t}(s_n(u,t) - u)
	 = R_n'(s_n(u,t)) 
	 	\end{equation}
and thus, for any fixed $v$,
\begin{equation}
  \frac{d}{d u} \left(\frac{(u - v)^2}{2\sigma^2_{t, t'}} + n^{-\gamma} \widehat{R}_n(s_n(u,t), u, t) \right) = \frac{u - v}{\sigma^2_{t, t'}} - \frac{s_n(u,t)-u}{tc_0 \hat{c}_{t}}.
\end{equation}
Hence, for $n$ large enough, the minimum of $u \mapsto \frac{(u - v)^2}{2\sigma^2_{t, t'}} + n^{-\gamma} \widehat{R}_n(s_n(u,t), u, t)$
is assumed, not at $u=v$, but at a nearby point $v^*_n$ that is such that
\begin{equation} \label{def:vstar}
  v^*_n = v + \frac{\sigma^2_{t, t'}}{tc_0 \widehat{c}_t} (s_n(v^*_n,t)-v^*_n)
  = v - n^{-\gamma} \sigma_{t,t'}^2 R_n'(s_n(v^*_n,t))).
\end{equation}
Note that $v^*_n$ depends analytically on $v$, and $v_n^* = v + O(n^{-\gamma})$ as $n \to \infty$.

We then have by \eqref{eq:peaked_gaussian}
\begin{equation} \label{eq:peaked2}
	F_n(u,v;t,t') = \frac{\sqrt{n^{1-2\gamma}}}{\sqrt{2\pi}\sigma_{t,t'}} 
e^{-\frac{n^{1-2\gamma}}{2\sigma_{t,t'}^2} (u-v^*_n)^2}
e^{n^{1-3\gamma} Q_n(u,v;t,t')}
\end{equation}
with
\begin{multline} \label{def:Qn} 
	Q_n(u,v;t,t') = \widehat{R}_n(s_n(v,t'),v,t') - \widehat{R}_n(s_n(u,t),u,t) \\
	- \frac{n^{\gamma}}{2\sigma^2_{t,t'}} (v^*_n-v)^2
	- \frac{n^{\gamma}}{\sigma^2_{t,t'}}(u-v^*_n)(v^*_n-v). \end{multline}

By definition of $v^*_n$, we have
\begin{equation} \label{eq:Qnu} 
	\frac{\partial Q_n}{\partial u}(v^*_n,v;t,t') = 0 \end{equation}
and it is also easy to see from \eqref{def:Qn} and \eqref{eq:Rnulimit} 
that
\begin{equation} \label{eq:Qnuu}
	\frac{\partial^2 Q_n}{\partial u^2}(v^*_n,v;t,t') =
	- \frac{d^2}{du^2} \widehat{R}_n(s_n(u,t),u,t) =  
	 - \frac{V'''(x^*)}{4c_0^3} u^2 + O(n^{-\gamma} u^3), \end{equation}
	 as $n\to\infty$.
However, it is not at all obvious that $Q_n \equiv 0$ on the locus $u = v^*_n$. 
We state this fact as a separate lemma.
\begin{lemma} We have
\begin{equation} \label{eq:snidentity}
	s_n(v^*_n,t) = s_n(v,t')
	\end{equation}
	and 
	\begin{equation} \label{eq:Qnidentity} 
	Q_n(v^*_n,v;t,t') = 0.
	\end{equation}
\end{lemma}
\begin{proof}
We recall that by \eqref{eq:snut2}
\begin{align} \label{eq:snvtp} 
		s_n(v,t') & = v - t' c_0 \widehat{c}_{t'} n^{-\gamma} R_n'(s_n(v,t')) \\
		 \nonumber
		s_n(v_n^*,t) & = v_n^* - t c_0 \widehat{c}_{t} n^{-\gamma} R_n'(s_n(v_n^*,t))  \\
		& = v  - (\sigma_{t,t'}^2 + tc_0 \widehat{c}_{t} ) n^{-\gamma} R_n'(s_n(v_n^*,t)) \label{eq:snvstart}
		\end{align}
where we used \eqref{def:vstar}.
From the explicit formulas for $\sigma_{t,t'}^2$, and $\widehat{c}_t$,
see \eqref{eq:sigmatt} and \eqref{eq:constants multi}, we get
$\sigma_{t,t'}^2 + t c_0 \widehat{c}_t = t' c_0 \widehat{c}_{t'}$.
Inserting this in \eqref{eq:snvstart}, we obtain
\[ s_n(v_n^*,t) = v - t'c_0 \widehat{c}_{t'}  n^{-\gamma} R_n'(s_n(v_n^*,t)), \]
and combining this with \eqref{eq:snvtp},	
\begin{equation} \label{eq:sndiff} 
	s_n(v_n^*,t) - s_n(v,t')
	= - t' c_0 \widehat{c}_{t'} n^{-\gamma} (R_n'(s_n(v_n^*,t)) - R_n'(s_n(v,t')).
	\end{equation}
Due to \eqref{eq:TaylorRn} we have $R_n'(\xi) = O(\xi^2)$ as $\xi \to 0$, uniformly in $n$.
The mapping $\xi \mapsto t' c_0 \widehat{c}_{t'} n^{-\gamma} R_n'(\xi)$
is therefore a contraction on some interval around $\xi = 0$. If $v$ is close
enough to zero then $v$ and $v_n^*$ are in this interval, and it follows
from \eqref{eq:sndiff} that $s_n(v_n^*,t) = s_n(v,t')$ for $v$ in some interval around $v=0$.
By analytic continuation the identity holds for all $v$. This proves \eqref{eq:snidentity}.
	
	\medskip
	
If $v=0$ then  $v_n^* = 0$ as well and it can be checked from \eqref{eq:snutlimit} 
and \eqref{eq:Rnhat} that $s_n(0,t) =0$
and $\widehat{R}_n(0,0,t) = 0$, and therefore $Q_n =0$.
Thus to establish \eqref{eq:Qnidentity} it will be enough to prove that $\frac{d}{dv} Q_n(v_n^*,v;t,t') = 0$, where we have to recall that 
$v_n^*$ depends on $v$.
Since $\frac{\partial Q_n}{\partial u}(v_n^*,v;t,t') = 0$,
we have by \eqref{def:Qn}, \eqref{eq:Rnhatprime} and \eqref{def:vstar}
\begin{align*}
		\frac{d}{dv} Q_n(v_n^*,v;t,t') 
			& = \frac{\partial Q_n}{\partial v}(v_n^*,v;t,t') \\
			& = R_n'(s_n(v,t'))
				+ \frac{n^{\gamma}}{\sigma_{t,t'}^2} (v_n^*-v)  \\
			& = R_n'(s_n(v,t')) - R_n'(s_n(v_n^*,t)),
		\end{align*}
		and this is zero because of \eqref{eq:snidentity}.
\end{proof}

We continue with the proof of \eqref{eq:Gnlimit1}.
Suppose $u$ and $v$ are in a compact set $K$.
Suppose $\varphi$ is a continuous function of two variables
with support in $K \times K$. Then it is a standard fact that
\begin{equation} \label{eq:Fnustarlimit}
\lim_{n\to\infty} \frac{\sqrt{n^{1-2\gamma}}}{\sqrt{2\pi} \sigma_{t,t'}}
\iint e^{- \frac{n^{1-2\gamma}}{2\sigma^2_{t,t'}}  (u-v_n^*)^2}  \varphi(u,v) dudv = 
\int \varphi(u,u) du,
\end{equation}
since $v_n^* \to v$ as $n \to \infty$. 

We conclude  
from \eqref{eq:Qnidentity}, \eqref{eq:Qnu} and \eqref{eq:Qnuu} that
\begin{equation} \label{eq:Qnestimate} 
	|Q_n(u,v;t,t')| \leq C_K (u-v_n^*)^2 \end{equation}
for some constant $C_K > 0$, which depends on $K$, but is independent 
of $n$.
Then it is rather straightforward to check that inclusion
of the factor $e^{n^{1-3\gamma} Q_n(u,v;t,t')}$ in the double
integral will not affect the limit \eqref{eq:Fnustarlimit}.

Indeed, the  double integral in \eqref{eq:Fnustarlimit}
is concentrated on the region where $|u-v_n^*| \leq n^{-1/2 + 4\gamma/3}$
as $n \to \infty$. In this region we have from 
\eqref{eq:Qnestimate} that $n^{1-3\gamma} |Q_n(u,v; t,t')| \leq C_K n^{-\gamma/3}$ and therefore
\[ e^{n^{1-3\gamma} Q_n(u,v; t, t')} = 1 + O(n^{-\gamma/3}) \]
as $n \to \infty$, uniformly for $u,v \in K$ with $|u-v_n^*| \leq n^{-1/2 + 4 \gamma/3}$.

On the other hand if $|u-v_n^*| \geq n^{-1/2 + 4 \gamma/3}$, then by \eqref{eq:Qnestimate}
\[ e^{-\frac{n^{1-2\gamma}}{2\sigma_{t,t'}^2} (u-v_n^*)^2} e^{n^{1-3\gamma} Q_n(u,v;t,t')} \leq e^{-C n^{1-2\gamma} (u-v_n^*)^2} 
	\leq e^{-Cn^{2 \gamma/3}} \]
for some constant $C > 0$. Thus the contribution of the region
where $|u-v_n^*| \geq n^{-1/2 + 4 \gamma/3}$ remains negligible as $n \to \infty$, also if we include the factor $  e^{n^{1-3\gamma} Q_n(u,v; t, t')}$
in the double integral in \eqref{eq:Fnustarlimit}.
It follows that
\begin{equation} \label{eq:Fnustarlimit2}
\lim_{n\to\infty} \frac{\sqrt{n^{1-2\gamma}}}{\sqrt{2\pi} \sigma_{t,t'}}
\iint  
	e^{- \frac{n^{1-2\gamma}}{2\sigma^2_{t,t'}}  (u-v_n^*)^2} 
	e^{n^{1-3\gamma} Q_n(u,v;t,t')} \varphi(u,v) dudv = \int \varphi(u,u) du.
\end{equation}
Then we obtain \eqref{eq:Gnlimit1} from \eqref{eq:Fnuv}, \eqref{eq:peaked2},
 and \eqref{eq:Fnustarlimit2}, which completes the proof of Theorem \ref{thm:multi}\ref{enu:thm:multi_b}.

\appendix

\section{Correlation kernel for nonintersecting Brownian motions} \label{sec:corr_kernel_multi_time}

First we consider the one-time distribution function of the nonintersecting Brownian motion model defined in Section \ref{section:Brownian}, and give a proof of the following proposition.

\begin{proposition} \label{prop:NIBM}
 Let $\X(t), t\in[0,1]$, consist of $n$ nonintersecting Brownian paths with confluent ending points $0$ at $t=1$ and with the marginal distribution at $t=0$ the same as the eigenvalues of a Hermitian $n\times n$ matrix $M$. The joint probability density function of  $\X(t) = (x_1(t), \dotsc, x_n(t))$ at time $t \in [0, 1)$ is the same as the one of the eigenvalues of $(1 - t)M + \sqrt{t(1 - t)} H$. This is true for a deterministic or random Hermitian matrix $M$.
\end{proposition}

\begin{proof}
By a well-known result in the GUE with external source (see \cite{Bleher-Kuijlaars05, Brezin-Hikami96, Johansson01a}), for a fixed matrix $M$ with eigenvalues $a_1 < a_2 < \dotsb < a_n$, the density function of the eigenvalues of $M + \sqrt{\tau}H$ is given by
\begin{equation}\label{eq:jpdfextsource}
  P(x_1, \dotsc, x_n) = \frac{n^{n/2} \prod^n_{j = 1} e^{-\frac{n a^2_j}{2\tau}}}{n! (2\pi \tau)^{n/2} \Delta_n(a)} \Delta_n(x) \det \left( e^{\frac{n a_j x_k}{\tau}} \right)^n_{j, k = 1} \prod^n_{k = 1} e^{-\frac{nx^2_k}{2\tau}},
\end{equation}
where $\Delta_n(x)=\prod_{1\leq i<j\leq n}(x_j-x_i)$, and similarly for $\Delta_n(a)$. If $M$ is random and the eigenvalues have a distribution $\nu(a_1, \dotsc, a_n)$, then the joint probability density function of the eigenvalues of $M + \sqrt{\tau} H$ is given by
\begin{multline} \label{eq:density_M_sqrt_tau_H}
  P(x_1, \dotsc, x_n) = \frac{n^{n/2}}{n! (2\pi \tau)^{n/2}} \int \frac{ \prod^n_{j = 1} e^{-\frac{n a^2_j}{2\tau}}}{\Delta_n(a)}\det \left( e^{\frac{n a_j x_k}{\tau}} \right)^n_{j, k = 1} d\nu(a_1, \dots, a_n) \\
  \times \Delta_n(x) \prod^n_{k = 1} e^{-\frac{nx^2_k}{2\tau}},
\end{multline}
if the measure $\nu$ is such that the integral exists.

On the other hand, suppose that particles $x_1(t), \dotsc, x_n(t)$ in independent Brownian bridges with diffusion parameter $n^{-1/2}$ are placed at positions  $a_1 < a_2 < \dotsb < a_n$ at the initial time $t = 0$, and their ending positions at time $t = 1$ are $b_1 < \dotsb < b_n$. By the Karlin--McGregor theorem \cite{Karlin-McGregor59}, one can compute the probability that the Brownian paths do not intersect, which is given by
\begin{equation}\label{eq:Zn}
P_n=  \det \left( \frac{\sqrt{n}}{\sqrt{2\pi}} e^{-\frac{n(b_k - a_j)^2}{2}} \right)^n_{j, k = 1} = \left( \frac{n}{2\pi} \right)^{\frac{n}{2}} \prod^n_{j = 1} e^{-n \left( \frac{a^2_j}{2} + \frac{b^2_j}{2} \right)} \det \left( e^{n b_k a_j} \right)^n_{j, k = 1}.
\end{equation}
If we only consider the nonintersecting Brownian paths, the joint probability density function of the particles at time $t \in (0, 1)$ is
\begin{equation}
  \begin{split}
     P(x_1, \dotsc, x_n) &= \frac{1}{P_n}\det \left( \frac{\sqrt{n}}{\sqrt{2\pi t}} e^{-\frac{n(x_k - a_j)^2}{2t}} \right)^n_{j, k = 1} \det \left( \frac{\sqrt{n}}{\sqrt{2\pi(1 - t)}} e^{-\frac{n(x_k - b_l)^2}{2(1 - t)}} \right)^n_{k, l = 1} \\
    &=  \left( \frac{n}{2\pi t(1 - t)} \right)^{n/2} \prod^n_{j = 1} e^{-\frac{n(1-t) a^2_j}{2t}} \prod^n_{k = 1} e^{-\frac{nx^2_k}{2t(1 - t)}} \prod^n_{l = 1} e^{-\frac{n t b^2_l}{2(1 - t)}}\\
    &\qquad\qquad\times \det \left( e^{\frac{n a_j x_k}{t}} \right)^n_{j, k = 1} \frac{\det \left( e^{\frac{nx_k b_l}{1 - t}} \right)^n_{k, l = 1}}{\det \left( e^{n b_k a_j} \right)^n_{j, k = 1}}.
  \end{split}
\end{equation}
Now we take the limit where $b_j \to 0$ for all $j = 1, \dotsc, n$. We use the fact that
\begin{equation} \label{eq:b_to_0_limit}
  \lim_{b_j \to 0} \det \left( e^{\frac{c x_k b_l}{2}} \right)^n_{k, l = 1} \left/   \Delta_n(b)  \right. = c^{\frac{n(n - 1)}{2}} \prod^{n - 1}_{j = 0} j! \Delta_n(x).
\end{equation}
This implies that
\begin{equation}
     P(x_1, \dotsc, x_n) =  \frac{1}{C_n(t)} \frac{\Delta_n(x)}{\Delta_n(a)} \prod^n_{j = 1} e^{-\frac{n(1-t) a^2_j}{2t}} \prod^n_{k = 1} e^{-\frac{nx^2_k}{2t(1 - t)}} 
     \det \left( e^{\frac{n a_j x_k}{t}} \right)^n_{j, k = 1},
\end{equation}
for some constant $C_n(t)$ depending on $n$ and $t$ but not on the starting points. After the rescaling $x_k\mapsto (1-t)x_k$, this is the same as \eqref{eq:jpdfextsource} if we set $\tau=\frac{t}{1-t}$.
If the initial positions $a_j=x_j(0)$ are random with distribution $\nu(a_1,\ldots, a_n)$, the joint probability density function of the non-colliding particles at time $t$ becomes
\begin{multline}
  P(x_1, \dotsc, x_n) = \frac{1}{C_n(t)}
\int \frac{1}{\Delta_n(a)} \prod^n_{j = 1} e^{-\frac{n(1-t) a^2_j}{2t}}\det \left( e^{\frac{n a_j x_k}{t}} \right)^n_{j, k = 1} d\nu(a_1,\ldots, a_n)\\ \times \Delta_n(x) \prod^n_{k = 1} e^{-\frac{nx^2_k}{2t(1 - t)}},
\end{multline}
which becomes \eqref{eq:density_M_sqrt_tau_H} after the rescaling $x_k\mapsto (1-t)x_k$ and after setting $\tau=\frac{t}{1-t}$.
\end{proof}

Now we consider the multi-time distribution and the proof of formulas \eqref{eq:mult_kernel}--\eqref{eq:mult_kernel_tildePE}, which are a generalization of \cite[Theorem 2.3]{Claeys-Kuijlaars-Wang15}.
\begin{proposition} \label{prop:NIBM2}
  Let $\X(t), t\in[0,1]$, consist of $n$ nonintersecting Brownian paths with confluent ending points $0$ at $t=1$ and with the marginal distribution
  \begin{equation}\label{eq:jpdf a}
    \frac{1}{Z_n}\Delta_n(a)^2 \prod_{j=1}^n e^{-n V(a_j)}da_1\dotsc da_n.
  \end{equation}
  of the random starting points $a_1<\dotsb< a_n$. Let $m\in\mathbb N$ and fix $0<t_1<\dotsb<t_m<1$. The multi-time correlation kernel $K_n(x,y;t,t')$ for the particles $\X(t_1), \X(t_2),\ldots, \X(t_m)$ is given by \eqref{eq:mult_kernel}--\eqref{eq:mult_kernel_tildePE}.
\end{proposition}
\begin{proof}
First consider the Brownian paths with diffusion parameter $n^{-1/2}$, starting at $a_1 < \dotsb < a_n$ and ending at $b_1 < \dotsb < b_n$. If we require that the Brownian paths are nonintersecting, then similar to the single-time density, the density of the particles at times $t_1 < t_2 < \dotsb < t_m \in (0, 1)$ is, if we write $\vec{x}^{(k)} = (x^{(k)}_1, \dotsc, x^{(k)}_n)$,
\begin{multline}
  P(\vec{x}^{(1)}, \dotsc, \vec{x}^{(m)}) = \frac{1}{P_n}\det \left( \sqrt{\frac{n}{2\pi t_1}} e^{-\frac{n(x^{(1)}_k - a_j)^2}{2t_1}} \right)^n_{j, k = 1} \\
  \times \det(W_1(x^{(1)}_j, x^{(2)}_k)) \dotsm \det(W_{m - 1}(x^{(m - 1)}_j, x^{(m)}_k)) \det \left(\sqrt{\frac{n}{2\pi (1-t_m)}} e^{-\frac{n(x^{(m)}_k - b_l)^2}{2(1 - t_m)}} \right)^n_{k, l = 1},
\end{multline}
where
\begin{equation}
  W_i(x, y) = \frac{\sqrt{n}}{\sqrt{2\pi(t_{i+1}-t_i)}} e^{-\frac{n(y - x)^2}{2(t_{i + 1} - t_i)}} \quad \text{for} \quad i = 1, \dotsc, m - 1.
\end{equation}
 Letting $b_j \to 0$, using a similar calculation as for the single-time density, we have that the multi-time density function at times $t_1 < t_2 < \dotsb < t_m \in (0, 1)$ of the particles  is given by
\begin{multline} \label{eq:mult_time_density_X(t)}
  P(\vec{x}^{(1)}, \dotsc, \vec{x}^{(m)}) = \frac{1}{C_n(t_1,\dotsc, t_m)} \frac{1}{\Delta_n(a)} \prod^n_{j = 1}e^{-\frac{n(1-t_1) a^2_j}{2t_1}} \prod^n_{k = 1} e^{-\frac{n(x^{(1)}_k)^2}{2t_1}} 
     \det \left( e^{\frac{n a_j x_k^{(1)}}{t_1}} \right)^n_{j, k = 1} \\
  \times \det(W_1(x^{(1)}_j, x^{(2)}_k)) \dotsm \det(W_{m - 1}(x^{(m - 1)}_j, x^{(m)}_k)) \Delta_n(x^{(m)}) \prod^n_{k = 1} e^{-\frac{n(x^{(m)}_k)^2}{2(1 - t_m)}},
\end{multline}
for some constant $C_n(t_1,\dotsc,t_m)$ not depending on the $a_j$'s.

If we let the initial positions $a_1, \dotsc, a_n$ be random with distribution \eqref{eq:jpdf a}, then the density of the particles at times $t_1 < t_2 < \dotsb < t_m \in (0, 1)$ is given by
\begin{multline}
  P(\vec{x}^{(1)}, \dotsc, \vec{x}^{(m)}) =  \frac{1}{C_n(t_1,\ldots,t_m)Z_n}\\
  \times \,  \int_{-\infty < a_1 < \dotsb < a_n < \infty} \Delta_n(a) \prod^n_{j = 1}e^{\frac{n a^2_j}{2}-nV(a_j)}  
     \det \left( e^{-\frac{n (a_j -x_k^{(1)})^2}{2t_1}} \right)^n_{j, k = 1} da_1\dotsc da_n \\
  \times\ \det(W_1(x^{(1)}_j, x^{(2)}_k)) \dotsm \det(W_{m - 1}(x^{(m - 1)}_j, x^{(m)}_k)) \Delta_n(x^{(m)})  \prod^n_{k = 1} e^{-\frac{n(x^{(m)}_k)^2}{2(1 - t_m)}}.
\end{multline}
Now using the \Andreief\ formula, we have
\begin{multline}
  \int_{-\infty < a_1 < \dotsb < a_n < \infty} \Delta_n(a) \prod_{j=1}^n e^{-n V(a_j) +n a^2_j/2}\det \left(  e^{-\frac{n(x^{(1)}_k - a_j)^2}{2t_1}} \right)^n_{j, k = 1} da_1 \dotsm da_n  \\
=  \det \left(  \int^{\infty}_{-\infty} a^{j - 1} e^{-n V(a) + na^2/2}e^{- \frac{n(x^{(1)}_k - a)^2}{2t_1}} da \right)^n_{j, k = 1},
\end{multline}
and
\begin{multline}
  P(\vec{x}^{(1)}, \dotsc, \vec{x}^{(m)}) = \frac{1}{C_n(t_1,\dotsc, t_n)Z_n} \det \left(  \int^{\infty}_{-\infty} a^{j - 1}e^{-n V(a) + na^2/2} e^{- \frac{n(x^{(1)}_k - a)^2}{2t_1}} da \right)^n_{j, k = 1} \\
  \times \det(W_1(x^{(1)}_j, x^{(2)}_k)) \dotsm \det(W_{m - 1}(x^{(m - 1)}_j, x^{(m)}_k)) \Delta_n(x^{(m)})  \prod^n_{k = 1} e^{-\frac{n(x^{(m)}_k)^2}{2(1 - t_m)}}.
\end{multline}
Then by elementary linear operations, we can rewrite this as
\begin{multline}
  P(\vec{x}^{(1)}, \dotsc, \vec{x}^{(m)}) = \frac{1}{C_n'(t_1, \dotsc, t_m)} \det(\phi_j(x^{(1)}_k)) \det(W_1(x^{(1)}_j, x^{(2)}_k)) \dotsm  \\
  \times \det(W_{m - 1}(x^{(m - 1)}_j, x^{(m)}_k)) \det(\psi_j(x^{(m)}_k)),
\end{multline}
where
\begin{align}
  \phi_j(x) = {}& \frac{\sqrt{n}}{\sqrt{2\pi t_1}} \int^{\infty}_{-\infty} p_{j - 1}(a) e^{-n V(a) + na^2/2} e^{-\frac{n(x - a)^2}{2t_1}} da, \\
  \psi_j(x) = {}&\frac{\sqrt{n}}{\sqrt{2\pi t_m} i}  \left( \int^{+i\infty}_{-i\infty} p_{j - 1}(s) e^{\frac{n(x - (1 - t_m)s)^2}{2t_m(1 - t_m)}} ds \right) e^{-\frac{nx^2}{2(1 - t_m)}},
\end{align} $p_j(x) = p_{j, n}(x)$ are the orthogonal polynomials  defined in \eqref{eq:defn_p_jn}, and $C_n'(t_1, \dotsc, t_m)$ is a normalization constant. Below we apply the Eynard--Mehta theorem \cite{Eynard-Mehta98} to write down the correlation kernel for $\vec{x}^{(1)}, \dotsc, \vec{x}^{(m)}$. First we define preliminary notations. Let the operator $\Phi: L^2(\mathbb R) \to \ell^2(n)$ be 
\begin{equation}
  \Phi(f(x)) = \left( \int^{\infty}_{-\infty} f(x) \phi_1(x) dx, \dotsc, \int^{\infty}_{-\infty} f(x) \phi_n(x) dx \right)^T,
\end{equation}
and let the operator $\Psi: \ell^2(n) \to L^2(\mathbb R)$ be
\begin{equation}
  \Psi((v_1, \dotsc, v_n)^T) = \sum^n_{k = 1} v_k \psi_k(x).
\end{equation}
We interpret $W_k$ as the kernel of an integral operator from $L^2(\mathbb R)$ to $L^2(\mathbb R)$, and also use it to represent the integral operator itself by abuse of notation. Then we define the operators
\begin{equation} \label{eq:W_[i,j)_with_circ}
  W_{[i, j)} :=
  \begin{cases}
    W_i W_{i+1}\dotsm W_{j - 1} & \text{for $i < j$}, \\
    \Id & \text{for $i = j$}, \\
    0 & \text{for $i >j$},
  \end{cases}
  \quad \text{and} \quad
  \Wcirc_{[i, j)} :=
  \begin{cases}
    W_iW_{i+1} \dotsm W_{j - 1} & \text{for $i < j$}, \\
    0 & \text{for $i \geq j$}.
  \end{cases}
\end{equation}
We also define the operator $M: \ell^2(n) \to \ell^2(n)$ as
\begin{equation} \label{eq:defn_M}
  M := \Phi W_{[1, m)} \Psi,
\end{equation}
which is represented by the $n \times n$ matrix
\begin{equation}
  M_{ij} = \idotsint_{\mathbb R^m} \phi_i(x_1) W_1(x_1, x_2) \dotsm 
  W_{m - 1}(x_{m - 1}, x_m) \psi_j(x_m) dx_1 \dotsm dx_m.
\end{equation}
Then the correlation kernel $K^{\X}_n(x, y; t_i,t_j)$ for $\vec{x}^{(1)}, \dotsc, \vec{x}^{(m)}$ is given by (see \cite{Eynard-Mehta98, Borodin-Rains05}),
\begin{equation} \label{eq:two_parts_of_Kfirst}
  K^{\X}_n(x, y; t_i,t_j) = \widehat{K}^{\X}_n(x, y; t_i,t_j) - \Wcirc_{[i, j)} \quad \text{and} \quad \widehat{K}_n^{\X}(x, y;t_i, t_j) = W_{[i, m)} \Psi M^{-1} \Phi W_{[1, j)}.
\end{equation}
The correlation kernel is only defined up to multiplication with a factor of the form $\frac{F(x;t)}{F(y;t')}$, and it will be convenient for us to work with such a modified representation of the kernel, namely
\begin{equation} \label{eq:two_parts_of_K}
  K^{\X}_n(x, y; t_i,t_j) = \widetilde{K}^{\X}_n(x, y; t_i,t_j) - \frac{e^{\frac{nx^2}{2(1 - t_i)}}}{e^{\frac{ny^2}{2(1 - t_j)}}}\Wcirc_{[i, j)} \quad \text{and} \quad \widetilde{K}_n^{\X}(x, y; t_i, t_j) =\frac{e^{\frac{nx^2}{2(1 - t_i)}}}{e^{\frac{ny^2}{2(1 - t_j)}}} W_{[i, m)} \Psi M^{-1} \Phi W_{[1, j)}.
\end{equation}

It is straightforward to compute that if $i < j$, then
\begin{equation}
  W_{[i, j)}(x, y) = \Wcirc_{[i, j)}(x, y) = \frac{\sqrt{n}}{\sqrt{2\pi(t_j - t_i)}} e^{-\frac{n(x - y)^2}{2(t_j - t_i)}},
\end{equation}
which is equivalent to $G_n(x, y; t, t')$ in \eqref{eq:Gnxy} with $t = t_i$ and $t' = t_j$.

The operator $\Phi W_{[1, j)}$ is from $L^2(\mathbb R)$ to $\ell^2(n)$, and is represented by an $n$-dimensional column operator. Its $l$-th component is, analogous to $Q_k$ in \cite[Formula (2.7)]{Claeys-Kuijlaars-Wang15},
\begin{equation}
  \begin{split}
    (\Phi W_{[1, j)})_l(x) 
    = {}& \int^{\infty}_{-\infty} \phi_l(y) W_{[1, j)}(y, x) dy \\
    = {}& \int^{\infty}_{-\infty} \left( \frac{\sqrt{n}}{\sqrt{2\pi t_1}} \int^{\infty}_{-\infty} p_{l - 1}(a) e^{-nV(a)+na^2/2} e^{-\frac{n(y - a)^2}{2t_1}} da \right) \frac{\sqrt{n}}{\sqrt{2\pi(t_j - t_1)}} e^{-\frac{n(y - x)^2}{2(t_j - t_1)}} dy \\
    = {}& \int^{\infty}_{-\infty} p_{l - 1}(a) e^{-nV(a)+na^2/2} \left( \frac{\sqrt{n}}{\sqrt{2\pi t_1}} \int_{-\infty}^{+\infty}e^{-\frac{n(y - a)^2}{2t_1}} \frac{\sqrt{n}}{\sqrt{2\pi(t_j - t_1)}} e^{-\frac{n(y - x)^2}{2(t_j - t_1)}} dy \right) da \\
    = {}& \frac{\sqrt{n}}{\sqrt{2\pi t_j}} \int^{\infty}_{-\infty} p_{l - 1}(a) e^{-nV(a)+na^2/2} e^{-\frac{n(x - a)^2}{2t_j}} da.
  \end{split}
\end{equation}
Similarly, the operator $W_{[j, m)} \Psi$ is from $\ell^2(n)$ to $L^2(\mathbb R)$, and is represented by an $n$-dimensional row vector. Its $l$-th component is, analogous to $P_k$ in \cite[Formula (2.6)]{Claeys-Kuijlaars-Wang15},
\begin{equation}
  \begin{split}
    & (W_{[j, m)} \Psi)_l(x) \\
    = {}& \int^{\infty}_{-\infty} W_{[j, m)}(x, y) \psi_l(y) dy \\
    = {}& \int^{\infty}_{-\infty} \frac{\sqrt{n}}{\sqrt{2\pi (t_m - t_j)}} e^{-\frac{n(x - y)^2}{2(t_m - t_j)}} \frac{\sqrt{n}}{\sqrt{2\pi t_m} i} \left( \int^{+i\infty}_{-i\infty} p_{l - 1}(s) e^{\frac{n(y - (1 - t_m)s)^2}{2t_m(1 - t_m)}} ds \right) e^{-\frac{n y^2}{2(1 - t_m)}} dy \\
    = {}& \int^{+i\infty}_{-i\infty} p_{l - 1}(s) \int^{\infty}_{-\infty} \frac{\sqrt{n}}{\sqrt{2\pi (t_m - t_j)}} e^{-\frac{n(x - y)^2}{2(t_m - t_j)}} \frac{\sqrt{n}}{\sqrt{2\pi t_m} i} e^{\frac{n(y - (1 - t_m)s)^2}{2t_m(1 - t_m)}} e^{-\frac{n y^2}{2(1 - t_m)}} dy\,ds \\
    = {}& \frac{\sqrt{n}}{\sqrt{2\pi t_j} i} \left( \int^{+i\infty}_{-i\infty} p_{l - 1}(s) e^{\frac{n(x - (1 - t_j)s)^2}{2t_j(1 - t_j)}} ds \right) e^{-\frac{n x^2}{2(1 - t_j)}}.
  \end{split}
\end{equation}
The matrix $M$ is the identity matrix, since analogous to \cite[Formulas (3.8) and (3.9)]{Claeys-Kuijlaars-Wang15},
\begin{equation}
  \begin{split}
    M_{j, k} = {}& \int^{\infty}_{-\infty} \phi_j(x) (W_{[1, m)} \Psi)_k(x) dx \\
    = {}& \frac{n}{2\pi it_1} \int^{\infty}_{-\infty} \left( \int^{\infty}_{-\infty} p_{j - 1}(a) e^{-nV(a)+na^2/2} e^{-\frac{n(x - a)^2}{2t_1}} da \right) \\
    & \phantom{\frac{n}{2\pi t_1} \int^{\infty}_{-\infty}} \times \left( \int^{+i\infty}_{-i\infty} p_{k - 1}(s) e^{\frac{n(x - (1 - t_1)s)^2}{2t_1(1 - t_1)}} ds \right) e^{-\frac{n x^2}{2(1 - t_1)}} dx \\
    = {}& \frac{n}{2\pi it_1} \int^{\infty}_{-\infty} \left( \int^{\infty}_{-\infty} p_{j - 1}(a) e^{-n V(a)} e^{-\frac{n(x - (1 - t_1)a)^2}{2t_1(1 - t_1)}} da \right) \\
    & \phantom{\frac{n}{2\pi t_1} \int^{\infty}_{-\infty}} \times \left( \int^{+i\infty}_{-i\infty} p_{k - 1}(s) e^{\frac{n(x - (1 - t_1)s)^2}{2t_1 (1 - t_1)}} ds \right) dx \\
    = {}& \delta_{j, k},
  \end{split}
\end{equation}
after a straightforward calculation in which we use the orthogonality relation \eqref{eq:defn_p_jn}.

So finally we arrive at the formula,
\begin{equation}
  \begin{split}
    \widetilde{K}^\X_n(x, y;t_j, t_k) = {}& \frac{n}{2\pi \sqrt{t_j t_k} i} \frac{e^{\frac{nx^2}{2(1 - t_j)}}}{e^{\frac{ny^2}{2(1 - t_k)}}}\sum^{n - 1}_{l = 0} \left( \int^{+i\infty}_{-i\infty} p_{l - 1}(s) e^{\frac{n(x - (1 - t_j)s)^2}{2t_j(1 - t_j)}} ds \right) e^{-\frac{n x^2}{2(1 - t_j)}} \\
    & \phantom{\frac{n}{2\pi \sqrt{t_j t_k} i} \sum^{n - 1}_{l = 0}} \times \int^{\infty}_{-\infty} p_{l - 1}(w) e^{-nV(w)+nw^2/2} e^{-\frac{n(y - w)^2}{2t_k}} dw \\
    = {}&  \frac{n}{2\pi \sqrt{t_j t_k} i} \int^{+i\infty}_{-i\infty} ds \int^{\infty}_{-\infty} dw \left( \sum^{n - 1}_{l = 0} p_{l - 1}(s) p_{l - 1}(w) e^{-nV(w)} \right) \frac{e^{\frac{n(x - (1 - t_j)s)^2}{2t_j (1 - t_j)}}}{e^{\frac{n(y - (1 - t_k)w)^2}{2t_k(1 - t_k)}}} \\
    = {}&  \frac{n}{2\pi \sqrt{t_j t_k} i} \int^{+i\infty}_{-i\infty} ds \int^{\infty}_{-\infty} dw\, K^{\PE}_n(s, w) \frac{e^{\frac{n(x - (1 - t_j)s)^2}{2t_j (1 - t_j)}}}{e^{\frac{n(y - (1 - t_k)w)^2}{2t_k(1 - t_k)}}},
  \end{split}
\end{equation}
which is equivalent to \eqref{eq:mult_kernel_tildePE} with $t = t_j$ and $t' = t_k$. Thus we prove Proposition \ref{prop:NIBM2}.
\end{proof}

\section{Summary of the Riemann--Hilbert analysis for unitary ensembles and proofs of Proposition \ref{prop:prop17} and Lemma \ref{lem:Y}} \label{sec:RHP_summary}


In this appendix we prove Proposition \ref{prop:prop17} and Lemma \ref{lem:Y} as consequences of the analysis in \cite{Deift-Kriecherbauer-McLaughlin-Venakides-Zhou99}. 
We discuss the singular interior point case and the singular edge point case separately, and we consider only the singular right edge point case. Note that although in both Proposition \ref{prop:prop17} and Lemma \ref{lem:Y}, the singular point $x^*$ may not be the unique one, for the sake of notational simplicity we prove them only in the case that $x^*$ is the unique singular point, either in the interior or at a right edge. If other singular points exist, we just need to construct more local parametrices and the matrix $R(z) = I + O(n^{-c})$ with $c$ depending on all singular points, see the discussion in the end of \cite[Section 5]{Deift-Kriecherbauer-McLaughlin-Venakides-Zhou99}. All other arguments do not change.

Following the notations in \cite[Figure 1.1]{Deift-Kriecherbauer-McLaughlin-Venakides-Zhou99}, we assume that the support of the equilibrium measure $\mu_0$ is $J = [b_0, a_1] \cup \dotsb \cup [b_N, a_{N + 1}]$, where $-\infty < b_0 < a_1 < \dotsb < b_N < a_{N + 1} < \infty$. 

\subsection{Singular interior point case} \label{subsec:interior_RHP}

We assume as in \cite[Section 5.2]{Deift-Kriecherbauer-McLaughlin-Venakides-Zhou99} that $x^* = \hat{z} \in (b_{j - 1}, a_j) \subseteq J$ is the 
only singular point with exponent $\kappa$.

The strategy in \cite{Deift-Kriecherbauer-McLaughlin-Venakides-Zhou99} is as follows. First the matrix-valued function $Y_n(z)$ defined 
in \eqref{def:Yn} (the same as $Y(z)$ defined in \cite[Formula (2.2) and Formulas (1.86)--(1.88)]{Deift-Kriecherbauer-McLaughlin-Venakides-Zhou99}) 
is transformed into $M(z)$ by \cite[Formula (1.89)]{Deift-Kriecherbauer-McLaughlin-Venakides-Zhou99}, and then to $M^{(1)}(z)$ by 
\cite[Formulas (1.101)--(1.103)]{Deift-Kriecherbauer-McLaughlin-Venakides-Zhou99}. These transformations are given explicitly in terms of the equilibrium 
measure $\mu_0$ and the $g$-function defined in \eqref{def:gs}. $M^{(1)}(z)$ satisfies a Riemann-Hilbert problem (RHP) with 
jump contour $\Sigma^{(2)}$ \cite[Figures 1.5 and 5.3]{Deift-Kriecherbauer-McLaughlin-Venakides-Zhou99}, which divides the complex plane 
into $(N + 1)$ lenses and upper and lower outer infinite regions, as shown schematically in 
Figure \ref{fig:Sigma_2}.
\begin{figure}[htb]
  \begin{minipage}[t]{0.5\linewidth}
    \centering
    \includegraphics{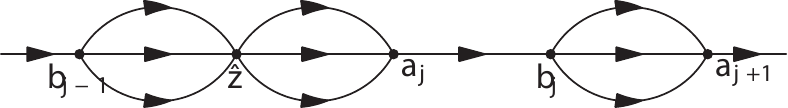}
    \caption{The schematic shape of the contour $\Sigma^{(2)}$.}
    \label{fig:Sigma_2}
  \end{minipage}
  \begin{minipage}[t]{0.5\linewidth}
    \centering
    \includegraphics{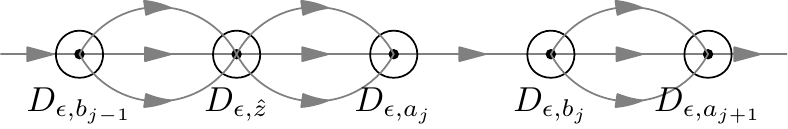}
    \caption{The schematic shape of the contours $D_{\epsilon, a_j}, D_{\epsilon, b_{j - 1}}$ and $D_{\epsilon, \hat{z}}$.}
    \label{fig:discs_interior}
  \end{minipage}
\end{figure}

Next, one constructs the global parametrix $M^{(\infty)}(z)$ on $\mathbb C$ with discontinuity on 
$\mathbb R \setminus J$. $M^{(\infty)}(z)$ is uniformly bounded on 
$\mathbb C\setminus (\bigcup^N_{k = 1} D_{\epsilon, a_k} \cup D_{\epsilon, b_{k - 1}})$, where 
$D_{\epsilon, a_k}$ and $D_{\epsilon, b_{k - 1}}$ are sufficiently small discs with radius $\epsilon$ 
and centered at  $a_j, b_{j - 1}$ respectively, as shown in Figure \ref{fig:discs_interior}. 
Note that in Figure \ref{fig:discs_interior} we also show a small disc $D_{\epsilon, \hat{z}}$ centered 
at $\hat{z}$. $M^{(\infty)}(z)$ is uniformly bounded in $D_{\epsilon, \hat{z}}$, although discontinuous 
on $D_{\epsilon, \hat{z}} \cap \mathbb R$. In all the $(2N + 3)$ small discs, one construct a local 
parametrix, denoted all by $M_p(z)$ (as in \cite[Formula (4.29)]{Deift-Kriecherbauer-McLaughlin-Venakides-Zhou99}). The construction of $M_p(z)$ 
in $D_{\epsilon, a_k}$ and $D_{\epsilon, b_{k - 1}}$ can be carried out explicitly by Airy functions
\cite[Formulas (4.76), (4.92), (4.102)]{Deift-Kriecherbauer-McLaughlin-Venakides-Zhou99}, while the construction of $M_p(z)$ in 
$D_{\epsilon, \hat{z}}$ is given in an implicit way \cite[Formula (5.90)]{Deift-Kriecherbauer-McLaughlin-Venakides-Zhou99}.
  
In the final step the matrix $R(z)$ is defined as
\begin{equation}\label{def:matrixR}
R(z) =
\begin{cases}
 M^{(1)}(z)M^{(\infty)}(z)^{-1} & \text {for } z  \text{ outside the discs}\\
  M^{(1)}(z)M_p(z)^{-1} & \text{ for } z \text{ inside the discs,}
\end{cases}
\end{equation}
and it can be shown to satisfy $R(z) = I + O(n^{-\gamma})$ uniformly on $\mathbb C$, where 
$\gamma = (\kappa + 1)^{-1}$ as in \eqref{def:gamma}. The asymptotics of $Y_n$ are then obtained by 
inverting the explicit transformations from $Y\mapsto M \mapsto M^{(1)}(z)\mapsto R$.

\subsubsection{Proof of Proposition \ref{prop:prop17} (singular interior point case)}
 
\begin{proof}
Let $u, v$ be complex numbers in a compact set, and $x, y$ are expressed by them as in 
\eqref{def:singlimit}, with $x^* = \hat{z}$. First we prove the Proposition 
under the assumption that both $u, v$ are in $\mathbb C_+$, and that $x, y$ are outside of 
the two lenses connecting to $\hat{z}$. 
Then by \cite[Formulas (5.90), (1.89) and (1.101)]{Deift-Kriecherbauer-McLaughlin-Venakides-Zhou99}, we have 
  \begin{multline} \label{eq:(21)_entry_interior_asy}
    \left[ Y_n^{-1}(y) Y_n(x) \right]_{2,1} = e^{n(g(x) + g(y) - \ell - \hat{c})} \\
    \times \left[ \hat{M}_p(\zeta(y))^{-1} \left( e^{-\frac{n \hat{c}}{2} \sigma_3} M^{(\infty)}(y)^{-1} R(y)^{-1} R(x) M^{(\infty)}(x) e^{\frac{n \hat{c}}{2} \sigma_3} \right) \hat{M}_p(\zeta(x)) \right]_{2, 1},
  \end{multline}
  where $\hat{c}$ is defined in \cite[Formula (5.22)]{Deift-Kriecherbauer-McLaughlin-Venakides-Zhou99}, $\hat{M}_p(z)$ is defined by the 
  RHP \cite[Formulas (5.13)--(5.15)]{Deift-Kriecherbauer-McLaughlin-Venakides-Zhou99}, and the mapping $z \mapsto \zeta(z)$ is defined in 
  \cite[Formula (5.23)]{Deift-Kriecherbauer-McLaughlin-Venakides-Zhou99}. Noting that $\hat{c}$ is a purely imaginary number,  
  $M^{(\infty)}(z), M^{(\infty)}(z)^{-1}$ and $\frac{d}{dz}M^{(\infty)}(z)$ are uniformly bounded in 
  a neighborhood of $\hat{z}$, and $y - x = O(n^{-\gamma})$, we have as $n\to\infty$ that
  \begin{equation}
    e^{-\frac{n \hat{c}}{2} \sigma_3} M^{(\infty)}(y)^{-1} R(y)^{-1} R(x) M^{(\infty)}(x) e^{\frac{n \hat{c}}{2} \sigma_3} = I + O(n^{-\gamma})
  \end{equation}
  uniform in $u, v$. On the other hand, by the definition of the mapping $\zeta(z)$, we have \footnote{Here and later several calculations depend on the relation between the density of $\mu_0$, which is denoted by $\Psi(x)$ in \cite{Deift-Kriecherbauer-McLaughlin-Venakides-Zhou99} and $h(x)$ in our paper, and the function $R^{1/2}_+(x) h(x)$ defined in \cite{Deift-Kriecherbauer-McLaughlin-Venakides-Zhou99}. But in \cite{Deift-Kriecherbauer-McLaughlin-Venakides-Zhou99} the relations are stated twice, in \cite[Formulas (1.6) and (3.6)]{Deift-Kriecherbauer-McLaughlin-Venakides-Zhou99}, and they differ by a constant multiple. We use \cite[Formula (3.6)]{Deift-Kriecherbauer-McLaughlin-Venakides-Zhou99}.}
  \begin{equation} \label{eq:zeta(x)_approx}
    \zeta(x) = (2\pi \gamma)^{\gamma} u + O(n^{-\gamma}),\qquad n\to\infty,
  \end{equation}
  uniform in $u$, and a similar expansion holds for $\zeta(y)$. We conclude that
  \begin{equation} \label{eq:Y^-1(y)Y(x)_by_RHP}
    \left[ Y_n^{-1}(y) Y_n(x) \right]_{2,1} = e^{n(g(x) + g(y) - \ell - \hat{c})} \left[ \hat{M}_p \left( (2\pi \gamma)^{\gamma} v \right)^{-1} \hat{M}_p \left( (2\pi \gamma)^{\gamma} u \right) \right]_{2, 1} (1 + O(n^{-\gamma})),
  \end{equation}
  where the $O(n^{-\gamma})$ term is uniform in $u, v$ as $n\to\infty$.

  At last we note that by the definition of $\hat{c}$, the Euler--Lagrange variational condition \eqref{eq:ELequation} and the Taylor expansion \eqref{eq:Gmu0expansion1} of $g'(z)$ at $x^* = \hat{z}$ (where $G_{\mu_0}(w) = g'(w)$), we have
  \begin{equation} \label{eq:approx_V-g-l-c}
    n\left(V(x) - g(x) - \frac{\ell + \hat{c}}{2}\right) = \pi i \gamma u^{\kappa + 1} + O(n^{-\gamma}),
  \end{equation}
  uniformly in $u$ as $n\to\infty$, and a similar approximation holds for $V(y) - g(y) - (\ell + \hat{c})/2$. 
  By \eqref{eq:Y^-1(y)Y(x)_by_RHP} and \eqref{eq:approx_V-g-l-c}, we have that the left-hand side 
  of \eqref{def:singlimit} that is expressed by \eqref{eq:KnMinYn} converges to a limit function in 
  $u, v$ as $n \to \infty$, as long as $u \neq v$. But the condition $u \neq v$ can be easily removed 
  by the analyticity of the functions. Thus we prove Proposition \ref{prop:prop17} for a singular interior 
  point  in case $x$ and $y$ are in the upper half plane and outside of the lenses. 

  If  $x$ or $y$ are inside one of the lenses, or in  $\mathbb C_-$, we compare the jump conditions 
  of the RHPs for $Y_n$, $M^{(1)}$, $M_p$, and $\hat{M}_p$, we find that \eqref{eq:Y^-1(y)Y(x)_by_RHP} still holds, if we replace $\hat{M}_p(z)$ by the function $\tilde{M}_p(z)$, which is an analytic function on $\mathbb C$ that agrees with $\hat{M}_p(z)$ in a sector containing $i\mathbb R_+$, and is the analytic continuation of $\hat{M}_p(z)$ in that sector. Thus Proposition \ref{prop:prop17} for a singular interior point 
  is proved for $u, v$ in any compact subset of $\mathbb C$.
\end{proof}

\subsubsection{Proof of Lemma \ref{lem:Y} for a singular interior point}

\begin{proof}
  We only need to consider the case that $z, w \in \mathbb C_+$, since $Y_n(\bar{z}) = \overline{Y_n(z)}$.
  We note that $\det Y_n(z) = 1$, so the entries of $Y_n(z)^{-1}$ are expressed  in terms of 
  the entries of $Y_n(z)$ in a very simple way. It suffices to show that for any $z \in \mathbb C_+$ 
  on $x^* + i\mathbb R$,
  \begin{equation} \label{eq:ineq_1st_column}
    \left\lvert \left[ Y_n(z) \right]_{1, 1} \right\rvert \leq C \left\lvert e^{ng(z)} \right\rvert, 
    \quad \left\lvert \left[ Y_n(z) \right]_{2, 1} \right\rvert \leq C \left\lvert e^{n(g(z) - \ell)} \right\rvert,
  \end{equation}
  and for any $w \in \mathbb C_+$ on $\Gamma$ or the interval $(x^* - R n^{-\gamma}, x^* + R n^{-\gamma})$, 
  \begin{equation} \label{eq:ineq_2nd_column}
    \left\lvert \left[ Y_n(w) \right]_{1, 2} \right\rvert \leq C \left\lvert e^{-n(g(w) - \ell)} \right\rvert, \quad \left\lvert \left[ Y_n(w) \right]_{2, 2} \right\rvert \leq C \left\lvert e^{-ng(w)} \right\rvert,
  \end{equation}
  where $Y_n(w)$ means $Y_{n, +}(w)$ if $w \in (\hat{z} - R n^{-\gamma}, \hat{z} + R n^{-\gamma})$. 
  In this proof we identify $x^*$ with $\hat{z}$.
  
  First, if $z \in \hat{z} + i\mathbb R$ is out of the disc $D_{\epsilon, \hat{z}}$, then it is in the 
  outer infinite region and outside of any lens, so by the argument in 
  \cite[Paragraph below the proof of Theorem 5.8]{Deift-Kriecherbauer-McLaughlin-Venakides-Zhou99} and \cite[Formulas (1.89) and (1.101)]{Deift-Kriecherbauer-McLaughlin-Venakides-Zhou99},
  \begin{equation} \label{eq:approx_outer_region}
    Y_n(z) = e^{\frac{n \ell}{2} \sigma_3} R(z) M^{(\infty)}(z) e^{n(g(z) - \frac{\ell}{2})\sigma_3},
  \end{equation}
  and \eqref{eq:ineq_1st_column} follows, since both $M^{(\infty)}(z)$ and $R(z)$ are uniformly bounded. 
  If $w \in \Gamma \cup (\hat{z} - R n^{-\gamma}, \hat{z} + R n^{-\gamma})$ is out of the disc 
  $D_{\epsilon, \hat{z}}$, then $w$ is in the outer infinite region and outside of any lens, so
  \eqref{eq:approx_outer_region} holds with $z$ replaced by $w$, and \eqref{eq:ineq_2nd_column} follows.

  Next, if $z \in \hat{z} + i\mathbb R$ is in the semi-disc $D_{\epsilon, \hat{z}} \cap \mathbb C_+$, then 
 since $z$ is outside of the two lenses connecting to $\hat{z}$, by 
 \cite[Formulas (5.90), (1.89) and (1.101)]{Deift-Kriecherbauer-McLaughlin-Venakides-Zhou99} we have
  \begin{equation} \label{eq:Y_n_near_z_hat_above}
    Y_n(z) = e^{\frac{n \ell}{2} \sigma_3} R(z) M^{(\infty)}(z) e^{\frac{n\hat{c}}{2} \sigma_3} \hat{M}_p(\zeta(z))  e^{-\frac{n\hat{c}}{2} \sigma_3} e^{n(g(z) - \frac{\ell}{2})\sigma_3}.
  \end{equation}
  where $\hat{c}$, $\hat{M}_p$ and $\zeta(z)$ are the same as in \eqref{eq:(21)_entry_interior_asy}. 
  By the properties that $\hat{c}$ is purely imaginary and $\hat{M}_p$ is uniformly bounded, we find 
  that \eqref{eq:ineq_1st_column} follows from \eqref{eq:Y_n_near_z_hat_above}. 
  Similarly, if $w \in \Gamma \cup (\hat{z} - R n^{-\gamma}, \hat{z} + R n^{-\gamma})$ is in the 
  semi-disc $D_{\epsilon, \hat{z}} \cup \mathbb C_+$, then we can deform the lenses so that $w$ is outside 
  of the two lenses connecting to $\hat{z}$ as long as $\lvert w - \hat{z} \rvert > 2R n^{-\gamma}$. 
  In this case, \eqref{eq:Y_n_near_z_hat_above} holds with $z$ replaced by $w$, and 
  \eqref{eq:ineq_2nd_column} follows.

  Additionally, if $w \in \Gamma \cup (\hat{z} - R n^{-\gamma}, \hat{z} + R n^{-\gamma})$, $w$ is in the semi-disc $D_{\epsilon, \hat{z}} \cap \mathbb C_+$, and $\lvert w - \hat{z} \rvert \leq 2R n^{-\gamma}$, then by comparing the jump conditions of the RHPs for $Y_n$, $M^{(1)}$, $M_p$ and $\hat{M}_p$, like we do in the proof of Proposition \ref{prop:prop17} (singular interior point case) above, we have that \eqref{eq:Y_n_near_z_hat_above} holds if $z$ is replaced by $w$ and $\hat{M}_p$ is replaced by the function $\tilde{M}_p(z)$, which is an analytic function on $\mathbb C$ that agrees with $\hat{M}_p(z)$ in a sector containing $i\mathbb R_+$, and is the analytic continuation of $\hat{M}_p(z)$ in that sector. Since by \eqref{eq:zeta(x)_approx}, $\zeta(w) = O(1)$ as $n\to\infty$ for our $w$, we have that $\tilde{M}_p(\zeta(w))$ is bounded, and \eqref{eq:ineq_2nd_column} follows.

\end{proof}

\subsection{Singular right-edge point case}

We assume as in \cite[Section 5.3]{Deift-Kriecherbauer-McLaughlin-Venakides-Zhou99} that $a_j$ is the only singular point with exponent $\kappa$, which is denoted by $x^*$ in the main body of the paper.

The strategy in \cite{Deift-Kriecherbauer-McLaughlin-Venakides-Zhou99} is parallel to the singular interior point case and we briefly repeat it, highlighting the differences. $Y_n(z)$ is transformed to $M(z)$, and then to $M^{(1)}(z)$. $M^{(1)}(z)$ satisfies a RHP with jump contour $\Sigma^{(1)}$ \cite[Figure 1.5]{Deift-Kriecherbauer-McLaughlin-Venakides-Zhou99}, which divides the complex plane into $N$ lenses and the upper and lower outer infinite regions.

Next, the construction of the global parametrix $M^{(\infty)}(z)$ is the same as before. The construction 
of $M_p(z)$ in $D_{\epsilon, a_k}$ with $k = 1, \dotsc, j - 1, j + 1, \dotsc, N$ and $D_{\epsilon, b_{k - 1}}$ with $k = 1, \dotsc, N$ are carried out explicitly in terms of Airy functions as before, 
while the construction of $M_p(z)$ in $D_{\epsilon, a_j}$ is given in a more implicit way 
\cite[Formula (5.159)]{Deift-Kriecherbauer-McLaughlin-Venakides-Zhou99}. Finally the matrix $R(z)$ is defined as in \eqref{def:matrixR}, and it 
can be shown to satisfy $R(z) = I + O(n^{-\gamma/2})$ as $n\to\infty$ uniformly on $\mathbb C$, where again 
$\gamma = (\kappa + 1)^{-1}$. Reversing all of the explicit matrix transformations again gives 
asymptotic formulas for $Y_n(z)$.
  
\subsubsection{Proof of Proposition \ref{prop:prop17} for a singular right-edge point}
\begin{proof}
 Let $u, v$ be complex numbers in a compact set, and let $x, y$ be as in \eqref{def:singlimit}, 
 with $x^* = a_j$. First we prove the proposition under the assumption
 that $x$ and $y$ are in the upper half plane and outside of the lens connected to $a_j$. 
 Then by \cite[Formulas (5.159), (1.89) and (1.101)]{Deift-Kriecherbauer-McLaughlin-Venakides-Zhou99}, we have 
  \begin{equation} \label{eq:(21)_entry_edge_asy}
    \left[ Y_n^{-1}(y) Y_n(x) \right]_{2,1} = e^{n(g(x) + g(y) - \ell - i\Omega_j)} \left[ \hat{M}_p(\zeta(y))^{-1} \left( L(y)^{-1} R(y)^{-1} R(x) L(x) \right) \hat{M}_p(\zeta(x)) \right]_{2, 1},
  \end{equation}
  where $\Omega_j$ is the constant defined in \cite[Formula (1.20)]{Deift-Kriecherbauer-McLaughlin-Venakides-Zhou99}, the mapping $z \mapsto \zeta(z)$ is defined in \cite[Formula (5.95)]{Deift-Kriecherbauer-McLaughlin-Venakides-Zhou99}, and the matrix valued function $L$ is defined in \cite[Proof of Theorem 5.8]{Deift-Kriecherbauer-McLaughlin-Venakides-Zhou99}.
  \begin{equation}
    L(z) = M^{(\infty)}(z) \left[ \zeta(y)^{-\frac{\sigma_3}{4}} \frac{1}{\sqrt{2}} 
      \begin{pmatrix}
        1 & 1 \\
        -1 & 1 
      \end{pmatrix}
      e^{- \left( \frac{i \pi}{4} + \frac{ni}{2} \Omega_j \right) \sigma_3} \right]^{-1},
  \end{equation}
  and $\hat{M}_p$ is defined by the RHP \cite[Formulas (5.99)--(5.106)]{Deift-Kriecherbauer-McLaughlin-Venakides-Zhou99}. By the definition of the mapping $\zeta(z)$, we have
  \begin{equation} \label{eq:zeta(x)_approxbis}
    \zeta(x) = (2\pi \gamma)^{\gamma} u + O(n^{-1/(2\kappa + 1)}),\qquad \mbox{ as }n\to\infty
  \end{equation}
  uniform in $u$, and similarly for $\zeta(y)$ in terms of $v$. 
  By the argument in \cite[Proof of Theorem 5.8]{Deift-Kriecherbauer-McLaughlin-Venakides-Zhou99}, $L(z)$ is analytic 
  in $D_{\epsilon, a_j}$, and $\det(L(z)) = 1$. Then for $u, v$ in a compact subset 
  of $\mathbb C$, direct computation yields
  \begin{equation}
    L(x) = n^{\gamma/4} L_0 + n^{-3\gamma/4} x L_1 + O(n^{-5\gamma/4}),\qquad \mbox{ as }n\to\infty,
  \end{equation}
  where $L_0, L_1$ are constant matrices, and a similar result for $L(y)$. Then we have
 \begin{equation}
   L(y)^{-1} R(y)^{-1} R(x) L(x) = I + O(n^{-\gamma/2}),\qquad \mbox{ as }n\to\infty,
  \end{equation}
  uniform in $u, v$. We conclude that
  \begin{equation} \label{eq:Y^-1(y)Y(x)_by_RHP_edge}
    \left[ Y_n^{-1}(y) Y_n(x) \right]_{2,1} = e^{n(g(x) + g(y) - \ell - i\Omega_j)} \left[ \hat{M}_p \left( (2\pi \gamma)^{\gamma} v \right)^{-1} \hat{M}_p \left( (2\pi \gamma)^{\gamma} u \right) \right]_{2, 1} (1 + O(n^{-\gamma/2})),
  \end{equation}
  where the $O((n^{-\gamma/2}))$ term is uniform in $u, v$ as $n\to\infty$.

  By the Euler--Lagrange variational condition \eqref{eq:ELequation} and the Puiseux expansion \eqref{eq:Gmu0expansion2} of $g'(z)$ at $x^* = a_j$ (where $G_{\mu_0}(w) = g'(w)$), we have
  \begin{equation} \label{eq:approx_V-g-l-cbis}
   n\left( V(x) - g(x) - \frac{\ell + i\Omega_j}{2}\right) = \pi \gamma u^{\kappa + 1} + O(n^{-\gamma}),\qquad \mbox{ as }n\to\infty,
  \end{equation}
  uniformly in $u$, and a similar approximation holds for $V(y) - g(y) - (\ell + i\Omega_j)/2$. 
  By \eqref{eq:Y^-1(y)Y(x)_by_RHP_edge} and \eqref{eq:approx_V-g-l-cbis}, we have that the left-hand 
  side of \eqref{def:singlimit} that is expressed by \eqref{eq:KnMinYn} converges to a limit function 
  in $u, v$ as $n \to \infty$, as long as $u \neq v$. But the condition $u \neq v$ can be easily 
  removed by the analyticity of the functions. Thus we prove Proposition \ref{prop:prop17} for a 
  singular right-edge point in  case $x$ and $y$ are outside of the lens, and in the upper half-plane. 

 For general $x$ and $y$ we use \cite[Formula (5.91)]{Deift-Kriecherbauer-McLaughlin-Venakides-Zhou99} for the formula of $M_p(x), M_p(y)$ if 
 $x$ or $y$ lies in $\mathbb C_-$. We still obtain \eqref {eq:Y^-1(y)Y(x)_by_RHP_edge} with $\hat{M}_p(z)$
  replaced by the $\tilde{M}_p(z)$ which is defined by analytic continuation 
 of $\hat{M}_p(z)$ from the sector containing $i\mathbb R_+$.

 Thus Proposition \ref{prop:prop17} for a singular right-edge point is proved for $u, v$ in 
 any compact subset of $\mathbb C$.
\end{proof}

\subsubsection{Proof of Lemma \ref{lem:Y} for a singular right-edge point}
\begin{proof}
  As in the singular interior point case, it suffices to consider $z, w \in \mathbb C_+$, and 
  show that for any $z$ on $x^* + i\mathbb R$,
  \begin{equation} \label{eq:ineq_1st_column_edge}
    \left\lvert \left[ Y_n(z) \right]_{1, 1} \right\rvert \leq C n^{\gamma/2} \left\lvert e^{ng(z)} \right\rvert, \quad 
    	\left\lvert \left[ Y_n(z) \right]_{2, 1} \right\rvert \leq C n^{\gamma/2} \left\lvert e^{n(g(z) - \ell)} \right\rvert,
  \end{equation}
  and for any $w$ on $\Gamma$ or the interval $(x^* - R n^{-\gamma}, x^* + R n^{-\gamma})$, 
  \begin{equation} \label{eq:ineq_2nd_column_edge}
    \left\lvert \left[ Y_n(w) \right]_{1, 2} \right\rvert < C n^{\gamma/2} \left\lvert e^{-n(g(w) - \ell)} \right\rvert, \quad \left\lvert \left[ Y_n(w) \right]_{2, 2} \right\rvert < C n^{\gamma/2} \left\lvert e^{-ng(w)} \right\rvert,
  \end{equation}
  where $Y_n(w)$ means $Y_{n, +}(w)$ if $w \in (x^* - R n^{-\gamma}, x^* + R n^{-\gamma})$. Note that in this proof we identify $x^*$ with $a_j$.
  
  First, if $z \in a_j + i\mathbb R$ or $w \in \Gamma \cup (a_j - R n^{-\gamma}, a_j + R n^{-\gamma})$ is
   outside of the disc $D_{\epsilon, a_j}$, by \eqref{eq:approx_outer_region} and the same argument as in 
   the singular interior point case, both \eqref{eq:ineq_1st_column_edge} and 
   \eqref{eq:ineq_2nd_column_edge} hold.

  Next, if $z \in a_j + i\mathbb R$ is in the semi-disc $D_{\epsilon, a_j} \cap \mathbb C_+$, then since $z$ is outside of the lens connecting to $a_j$, by \cite[Formulas (5.159), (1.89) and (1.101)]{Deift-Kriecherbauer-McLaughlin-Venakides-Zhou99} we have
  \begin{equation} \label{eq:Y_n_near_z_hat_above_edge}
    Y_n(z) = e^{\frac{n \ell}{2} \sigma_3} R(z) L(z) \hat{M}_p(\zeta(z))e^{-\frac{in}{2} \Omega_j \sigma_3} e^{n(g(z) - \frac{\ell}{2})\sigma_3},
  \end{equation}
  where $\Omega_j$, $L$, $\hat{M}_p$ and $\zeta(z)$ are the same as in \eqref{eq:(21)_entry_edge_asy}. From the definition of $L$, we know that $L(z) = O(n^{\gamma/4})$ if $z \in D_{\epsilon, a_j}$. Also we have $\zeta(z) = O( n^{\gamma})$ for $z \in D_{\epsilon, a_j}$. Then from the RHP satisfied to $\hat{M}_p$, especially the boundary condition \cite[Formula (5.105)]{Deift-Kriecherbauer-McLaughlin-Venakides-Zhou99}, we have that $\hat{M}_p(\zeta(z)) = O(n^{\gamma/4})$. Thus we find that \eqref{eq:ineq_1st_column_edge} follows from \eqref{eq:Y_n_near_z_hat_above_edge}. Similarly, if $w \in \Gamma \cup (a_j - R n^{-\gamma}, a_j + R n^{-\gamma})$ is in the semi-disc $D_{\epsilon, a_j} \cup \mathbb C_+$, then we can deform the lenses such that $w$ is outside of the lens connecting to $a_j$ unless $\lvert w - a_j \rvert < 2R n^{-\gamma}$. In this case, \eqref{eq:Y_n_near_z_hat_above_edge} holds with $z$ replaced by $w$, and \eqref{eq:ineq_2nd_column_edge} follows.

  Additionally, if $w \in \Gamma \cup (a_j - R n^{-\gamma}, a_j + R n^{-\gamma})$, $w$ is in the semi-disc $D_{\epsilon, a_j} \cap \mathbb C_+$, and $\lvert w - a_j \rvert < 2R n^{-\gamma}$, then by comparing the jump conditions of the RHPs for $Y_n$, $M^{(1)}$, $M_p$ and $\hat{M}_p$, like we do in the proof of the singular interior point case in Section \ref{subsec:interior_RHP}, we have that \eqref{eq:Y_n_near_z_hat_above_edge} holds if $z$ is replaced by $w$ and $\hat{M}_p$ is replaced by the function $\tilde{M}_p(z)$, which is an analytic function on $\mathbb C$ that agrees with $\hat{M}_p(z)$ in a sector containing $i\mathbb R_+$, and is the analytic continuation of $\hat{M}_p(z)$ in that sector. Since the estimates of $L(z)$ still holds and $\hat{M}_p(\zeta(w)) = O(1)$ as $n\to\infty$ if $\lvert w - a_j \rvert < 2Rn^{-\gamma}$, we also find that \eqref{eq:ineq_2nd_column_edge} holds.
\end{proof}

\subsubsection*{Acknowledgements}

T.C. and A.B.J.K. are supported by the Belgian Interuniversity Attraction Pole P07/18.

T.C. was supported by the European Research Council
under the European Union’s Seventh Framework Programme (FP/2007/2013)/ ERC
Grant Agreement n. 307074.

A.B.J.K. is supported by long term structural funding-Methusalem grant of the Flemish Government, 
by KU Leuven Research Grant OT/12/073, and by FWO Flanders projects G.0934.13 and G.0864.16.

K.L. is supported by a grant from the Simons Foundation (\#357872, Karl Liechty).

D.W. is supported by the Singapore AcRF Tier 1 grant R-146-000-217-112.

\end{document}